\documentclass{article}

\title{Rich doctrines and Henkin's Theorem}
\author{Francesca Guffanti}
\date{}
\usepackage[utf8]{inputenc}
\usepackage{csquotes}
\usepackage[italian,british]{babel}
\usepackage[margin=89.5pt]{geometry}
\usepackage[pdfa]{hyperref}

\hypersetup{
    pdftitle={},
    pdfauthor={},
    pdfmenubar=false,
    pdffitwindow=true,
    pdfstartview=FitH,
    colorlinks=true,
    linkcolor=teal,
    citecolor=magenta,
    urlcolor=cyan
}
\usepackage{todonotes}
\usepackage{setspace}
\usepackage{mathtools}
\usepackage{amssymb}
\usepackage{bbold}
\usepackage{amsmath,amsthm, graphicx,xspace}
\usepackage{mathrsfs}
\usepackage{latexsym}
\usepackage{rotating}
\usepackage{multicol}
\usepackage{verbatim}
\usepackage{faktor}
\usepackage{epigraph}
\usepackage{quiver}
\usepackage{esvect}
\usepackage{cancel}
\usepackage{oubraces}
\usepackage{bbm}
\usepackage{cleveref}
\usepackage{enumitem}
\setlist[itemize]{parsep=0pt}
\setlist[enumerate]{parsep=0pt}
\usepackage{tikz-cd}
\usepackage{tikz}
\usetikzlibrary{arrows}

\numberwithin{equation}{section}
\theoremstyle{definition}
\newtheorem{theorem}{Theorem}[section]
\newtheorem{proposition}[theorem]{Proposition}
\newtheorem{lemma}[theorem]{Lemma}
\newtheorem{claim}[theorem]{Claim}
\newtheorem{definition}[theorem]{Definition}
\newtheorem{notation}[theorem]{Notation}
\newtheorem{example}[theorem]{Example}
\newtheorem{remark}[theorem]{Remark}
\mathcode`\:="603A % : = \colon
\mathchardef\colon="303A % re\def for use in :=

\def\pws{\mathop{\mathscr{P}\kern-1.1ex_{\ast}}}
\DeclareMathOperator{\OPRid}{id}
\newcommand{\id}[1]{\ensuremath{\OPRid_{#1}}}
\newcommand{\ct}[1]{\ensuremath{\mathbb{#1}}\xspace}
\newcommand{\CC}{\ct{C}}
\newcommand{\Pos}{\mathbf{Pos}}
\newcommand{\Cat}{\mathbf{Cat}}
\newcommand{\Hom}{\mathop{\mathrm{Hom}}}
\newcommand{\Dott}{\mathbf{Dct}}
\newcommand{\ple}[1]{\langle#1\rangle}
\newcommand{\op}{^{\mathrm{op}}}
\newcommand{\blank}{\mathrm{-}}
\newcommand{\pr}[1]{\ensuremath{\mathrm{pr}_{#1}}}
\newcommand{\tmn}{\mathbf{t}}
\newcommand{\Set}{\textnormal{Set}}
\newcommand{\des}{\mathscr{D}es}
\newcommand{\fbf}{\mathtt{LT}}
\newcommand{\ctx}{\mathrm{\mathbb{C}tx}}

\newcommand{\pointinproof}[1]{\newline{\bf#1:}}
\def\ECC(((#1))){\left[#1\right]}
\def\ecc#1{\ECC((#1))}
\newcommand{\DELTA}{\mathbf{\Delta}}
\begin{document}
%%% IMPOSTAZIONI TIPOGRAFICHE %%%
\newlength{\myindent} %Calcola e salva spazi di indentazione
\setlength{\myindent}{\parindent}
\parindent 0em %No indentazione pargrafi
\maketitle

\begin{abstract}
We find a possible interpretation of Henkin’s Theorem in the language of existential implicational doctrines. Under some smallness assumption, starting from an implicational existential doctrine, with non-trivial fibers, we construct a new doctrine which is rich---meaning that for every formula $\varphi(x)$ there is a constant $c$ such that $\exists x\varphi(x)$ has the same truth-value of $\varphi(c)$---and consistent. To obtain this result, we add a suitable amount of constants and axioms to the starting doctrine. We then show that a rich consistent doctrine admits an appropriate morphism towards the doctrine of subsets---a model. Henkin’s Theorem for doctrines follows from these two results, modeling our proof on the main lines of the original theorem.
\end{abstract}

\setcounter{tocdepth}{5}
\tableofcontents

\section{Introduction}
In a series of seminal papers \cite{adjfound,lawdiag,lawequality}, Lawvere introduced the concept of hyperdoctrine, aiming to interpret the syntax and the semantics of first-order theories in the same categorical framework. In this work we will actually deal with their further generalization---doctrines---adding the needed structure along the way. A doctrine is a functor $P:\CC\op\to\Pos$, from a category  $\CC$ with finite products into the category  $\Pos$ of partially ordered sets and monotone functions. A class of problems in this field involves determining which and how some classical results in logic can be interpreted with doctrines. Of course, one of the most important results is Gödel's Completeness Theorem for first-order logic. The theorem states that if a formula is valid in every model, then there is a formal proof of the formula. Modern proofs of the Completeness Theorem actually use Henkin's Theorem \cite{henkin}. The aim of this work is the analysis of this last theorem, formulated as follows:
\begin{quote}\emph{Every consistent theory has a model}.
\end{quote}
Some key points in the proof of the original theorem are adding a suitable amount of constants to the starting language, and then adding some axioms of the extended language to the starting theory.
In \cite{guff2} we began the investigation on how to interpret both of these instances, but in a finite way. We can then proceed with the interpretation of Henkin's proof by adding a suitable amount of new constant symbols to a starting doctrine $P$. To this aim, we first need to compute colimits of directed diagrams in the category of doctrines $\Dott$: to have an insight into this process, once we know how to add one constant symbol, we can iterate the construction to add a finite number of constant symbols. Then, taking the colimit over a convenient directed diagram $D:J\to\Dott$ in which every image $D(j)$ for $j\in J$ is a doctrine with a finite number of constants added, we can add an infinite amount of constants. This construction gives a morphism $P\to\underline{P}$ from the original doctrine into the colimit $\underline{P}$.
The next step is to add new axioms to the new doctrine $\underline{P}$. To do this, we work with implicational existential doctrines---i.e.\ doctrines in which we can interpret finite conjunctions, the implication and the existential quantifier. In this setting, for any formula $\varphi(x)$, we make true a formula of the kind $\exists x \varphi(x)\to\varphi(c)$ for some suitable constant $c$. Since there is an infinite number of axioms that we have to add, we use a similar technique to the one seen before: we define a directed diagram $\DELTA:I\to\Dott$ in which every image $\DELTA(i)$ for $i\in I$ is a doctrine with a finite number of axioms added, so the colimit adds all the needed axioms. This construction gives another morphism $\underline{P}\to\underrightarrow{P}$ into the colimit $\underrightarrow{P}$, and in particular a morphism $P\to\underrightarrow{P}$. These constructions from $P$ to $\underline{P}$ to $\underrightarrow{P}$ are done in \Cref{sect:dir_colim,sect:label}.

In \Cref{sub:weak_univ_prop} we show that the doctrine $\underrightarrow{P}$ is rich: for each formula $\varphi(x)$ there exists a constant $c$ such that $\varphi(c)$ and $\exists x\varphi(x)$ have the same truth-value.

When the starting doctrine $P$ is also bounded---i.e.\ a doctrine in which we can also interpret the false---, we find the properties for $P$ in order to have that the doctrine $\underrightarrow{P}$ is coherent, since we obviously do not want the doctrine $\underrightarrow{P}$ to be such that the true constant and the false collapse in the same formula. \Cref{sect:coher} collect all these results: initially \Cref{prop:coher_bool} establishes the consistency of $\underrightarrow{P}$ in the Boolean case, then \Cref{prop:coher} shows consistency in the implicational setting, following from a weak universal property of $\underrightarrow{P}$.

Finally, we prove in \Cref{prop:mod} that a bounded consistent implicational existential rich doctrine has a morphism to the doctrine of subsets, the ``standard" model. Applying this proposition to the rich doctrine $\underrightarrow{P}$, we obtain \Cref{thm:mod_ex}:
\begin{quote}\emph{Let $P$ be a bounded existential implicational doctrine, with non-trivial fibers and with a small base category. Then there exists a bounded existential implicational model of $P$ in the doctrine of subsets $\pws :\Set_{\ast}\op\to\Pos$.}\end{quote}

\nocite{borhandbk,elephant,CatWorMat}
%%%%%%%%%%%%%%%%%%%%%%%%%%%%%%%%%%%%%%%%%%%%%%%%%%%%%%
%%%%%%%%%%%%%%%%%%%%%%%%%%%%%%%%%%%%%%%%%%%%%%%%%%%%%%
\section{Preliminaries on doctrines}
In this section, we define the 2-category of doctrines and show some relevant examples. The definition in the following form can be found in \cite{trotta}, as the generalization of Lawvere's hyperdoctrine reduced to its basic structure. Then we will gradually add more structure in order to be able to interpret symbols of first-order logic---such as connectives and quantifiers---in the context of doctrines.
\begin{definition}
Let $\ct{C}$ be a category with finite products and let $\Pos$ be the category of partially-ordered sets and monotone functions. A \emph{doctrine} is a functor $P:\ct{C}\op\to\Pos$. The category $\CC$ is called \emph{base category of $P$}, each poset $P(X)$ for an object $X\in\CC$ is called \emph{fiber}, the function $P(f)$ for an arrow $f$ in $\CC$ is called \emph{reindexing}.
\end{definition}
\begin{notation}
In the paper, we write $\tmn_\CC$ for the terminal object of the category $\CC$. We omit the subscript when there's no confusion.
\end{notation}
\begin{example}\label{ex:doctr}We propose the following examples.
\begin{enumerate}[label=(\alph*)]
\item The functor $\mathscr{P}:\Set\op\to\Pos$, sending each set in the poset of its subsets, ordered by inclusion, and each function $f:A\to B$ to the inverse image $f^{-1}:\mathscr{P}(B)\to\mathscr{P}(A)$ is a doctrine.
\item For a given theory $\mathcal{T}$ in a one-sorted first-order language $\mathcal{L}$, define the category $\ctx_\mathcal{L}$ of contexts: an object is a finite list of distinct variables and an arrow between $\vec x=(x_1,\dots, x_n)$ and $\vec y=(y_1,\dots, y_m)$ is
\begin{equation*}(t_1(\vec x),\dots,t_m(\vec x)):(x_1,\dots, x_n) \to(y_1,\dots, y_m)\end{equation*}
an $m$-tuple of terms in the context $\vec x$. The functor $\fbf^\mathcal{L}_{\mathcal{T}}:\ctx_\mathcal{L}\op\to\Pos$ sends each list of variables to the poset reflection of well-formed formulae written with at most those free variables ordered by provable consequence in $\mathcal{T}$; moreover, the functor $\fbf^\mathcal{L}_{\mathcal{T}}$ sends an arrow $\vec{t}(\vec{x}):\vec{x}\to\vec{y}$ into the substitution $[\vec{t}(\vec{x})/\vec{y}]$.
\end{enumerate}
\end{example}
\begin{definition}
A \emph{doctrine morphism}---or 1-arrow---between $P:\ct{C}\op\to\Pos$ and $R:\ct{D}\op\to\Pos$ is a pair $(F,\mathfrak f)$ where $F:\ct{C}\to\ct{D}$ is a functor that preserves finite products and $\mathfrak{f}:P\xrightarrow{\cdot}R\circ F\op$ is a natural transformation. Sometimes a morphism between $P$ and $R$ will be called a model of $P$ in $R$.
A \emph{2-cell} between $(F,\mathfrak f)$ and $(G,\mathfrak g)$ from $P$ to $R$ is a natural transformation $\theta:F\xrightarrow{\cdot} G$ such that $\mathfrak{f}_A(\alpha)\leq R(\theta_A)(\mathfrak{g}_A(\alpha))$ for any object $A$ in $\ct{C}$ and $\alpha\in P(A)$. Doctrine, doctrine morphisms with 2-cells defined here form a $2$-category, that will be denoted $\Dott$.
\end{definition}
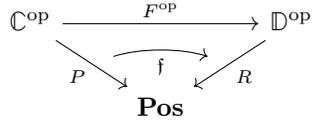
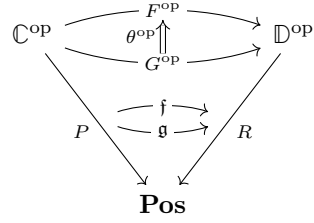
\[\begin{tikzcd}
	{\ct{C}\op} && {\ct{D}\op} && {\ct{C}\op} && {\ct{D}\op} \\
	& \Pos \\
	&&&&& \Pos
	\arrow["F\op", from=1-1, to=1-3]
	\arrow[""{name=0, anchor=center, inner sep=0}, "P"', from=1-1, to=2-2]
	\arrow[""{name=1, anchor=center, inner sep=0}, "R", from=1-3, to=2-2]
	\arrow[""{name=2, anchor=center, inner sep=0}, "F\op"{description}, curve={height=-12pt}, from=1-5, to=1-7]
	\arrow[""{name=3, anchor=center, inner sep=0}, "P"', from=1-5, to=3-6]
	\arrow[""{name=4, anchor=center, inner sep=0}, "R", from=1-7, to=3-6]
	\arrow[""{name=5, anchor=center, inner sep=0}, "G\op"{description}, curve={height=12pt}, from=1-5, to=1-7]
	\arrow["{\mathfrak{f}}"', curve={height=-6pt}, shorten <=8pt, shorten >=8pt, from=0, to=1]
	\arrow["{\mathfrak f}"{description}, curve={height=-6pt}, shorten <=8pt, shorten >=8pt, from=3, to=4]
	\arrow["{\mathfrak g}"{description}, curve={height=6pt}, shorten <=8pt, shorten >=8pt, from=3, to=4]
	\arrow["\theta\op", shorten <=3pt, shorten >=3pt, Rightarrow, from=5, to=2]
\end{tikzcd}\]
By definition of doctrine, the fibers are simply posets. However, we can define specific doctrines by imposing additional structure on these posets or by requiring the existence of adjoints to certain reindexing. To work in a setting that interprets the conjunction of formulae and the true constant, primary doctrines are necessary, which can be found in \cite{ElemQuotCompl,quotcomplfoun,unifying}.
\begin{definition}
A \emph{primary doctrine} $P:\ct{C}\op\to\Pos$ is a doctrine such that for each object $A$ in $\ct{C}$, the poset $P(A)$ has finite meets, and the related operations $\land:P\times P\xrightarrow{\cdot} P$ and $\top:\mathbf{1}\xrightarrow{\cdot} P$ yield natural transformations.
\end{definition}
\begin{example}
All doctrines in \Cref{ex:doctr} are primary doctrines:
\begin{enumerate}[label=(\alph*)]
\item In the doctrine $\mathscr{P}:\Set\op\to\Pos$, for any set $A$, intersection of two subsets is their meet, $A$ is the top element.
\item In the doctrine $\fbf^\mathcal{L}_{\mathcal{T}}:\ctx_\mathcal{L}\op\to\Pos$, for any list $\vec x$, the conjunction of two formulae is their binary meet, the true constant $\top$ is the top element.
\end{enumerate}
\end{example}
In order to interpret equality, the existence of left adjoints to reindexing of diagonal arrows is required, as originally observed by Lawvere in \cite{lawequality}. The definition of elementary doctrine we propose here can be found in Proposition 2.5 of \cite{EmPaRo}, and is equivalent to other definitions that are used in \cite{ElemQuotCompl,quotcomplfoun}:
\begin{definition}\label{def:elem}
A primary doctrine $P:\ct{C}\op\to\Pos$ is \emph{elementary} if for any object $A$ in $\ct{C}$ there exists an element $\delta_A\in P(A\times A)$ such that:
\begin{enumerate}
\item\label{i:1-equality} $\top_A\leq P(\Delta_A)(\delta_A)$;
\item\label{i:2-equality} $P(A)=\des_{\delta_A}:=\{\alpha\in P(A)\mid P(\pr1)(\alpha)\land\delta_A\leq P(\pr2)(\alpha)\}$;
\item\label{i:3-equality} $\delta_A\boxtimes\delta_B\leq\delta_{A\times B}$, where $\delta_A\boxtimes\delta_B=P(\ple{\pr1,\pr3})(\delta_A)\land P(\ple{\pr2,\pr4})(\delta_B)$.
\end{enumerate}
In \ref{i:2-equality}., $\pr1$ and $\pr2$ are the projections from $A\times A$ in $A$; in \ref{i:3-equality}., the projections are from $A\times B\times A\times B$. The element $\delta_A$ will be called \emph{fibered equality} on $A$.
\end{definition}
The following lemma will be useful to do some computation later. Its proof can be found in Proposition 2.5 of \cite{EmPaRo} (equation $(1)$).
\begin{lemma}\label{lemma:aeq}
Let $P:\CC\op\to\Pos$ be an elementary doctrine, let $C$ be an object in the base category and $\gamma\in P(C\times C)$. Then
\begin{equation*}P(\ple{\pr1,\pr1})(\gamma)\land\delta_C\leq\gamma.\end{equation*}
\end{lemma}
\begin{remark}\label{rmk:aeq}
The converse inequality of property \ref{i:3-equality}.\ in \Cref{def:elem} holds as well:
\begin{equation*}\delta_{A\times B}\leq P(\ple{\pr1,\pr3})(\delta_A)\land P(\ple{\pr2,\pr4})(\delta_B).\end{equation*}
To show $\delta_{A\times B}\leq P(\ple{\pr1,\pr3})(\delta_A)$, apply \Cref{lemma:aeq} to observe that
\begin{equation*}P(\ple{\pr1,\pr2,\pr1,\pr2})P(\ple{\pr1,\pr3})\delta_A\land\delta_{A\times B}=\delta_{A\times B}\leq P(\ple{\pr1,\pr3})\delta_A.\end{equation*}
Similarly $\delta_{A\times B}\leq P(\ple{\pr2,\pr4})(\delta_B)$.
\end{remark}
\begin{example}
\begin{enumerate}[label=(\alph*)]
\item In the doctrine $\mathscr{P}$, for any set $A$, the subset $\Delta_A=\{(a,a)\mid a \in A\}\subseteq A\times A$ is the fibered equality on $A$.
\item In the doctrine $\fbf^\mathcal{L}_{\mathcal{T}}$, if the language $\mathcal{L}$ has equality, then for any list $\vec x$, the formula $\big(x_1=x_1'\land\dots\land x_n=x_n'\big)$ in $\fbf^\mathcal{L}_\mathcal{T}(\vec x;\vec x')$ is the fibered equality on $\vec{x}$.
\end{enumerate}
\end{example}
We now interpret the existential and universal quantifier as respectively the left and the right adjoint to the reindexing along a product projection.
\begin{definition}[\cite{quotcomplfoun}]\label{def:exist}
A primary doctrine $P:\ct{C}\op\to\Pos$ is \emph{existential} if for any pair of objects $B, C$ of $\ct{C}$, the map $P(\pr1):P(C)\to P(C\times B)$ has a left adjoint
\begin{equation*}\exists^B_C:P(C\times B)\to P(C),\end{equation*}
satisfying:
\begin{itemize}
\item[-] Beck-Chevalley condition with respect to pullback diagrams of the form:
\[\begin{tikzcd}[ampersand replacement=\&]
	{C\times B} \& C \\
	{C'\times B} \& {C'}
	\arrow["{f\times \id{B}}"', from=1-1, to=2-1]
	\arrow["f", from=1-2, to=2-2]
	\arrow["\pr1"', from=2-1, to=2-2]
	\arrow["\lrcorner"{anchor=center, pos=0.125}, draw=none, from=1-1, to=2-2]
	\arrow["\pr1", from=1-1, to=1-2]
\end{tikzcd}\]
that is, $\exists^B_CP(f\times\id{B})=P(f)\exists^B_{C'}$.
\item[-] Frobenius reciprocity, that is, for any $\alpha\in P(C\times B)$ and $\beta\in P(C)$ the equality
\[\exists^B_C(\alpha\land P(\pr1)(\beta))=\exists^B_C(\alpha)\land \beta\]
holds.
\end{itemize}
\end{definition}
\begin{definition}[\cite{pasq}]
A doctrine $P:\ct{C}\op\to\Pos$ is \emph{universal} if for any pair of objects $B, C$ of $\ct{C}$, the map $P(\pr1):P(C)\to P(C\times B)$ has a right adjoint $\forall^B_C:P(C\times B)\to P(C)$, satisfying Beck-Chevalley condition with respect to pullback diagrams of the same form as \Cref{def:exist}, that is $P(f)\forall^B_{C'}=\forall^B_CP(f\times\id{B})$.
\end{definition}
\begin{definition}
A doctrine $P:\ct{C}\op\to\Pos$:
\begin{itemize}
\item is \emph{implicational} if for any object $A$, the poset $P(A)$ is cartesian closed, and the related operations $\land:P\times P\xrightarrow{\cdot} P$, $\top:\mathbf{1}\xrightarrow{\cdot} P$, $\to:P\op\times P\xrightarrow{\cdot} P$ yield natural transformations---in particular it is a primary doctrine (\cite{quotcomplfoun});
\item \emph{has bottom element} if for any object $A$, the poset $P(A)$ has a bottom element, and the related operation, $\bot:\mathbf{1}\xrightarrow{\cdot} P$ yields a natural transformation;
\item \emph{is bounded} if for any object $A$, the poset $P(A)$ has a top and a bottom element, and the related operation, $\top:\mathbf{1}\xrightarrow{\cdot} P$ and $\bot:\mathbf{1}\xrightarrow{\cdot} P$ yield natural transformations;
\item \emph{has finite joins} if for any object $A$, the poset $P(A)$ has finite joins, and the related operations $\lor:P\times P\xrightarrow{\cdot} P$, $\bot:\mathbf{1}\xrightarrow{\cdot} P$ yield natural transformations;
\item is \emph{Heyting} if for any object $A$, the poset $P(A)$ is an Heyting algebra, and the related operations $\land:P\times P\xrightarrow{\cdot} P$, $\top:\mathbf{1}\xrightarrow{\cdot} P$, $\to:P\op\times P\xrightarrow{\cdot} P$, $\lor:P\times P\xrightarrow{\cdot} P$, $\bot:\mathbf{1}\xrightarrow{\cdot} P$ yield natural transformations;
\item is \emph{Boolean} if it is Heyting and the operation $\lnot(\blank):=(\blank)\to\bot:P\op\xrightarrow{\cdot}P$ is an isomorphism. 
\end{itemize}
\end{definition}
\begin{definition}
Any morphism $(F,\mathfrak f):P\to R$ from $P:\ct{C}\op\to\Pos$ to $R:\ct{D}\op\to\Pos$ is called respectively \emph{primary, elementary, existential, universal, implicational, bounded, Heyting, Boolean} if both $P$ and $R$ are, and $\mathfrak{f}$ preserves the said structure. 
\end{definition}
To conclude our introduction on doctrines, let us briefly recall from \cite{guff2} the construction that adds constant and that adds an axiom to a doctrine, since we will use them as the finite steps of the directed diagrams in \Cref{sect:dir_colim,sect:dir_colim2}.

We start with adding a constant: let $P:\CC\op\to\Pos$ be a doctrine and let $X$ be a fixed object in the base category. Consider the reader comonad $X\times\blank$ on $\CC$, and let $\CC_X$ be the Kleisli category of the comonad, and use the following presentation of the category $\CC_X$: it has the same objects as $\CC$; an arrow in $\CC_X$ from $A$ to $B$---we will write $A\rightsquigarrow B$---is actually $\CC$-arrow $X\times A\to B$. Composition between $f:A\rightsquigarrow B$ and $g:B\rightsquigarrow C$ is the arrow $g\ple{ \pr1,f}:A\rightsquigarrow C$, the identity $A\rightsquigarrow A$ on $A$ is given by the projection over the second component $X\times A\to A$. In particular, the new constant of sort $X$ added to the base category is $\id{X}:\tmn\rightsquigarrow X$. Define the doctrine $P_X:{\CC_X}\op\to\Pos$ as follows:
\begin{equation}\label{def:const}
\text{For}
\begin{tikzcd}
B\\
A\arrow[u,"f", squiggly]
\end{tikzcd}
\text{the reindexing is}
\begin{tikzcd}
P(X\times B)\arrow[d,"{P(\ple{ \pr1,f })}"]\\
P(X\times A)
\end{tikzcd}.
\end{equation}
The order in the fibers of the doctrine $P_X$ is computed as in $P$.
This new doctrines also comes with a 1-arrow $(F_X,\mathfrak{f}_X):P\to P_{X}$, where $F_X$ is the cofree functor sending $g:A\to B$ to $g\pr2:A\rightsquigarrow B$, while each component of the natural transformation $\mathfrak{f}_X$ is given by reindexing along the second projection: $(\mathfrak{f}_X)_A:=P(\pr2):P(A)\to P(X\times A)$.
The construction satisfies the following universal property---see Theorem 6.2 and Corollary 6.7 in \cite{guff2}.
\begin{theorem}\label{coroll:const}
Let $P:\CC\op\to\Pos$ be a doctrine. Given an object $X$ in the base category, the 1-arrow $(F_X,\mathfrak{f}_X):P\to P_X$ and the $\CC_X$-arrow $\id{X}:\tmn_{\CC_X}\rightsquigarrow X$ are universal, i.e.\ for any 1-arrow $(G,\mathfrak{g}):P\to R$, where $R:\ct{D}\op\to\Pos$ is a doctrine, and any $\ct{D}$-arrow $c:\tmn_\ct{D}\to G(X)$ there exists a unique up to a unique natural isomorphism 1-arrow $(G',\mathfrak{g}'):P_X\to R$ such that $(G',\mathfrak{g}')\circ(F_X,\mathfrak{f}_X)=(G,\mathfrak{g})$ and $G'(\id{X})=c$.
\end{theorem}
We then proceed by recalling how to add an axiom, this time to a primary doctrine: let $P:\CC\op\to\Pos$ be a primary doctrine and let $\varphi\in P(\tmn)$ be a fixed formula in the fiber over the terminal object $\tmn$ of the base category. Define the doctrine $P_\varphi:{\CC}\op\to\Pos$ as follows:
\begin{equation}\label{def:ax}
\text{For}
\begin{tikzcd}
B\\
A\arrow[u,"f"]
\end{tikzcd}
\text{the reindexing is}
\begin{tikzcd}
P(B)_{\downarrow P(!_B)\varphi}\arrow[d,"{P(f )}"]\\
P(A)_{\downarrow P(!_A)\varphi}
\end{tikzcd},
\end{equation}
where $P(!_A)$ is the reindexing along the unique arrow $!_A:A\to \tmn$ and $P(A)_{\downarrow P(!_A)\varphi}=\{\alpha\in P(A)\mid \alpha\leq P(!_A)\varphi\}$. This is a primary doctrine, where in the fibers of $P_\varphi$ order and binary conjunctions are computed as in $P$, and the top element is given by $P(!)\varphi$.
This new doctrines also comes with a primary 1-arrow $(\id{\CC},\mathfrak{f}_\varphi):P\to P_{\varphi}$, where for any object $A$ the corresponding component of the natural transformation $(\mathfrak{f}_\varphi)_A:P(A)\to P_\varphi(A)$ maps an element $\alpha\in P(A)$ to $P(!_A)\varphi\land \alpha\in P_\varphi(A)$. Also this construction satisfies a universal property---see Theorem 6.2 and Corollary 6.5 in \cite{guff2}.
\begin{theorem}\label{coroll:ax}
Let $P:\CC\op\to\Pos$ be a primary doctrine. Given an element $\varphi\in P(\tmn)$, the 1-arrow $(\id{\CC},\mathfrak{f}_\varphi):P\to P_\varphi$ is such that $\top\leq (\mathfrak{f}_\varphi)_\tmn(\varphi)$ in $P_\varphi(\tmn)$, and it is universal with respect to this property, i.e.\ for any primary 1-arrow $(G,\mathfrak{g}):P\to R$, where $R:\ct{D}\op\to\Pos$ is a primary doctrine, such that $\top\leq \mathfrak{g}_\tmn(\varphi)$ in $R(\tmn_\ct{D})$ there exists a unique up to a unique natural isomorphism primary 1-arrow $(G',\mathfrak{g}'):P_{\varphi}\to R$ such that $(G',\mathfrak{g}')\circ(\id{\CC},\mathfrak{f}_\varphi)=(G,\mathfrak{g})$.
\end{theorem}
Both constructions and their universal properties respect additional structures that the starting doctrine may enjoy. The proof of the following are instances of Propositions from 5.3 to 5.12 and of Theorem 6.3 in \cite{guff2}.
\begin{proposition}\label{prop:const}
Let $P:\CC\op\to \Pos$ be a doctrine, and let $P_X$, $(F_X,\mathfrak{f}_X):P\to P_X$ be defined as in \eqref{def:const}. Then if $P$ is primary (resp.\ elementary, existential, universal, implicational, bounded, Boolean), then $P_X$, and $(F_X,\mathfrak{f}_X)$ are primary (resp.\ elementary, existential, universal, implicational, bounded, Boolean).
Moreover let $R$, $(G,\mathfrak{g}):P\to R$ be the doctrines and a morphism with the same assumption of \Cref{coroll:const} above, and $(G',\mathfrak{g}'):P_X\to R$ be the morphism defined by the Theorem. Then if $P$, $R$ and $(G,\mathfrak{g})$ are primary (resp.\ elementary, existential, universal, implicational, bounded, Boolean), then also $(G',\mathfrak{g}')$ is primary (resp.\ elementary, existential, universal, implicational, bounded, Boolean).
\end{proposition}
\begin{proposition}\label{prop:ax}
Let $P:\CC\op\to \Pos$ be a primary doctrine, and let $P_\varphi$, $(\id{\CC},\mathfrak{f}_\varphi):P\to P_\varphi$ be defined as in \eqref{def:ax}. Then if $P$ is elementary (resp.\ existential, universal, implicational, bounded, Boolean), then $P_\varphi$, and $(\id{\CC},\mathfrak{f}_\varphi)$ are elementary (resp.\ existential, universal, implicational, bounded, Boolean).
Moreover let $R$, $(G,\mathfrak{g}):P\to R$ be the doctrines and a morphism with the same assumption of \Cref{coroll:ax} above, and $(G,\mathfrak{g}'):P_\varphi\to R$ be the morphism defined by the same theorem. Then if $P$, $R$ and $(G,\mathfrak{g})$ are elementary (resp.\ existential, universal, implicational, bounded, Boolean), then also $(G,\mathfrak{g}')$ is elementary (resp.\ existential, universal, implicational, bounded, Boolean).\end{proposition}
%

%%%%%%%%%%%%%%%%%%%%%%%%%%%%%%%%%%%%%%%%%%%%%%%%%%%%%%%%%%
%%%%%%%%%%%%%%%%%%%%%%%%%%%%%%%%%%%%%%%%%%%%%%%%%%%%%%%%%%
\section{Build a rich doctrine}\label{sect:dir_col}

The ultimate goal of the paper is to provide some conditions on a doctrine $P:\CC\op\to\Pos$ in order to admit a doctrine morphism in the doctrine $\pws:\Set_{\ast}\op\to\Pos$. The latter doctrine is a variation of the doctrine of subsets $\mathscr{P}$, where the only difference is that we remove the empty set from the base category. This corresponds to the fact that we only want to consider non-empty models.

We do not assume that the doctrine $P$ has any specific structure at this time, but we add the necessary properties as we proceed through the section. Every structural property that we add to $P$, is then asked to be preserved by the doctrine morphism involved. Similarly to what happens in the proof of Henkin's Theorem \cite{henkin}, we split in two the problem of finding a model: Henkin's idea is to at first extend the language and the theory in order to obtain a rich theory, and then to show that one can easily define a model of a rich theory, defining a suitable interpretation on the set of closed terms. We do a similar thing, providing at first in this section some doctrine morphism $P\to \underrightarrow{P}$, where $\underrightarrow{P}$ is rich in the sense of \Cref{def_rich} below, and then we ``easily" define a model of a rich doctrine in \Cref{sect:mod_rch}.
\begin{definition}\label{def_rich}
Let $R:\mathbb{D}^{\op}\to\Pos$ be an existential doctrine. The doctrine $R$ is \emph{rich} if for all $A\in\mathrm{ob}\mathbb{D}$ and for all $\sigma\in R(A)$ there exists a $\mathbb{D}$-arrow $d:\tmn\to A$ such that
\begin{equation}\label{eq:rich}\exists^A_\tmn\sigma\leq R(d)\sigma.\end{equation}
\end{definition}
\begin{remark}
For every object $A$ in the base category of a rich doctrine, there exists an arrow from the terminal object to $A$.
\end{remark}
\begin{remark}
Observe that the condition \eqref{eq:rich} in \Cref{def_rich} is actually an equality. Indeed, to prove the converse direction, it is enough to apply the reindexing $R(d)$ to the inequality $\sigma\leq R(!_A)\exists^A_\tmn\sigma$, that holds by adjunction.
\end{remark}
\begin{example}
The subsets doctrine $\mathscr{P}:\Set\op\to\Pos$ is not rich, since there exists no arrow $\tmn\to\emptyset$.
However, we can remove the empty set from the base category and consider the doctrine $\pws :\Set_{\ast}\op\to\Pos$, which is rich. Another example of a rich doctrine will be provided in \Cref{ex:real}.
\end{example}
\begin{remark}
Recall from \cite{triposes} that an existential doctrine $R:\mathbb{D}\op\to\Pos$ is \emph{equipped with $\varepsilon$-operator} if for every $A,B\in\mathrm{ob}\mathbb{D}$ and every $\alpha\in R(B\times A)$ there is an arrow $\varepsilon_\alpha: B\to A$ such that
\[\exists^A_B\alpha=R(\ple{\id{B},\varepsilon_\alpha})\alpha\]
in $R(A)$\footnote{Actually, in \cite{triposes}, the definition of a doctrine equipped with $\varepsilon$-operator is given for existential \emph{elementary} doctrine. Since the equality is not involved in the definition, we can provide the definition of $\varepsilon$-operator for existential doctrines.}. \Cref{def_rich} is similar, except for the fact that we ask the condition above to hold not for any object $B$ in the base category but only for $B=\tmn$.
\end{remark}
In order to define a rich doctrine $\underrightarrow{P}$, we need a middle step $P\to\underline{P}\to\underrightarrow{P}$, where in the doctrine $\underline{P}$ we add a suitable amount of constant to the doctrine $P$. To achieve this result, colimits over directed preorder in the category of doctrines are needed; moreover, if every doctrine in the image of the diagram has a property (such as being primary, implicational, elementary, existential, \dots), preserved by the morphisms in the diagram, then the colimit has the same property.
\begin{proposition}\label{prop:dir_colim}
The category $\Dott$ has colimits over directed preorders.
\end{proposition}
\begin{proposition}\label{prop:addit_struct}
Let $I$ be a directed preorder, let $D:I\to\Dott$ be a diagram, $D(i\leq j)=[(F_{ij},\mathfrak{f}_{ij}):P_i\to P_j]$ for any $i,j\in I$, and let $(P_\bullet,\{(F_i,\mathfrak f _i)\}_{i\in I})$ be the colimit of $D$. Suppose that for every $i,j\in I$, the doctrine $P_i$ and the morphism $(F_{ij},\mathfrak{f}_{ij})$ are primary. Then the doctrine $P_\bullet$ is a primary doctrine, and for every $i\in I$ the morphism $(F_i,\mathfrak f _i)$ is primary.
Moreover, if in a cocone $(R,\{(G_i,\mathfrak g _i)\}_{i\in I})$, $R$ and $(G_i,\mathfrak g _i)$ are primary, then the unique arrow $(G,\mathfrak g):P_\bullet\to R$ defined by the universal property of the colimit is primary.
The same statement holds if we write respectively bounded, with binary joins, implicational, elementary, existential, universal, Heyting, Boolean instead of primary.
\end{proposition}
The proofs of both propositions can be found in \Cref{sub:dir_colim}.

%%%%%%%%%%%%%%%%%%%%%%%%%%%%%%%%%%%%%%%%%%%%%%%%%%%%%%%%%%
\subsection{The construction of the directed colimit $\underline{P}$}\label{sect:dir_colim}
\begin{center}
\fbox{\parbox{.5\textwidth}{\raggedright From now on, $P:\CC\op\to\Pos$ is a fixed doctrine, unless otherwise specified.}}
\end{center}
%\pointinproof{The directed preorder $J$}
{\bf The directed preorder $J$:}
For a fixed cardinal $\Lambda\neq 0$, define $J$ the set of finite lists with different entries with values in $\{(X,\lambda)\}_{X\in\text{ob}\CC,\lambda\in\Lambda}$. We ask the empty list to belong to $J$. Define a preorder in $J$ as follows:
\begin{equation*}\big((X_1,x_1),\dots,(X_n,x_n)\big)\leq\big((Y_1,y_1),\dots,(Y_m,y_m)\big)\end{equation*}
if and only if
\begin{equation*}\big\{(X_1,x_1),\dots,(X_n,x_n)\big\}\subseteq\big\{(Y_1,y_1),\dots,(Y_m,y_m)\big\}.\end{equation*}
Whenever we have $\bar{X}\leq\bar{Y}$ in $J$, there exists a unique function $\tau:\{1,\dots,n\}\to\{1,\dots,m\}$  induced by the inclusion such that $(X_i,x_i)=(Y_{\tau(i)},y_{\tau(i)})$ for all $i=1,\dots,n$.

Observe that $J$ is a directed preorder: given $\bar{X},\bar{Y}\in J$, define the list $\bar{Z}$ to be the juxtaposition of $\bar{X}$ with all the entries of $\bar{Y}$ that do not appear in $\bar{X}$; then $\bar{X}\leq\bar{Z}\geq\bar{Y}$.

On a side note, we point out that we will not study the case $J=\emptyset$, since this would imply the category $\CC$ to have no object.
\pointinproof{The diagram $D:J\to\Dott$}
Define the following diagram on $J$:
\[\begin{tikzcd}
	J &&& \Dott \\
	\emptyset &&& {P:\CC\op\to\Pos} \\
	{\bar{X}=\big((X_1,x_1),\dots,(X_n,x_n)\big)} &&& {P_{\Pi_{a=1}^nX_a}:\CC_{\Pi_{a=1}^nX_a}\op\to\Pos} \\
	{\bar{Y}=\big((Y_1,y_1),\dots,(Y_m,y_m)\big)} &&& {P_{\Pi_{b=1}^mY_b}:\CC_{\Pi_{b=1}^mY_b}\op\to\Pos}
	\arrow[from=1-1, to=1-4, "D"]
	\arrow[maps to, from=2-1, to=2-4]
	\arrow["\leq"{marking}, draw=none,, from=2-1, to=3-1]
	\arrow["\leq"{marking}, draw=none,, from=3-1, to=4-1]
	\arrow[maps to, from=3-1, to=3-4]
	\arrow[maps to, from=4-1, to=4-4]
	\arrow["{(F_{\bar{X}\bar{Y}},\mathfrak{f}_{\bar{X}\bar{Y}})}", from=3-4, to=4-4]
	\arrow["{(F_{\bar{X}},\mathfrak{f}_{\bar{X}})}", from=2-4, to=3-4]
\end{tikzcd}\]
where:
\begin{itemize}
\item $\CC_{\Pi_{a=1}^nX_a}$ has the same objects of $\CC$ and an arrow $A\rightsquigarrow B$ from $A$ to $B$ is actually a \CC-arrow $\prod_{a=1}^nX_a\times A\to B$;
\item $P_{\Pi_{a=1}^nX_a}(A)=P(\prod_{a=1}^nX_a\times A)$, with definition on arrows as in \eqref{def:const};
\item $F_{\bar{X}}\big(f:A\to B\big)=\big(f\circ\pr{A}:A\rightsquigarrow B\big)$ is the composition $\prod_{a=1}^nX_a\times A\to A\to B$;
\item $(\mathfrak{f}_{\bar{X}})_A:P(A)\to P_{\Pi_{a=1}^nX_a}(A)=P(\prod_{a=1}^nX_a\times A)$ is the reindexing along the projection over $A$;
\item $F_{\bar{X}\bar{Y}}\big(f:A\rightsquigarrow B\big)=\big(f\circ(\ple{ \pr{\tau(1)},\dots,\pr{\tau(n)}}\times \id{A}):A\rightsquigarrow B\big)$ is the following composition $\prod_{b=1}^mY_b\times A\to \prod_{a=1}^nX_a\times A\to B$. Here $\ple{ \pr{\tau(1)},\dots,\pr{\tau(n)}}$ is the projection on the corresponding components from $\prod_{b=1}^mY_b$ to $\prod_{a=1}^nX_a$, since $X_i$ appears as the $\tau(i)$-th component of $\bar{Y}$;
\item $(\mathfrak{f}_{\bar{X}\bar{Y}})_A:P(\prod_{a=1}^nX_a\times A)\to P(\prod_{b=1}^mY_b\times A)$ is defined as the reindexing along the map $\ple{ \pr{\tau(1)},\dots,\pr{\tau(n)}}\times \id{A}$.
\end{itemize}
For any $\emptyset\leq\bar{X}\leq\bar{Y}$ observe that the composition $(F_{\bar{X}\bar{Y}},\mathfrak{f}_{\bar{X}\bar{Y}})(F_{\bar{X}},\mathfrak{f}_{\bar{X}})$ is $(F_{\bar{Y}},\mathfrak{f}_{\bar{Y}})$. Indeed, between the base categories we have:
\begin{equation*}F_{\bar{X}}:\bigg(f:A\to B\bigg)\mapsto\bigg(f\pr{A}:\prod_{a=1}^nX_a\times A\to B\bigg)\end{equation*}
and then
\begin{equation*}F_{\bar X\bar Y}:f\pr{A}\mapsto\bigg(f\pr{A}\circ(\ple{ \pr{\tau(1)},\dots,\pr{\tau(n)}}\times \id{A}):\prod_{b=1}^mY_b\times A\to B\bigg)=\bigg(f\pr{A}:\prod_{b=1}^mY_b\times A\to B\bigg),\end{equation*}
so $F_{\bar X\bar Y}F_{\bar X}=F_{\bar Y}$.
Moreover $(\mathfrak{f}_{\bar{X}\bar{Y}})_A(\mathfrak{f}_{\bar{X}})_A=P(\ple{ \pr{\tau(1)},\dots,\pr{\tau(n)}}\times \id{A})P(\pr{A})=P(\pr{A})=(\mathfrak{f}_{\bar{Y}})_A$. Observe that both equalities follow from the fact that $\pr{A}\circ(\ple{ \pr{\tau(1)},\dots,\pr{\tau(n)}}\times \id{A})=\pr{A}$. 

Similarly, for any $\bar{X}\leq\bar{Y}\leq\bar{Z}$ with induced functions respectively  $\tau:\{1,\dots,n\}\to\{1,\dots,m\}$ and $\tau':\{1,\dots,m\}\to\{1,\dots,s\}$, we observe that the composition $(F_{\bar{Y}\bar{Z}},\mathfrak{f}_{\bar{Y}\bar{Z}})(F_{\bar{X}\bar{Y}},\mathfrak{f}_{\bar{X}\bar{Y}})$ is $(F_{\bar{X}\bar{Z}},\mathfrak{f}_{\bar{X}\bar{Z}})$ using the fact that
\begin{equation*}(\ple{ \pr{\tau(1)},\dots,\pr{\tau(n)}}\times \id{A})\circ(\ple{ \pr{\tau'(1)},\dots,\pr{\tau'(m)}}\times \id{A})=(\ple{ \pr{\tau'\tau(1)},\dots,\pr{\tau'\tau(n)}}\times \id{A}).\end{equation*}
So $D:J\to\Dott$ is indeed a diagram. 
\pointinproof{The colimit of $D$}
Let $\underline{P}:\underline{\CC}\op\to\Pos$ be the colimit of $D$ in $\Dott$---we refer to the proof of \Cref{prop:dir_colim_app} in \Cref{sub:dir_colim} to know the details of how it is computed. Objects in the base category are the same as \CC, since $F_{\bar{X}\bar{Y}}$'s act like the identity on objects. An arrow $\ecc{(f,\bar{X})}$ in $\Hom_{\underline\CC}(A,B)$---we write $\ecc{(f,\bar{X})}:A\dashrightarrow B$---is the equivalence class of an arrow $f:\prod_{a=1}^nX_a\times A \to B$ for some fixed $\bar{X}=\big((X_1,x_1),\dots,(X_n,x_n)\big)\in J$. One has $\ecc{(f,\bar{X})}=\ecc{(f',\bar{Y})}$, for some $f':\prod_{b=1}^mY_b\times A \to B$ with $\bar{Y}=\big((Y_1,y_1),\dots,(Y_m,y_m)\big)\in J$ if and only if there exists $\bar{Z}\in J$ such that $\bar{X}\leq\bar{Z}\geq\bar{Y}$ making the following diagram commute:
\[\begin{tikzcd}
	&& {\prod_{a=1}^nX_a\times A} \\
	{\prod_{c=1}^sZ_c\times A} && 
 && B \\
	&& {\prod_{b=1}^mY_b\times A}
	\arrow["{\ple{ \pr{\tau(1)},\dots,\pr{\tau(n)}}\times \id{A}}", from=2-1, to=1-3]
	\arrow["f", from=1-3, to=2-5]
	\arrow["{\ple{ \pr{\tau'(1)},\dots,\pr{\tau'(m)}}\times \id{A}}"', from=2-1, to=3-3]
	\arrow["{f'}"', from=3-3, to=2-5]
\end{tikzcd}.\]
Here $\tau$ and $\tau'$ are induced by $\bar{X}\leq\bar{Z}$ and $\bar{Y}\leq\bar{Z}$ in $J$ respectively.

For any object $A$, we have $\underline{P}(A)\ni\ecc{(\varphi,\bar{X})}$ for some $\varphi\in P(\prod_{a=1}^nX_a\times A)$. Here $\ecc{(\varphi,\bar{X})}=\ecc{(\varphi',\bar{Y})}$, where $\varphi'\in P(\prod_{b=1}^mY_b\times A)$ if and only if there exists $\bar{Z}\in J$ such that $\bar{X}\leq\bar{Z}\geq\bar{Y}$ with induced function $\tau$ and $\tau'$ such that $P(\ple{ \pr{\tau(1)},\dots,\pr{\tau(n)}}\times \id{A})\varphi=P(\ple{ \pr{\tau'(1)},\dots,\pr{\tau'(m)}}\times \id{A})\varphi'$ in $P(\prod_{c=1}^sZ_c\times A)$. The reindexing is defined in a common list of $J$: if $\ecc{(f,\bar{X})}:A\dashrightarrow B$ and $\ecc{(\psi,\bar{Y})}\in\underline{P}(B)$, take $\bar{X}\leq\bar{Z}\geq\bar{Y}$; then
\begin{align*}&\underline{P}\big(\ecc{(f,\bar{X})}\big)\ecc{(\psi,\bar{Y})}\\
&=\underline{P}\big(\ecc{(f\circ(\ple{\pr{\tau(1)},\dots,\pr{\tau(n)}}\times\id{A}),\bar{Z})}\big)\ecc{(P(\ple{\pr{\tau'(1)},\dots,\pr{\tau'(m)}}\times\id{B})\psi,\bar{Z})}\\
&=\ecc{(P(\ple{\pr{1},\dots,\pr{s},f\circ(\ple{\pr{\tau(1)},\dots,\pr{\tau(n)}}\times\id{A})})P(\ple{\pr{\tau'(1)},\dots,\pr{\tau'(m)}}\times\id{B})\psi,\bar{Z})}\\
&=\ecc{(P(\ple{\pr{\tau'(1)},\dots,\pr{\tau'(m)},f\circ(\ple{\pr{\tau(1)},\dots,\pr{\tau(n)}}\times\id{A})})\psi,\bar{Z})}.\end{align*}
\[\begin{tikzcd}
	{\prod Z_c\times A} &&& {\prod X_a\times A} & B \\
	{\prod Z_c\times B} &&& {\prod Y_b\times B}
	\arrow["{\ple{ \pr{\tau(1)},\dots,\pr{\tau(n)}}\times \id{A}}", from=1-1, to=1-4]
	\arrow["f", from=1-4, to=1-5]
	\arrow["{\ple{ \pr{\tau'(1)},\dots,\pr{\tau'(m)}}\times \id{B}}"'yshift=-1pt, from=2-1, to=2-4]
	\arrow["{\ple{\pr{1},\dots,\pr{s},f\circ(\ple{\pr{\tau(1)},\dots,\pr{\tau(n)}}\times\id{A})}}"', from=1-1, to=2-1]
\end{tikzcd}\]
\begin{remark}
Call $(\underline{F},\underline{\mathfrak f}):P\to\underline{P}$ the map in the colimit starting from $D(\emptyset)$: the functor $\underline{F}$ maps a $\CC$-arrow $f:A\to B$ into $[f,\emptyset]:A\dashrightarrow B$, a component of the natural transformation $\underline{\mathfrak{f}}_A$ sends $\alpha\in P(A)$ into $[\alpha,\emptyset]\in \underline{P}(A)$. Moreover, by the universal property stated in \Cref{coroll:const}, any morphism $D(\bar{X})\to\underline{P}$ is uniquely determined by the morphism $(\underline{F},\underline{\mathfrak f}):P\to\underline{P}$ and a choice of a constant of sort $\prod_{a=1}^nX_a$ in the base category of $\underline{P}$---that is $\ecc{(\id{\Pi X_a},\bar{X})}:\tmn\dashrightarrow \prod_{a=1}^nX_a$. For this reason and by definition of colimit, any doctrine morphism $(G,\mathfrak g):\underline{P}\to R$ is uniquely determined by its precompositions with $(\underline{F},\underline{\mathfrak f})$ and a choice of a constant of sort $GX$ in the base category of $R$ for any pair $(X,\lambda)$ for every object $X$ in $\CC$ and any $\lambda\in\Lambda$. 
\end{remark}
\begin{remark}
Note that the same construction can be made if we change the cardinals over the objects: take for any object $X$ a cardinal $\Lambda_X$, and call $J$ the set of finite lists with values in $\{(X,\lambda)\}_{X\in\text{ob}\CC,\lambda\in\Lambda_X}$. In this case, we just ask for the existence of at least one cardinal $\Lambda_X$ different from $0$.
\end{remark}
%

%%%%%%%%%%%%%%%%%%%%%%%%%%%%%%%%%%%%%%%%%%%%%%%%%%%%%%%%%%
\subsection{The construction of the directed colimit $\protect\underrightarrow{P}$}\label{sect:dir_colim2}\label{sect:label}
\begin{center}
\fbox{\parbox{.5\textwidth}{\raggedright From now on, $P:\CC\op\to\Pos$ is a fixed implicational existential doctrine, with a small base category, unless otherwise specified.}}
\end{center} 
%\pointinproof{Listing formulae and labeling new constants}
{\bf Listing formulae and labeling new constants:}
Let $\Lambda=\text{card}\big(\bigsqcup_{X\in\text{ob}\CC}P(X)\big)$ and build the colimit doctrine $\underline{P}$ as in the previous section with respect to this cardinal.
By \Cref{prop:const}, every doctrine and morphism that appear in the image of the diagram $D$ are implicational and existential, thus by \Cref{prop:addit_struct} also $\underline{P}$ is implicational and existential.
First of all we list all objects of $\CC$---hence also all objects of $\underline{\CC}$---as the set $\text{ob}\CC=\{B\}_{B\in\text{ob}\CC}$. For any fixed $B$, we can surely list all elements of $\underline{P}(B)$ as $\left\{\ecc{(\varphi^{B}_j,\bar{X}^{(B,j)})}\right\}_{j\in\Lambda}$ where we fix a representative
\begin{equation*}\varphi^{B}_j\in P(\prod_{a=1}^{n^{(B,j)}}X_a^{(B,j)}\times B)\end{equation*}
for a given list $\bar{X}^{(B,j)}=\big((X_1^{(B,j)},x_1^{(B,j)}),\dots,(X_{n^{(B,j)}}^{(B,j)},x_{n^{(B,j)}}^{(B,j)})\big)$ in $J$.
Now consider all formulae of the kind
\begin{equation*}\underline{\exists}^{B}_\tmn\ecc{(\varphi^{B}_j,\bar{X}^{(B,j)})}, \text{ for all }j\in\Lambda.\end{equation*}
Then we have in $\underline{P}(\tmn)$
\begin{equation*}\underline{\exists}^{B}_\tmn\ecc{(\varphi^{B}_j,\bar{X}^{(B,j)})}=\ecc{(\exists^{B}_{\Pi X_a^{(B,j)}}\varphi^{B}_j,\bar{X}^{(B,j)})},\end{equation*}
where we recall the adjunction in $\Pos$:
\[\begin{tikzcd}
	{P(\prod X_a^{(B,j)}\times B)} & {P(\prod X_a^{(B,j)})}
	\arrow[""{name=0, anchor=center, inner sep=0}, "{\exists^{B}_{\Pi X_a^{(B,j)}}}", shift left=2, from=1-1, to=1-2]
	\arrow[""{name=1, anchor=center, inner sep=0}, "{P(\pr1)}", shift left=2, from=1-2, to=1-1]
	\arrow["\bot"{description}, Rightarrow, draw=none, from=0, to=1]
\end{tikzcd}.\]
Here we write $\pr1$ meaning the first projection from the product $(\prod X_a^{(B,j)})\times B$.

For any fixed $B\in\text{ob}\CC$ and for any $j\in\Lambda$ define $d^{B}_j$ as follows:
\begin{itemize}
\item if $j=0$, then $d^{B}_0$ is the smallest ordinal such that
\begin{equation*}d^{B}_0> x_a^{(B,0)}\text{ for any }{a=1,\dots,n^{(B,0)}};\end{equation*}
\item if $j$ is a successor or a limit ordinal, then $d^{B}_j$ is the smallest ordinal such that $d^{B}_j>d^{B}_h$ for all $h<j$ and such that
\begin{equation*}d^{B}_j>x_a^{(B,k)}\text{ for any }a=1,\dots,{n^{(B,k)}}\text{ and } k\leq j.\end{equation*}
\end{itemize}
Note that in particular for any $j\in\Lambda$:
\begin{equation*}
(B,d^{B}_j)\notin\{(X_a^{(B,j)},x_a^{(B,j)})\}_{a=1}^{n^{(B,j)}}.
\end{equation*}
Now, since
\begin{equation*}\varphi^{B}_j\in P(\prod_{a=1}^{n^{(B,j)}}X_a^{(B,j)}\times B)\end{equation*}
we can take its equivalence class with respect to $\bar{X}^{(B,j)}\in J$, hence we end up in $\underline{P}(B)$, or with respect to the list $\bar{X}^{(B,j)}_\star=\big((X_1^{(B,j)},x_1^{(B,j)}),\dots,(X_{n^{(B,j)}}^{(B,j)},x_{n^{(B,j)}}^{(B,j)}),(B,d^B_j)\big)$---i.e.\ adding $(B,d^B_j)$ to the list $\bar{X}^{(B,j)}$---, hence we end up in $\underline{P}(\tmn)$.
We have in $\underline{P}(\tmn)$ the element
\begin{equation*}\underline{\exists}^{B}_\tmn\ecc{(\varphi^{B}_j,\bar{X}^{(B,j)})}\longrightarrow\ecc{(\varphi^{B}_j,\bar{X}^{(B,j)}_\star)}.\end{equation*}
Define in $P(\prod X_a^{(B,j)}\times B)$
\begin{equation*}\psi^B_{j}\coloneqq P(\pr1)\exists^{B}_{\Pi X_a^{(B,j)}}\varphi^{B}_j\longrightarrow\varphi^{B}_j\end{equation*}
so that taking its class with respect to $\bar{X}^{(B,j)}_\star$ we get
\begin{equation}\label{eq:comp_impl}\ecc{(\psi^B_{j},\bar{X}^{(B,j)}_\star)}\in\underline{P}(\tmn),\text{ with }\ecc{(\psi^B_{j},\bar{X}^{(B,j)}_\star)}=\underline{\exists}^{B}_\tmn\ecc{(\varphi^{B}_j,\bar{X}^{(B,j)})}\longrightarrow\ecc{(\varphi^{B}_j,\bar{X}^{(B,j)}_\star)}.\end{equation}

Now, starting from $\underline{P}$, we do another construction.
\pointinproof{The directed preorder $I$ and the diagram $\DELTA:I\to\Dott$}
Define the poset $I$ of finite sets of pairs of the kind $(B,j)$, where $B\in\text{ob}\CC$ and $j\in\Lambda$, ordered by inclusion. We also want the empty set to belong to $I$.
\[\begin{tikzcd}
	I &&& \Dott \\
	\emptyset &&& {\underline{P}:\underline{\CC}\op\to\Pos} \\
	{\mathcal{U}=\{(B_1,j_1),\dots,(B_n,j_n)\}} &&& {\underline{P}^\mathcal{U}:\underline{\CC}\op\to\Pos} \\
	{\mathcal{V}=\{(B_1,j_1),\dots,(B_{n+m},j_{n+m})\}} &&& {\underline{P}^\mathcal{V}:\underline{\CC}\op\to\Pos}
	\arrow["{\DELTA}", from=1-1, to=1-4]
	\arrow[maps to, from=2-1, to=2-4]
	\arrow["\subseteq"{marking}, draw=none, from=2-1, to=3-1]
	\arrow["\subseteq"{marking}, draw=none, from=3-1, to=4-1]
	\arrow["{(\id{},\mathfrak{f}_{\mathcal{U}})}", from=2-4, to=3-4]
	\arrow["{(\id{},\mathfrak{f}_{\mathcal{UV}})}", from=3-4, to=4-4]
	\arrow[maps to, from=3-1, to=3-4]
	\arrow[maps to, from=4-1, to=4-4]
\end{tikzcd}\]
where:
\begin{itemize}
\item $\underline{P}^\mathcal{U}(A)=\underline{P}(A)_{\downarrow\underline{P}(!)\bigwedge_{i=1}^n\ecc{(\psi^{B_i}_{j_i},\bar{X}^{(B_i,j_i)}_\star)}}$, with definition on arrows as in \eqref{def:ax};
\item $(\mathfrak{f}_{\mathcal{U}})_A:\underline{P}(A)\to \underline{P}(A)_{\downarrow\underline{P}(!)\bigwedge_{i=1}^n\ecc{(\psi^{B_i}_{j_i},\bar{X}^{(B_i,j_i)}_\star)}}$ is the assignment
\begin{equation*}\ecc{(\alpha,\bar{Y})}\mapsto\ecc{(\alpha,\bar{Y})}\land\underline{P}(!)\bigwedge_{i=1}^n\ecc{(\psi^{B_i}_{j_i},\bar{X}^{(B_i,j_i)}_\star)};\end{equation*}
\item $(\mathfrak{f}_{\mathcal{UV}})_A:\underline{P}(A)_{\downarrow\underline{P}(!)\bigwedge_{i=1}^n\ecc{(\psi^{B_i}_{j_i},\bar{X}^{(B_i,j_i)}_\star)}}\to\underline{P}(A)_{\downarrow\underline{P}(!)\bigwedge_{i=1}^{n+m}\ecc{(\psi^{B_i}_{j_i},\bar{X}^{(B_i,j_i)}_\star)}}$ is again the assignment
\begin{equation*}\ecc{(\alpha,\bar{Y})}\mapsto\ecc{(\alpha,\bar{Y})}\land\underline{P}(!)\bigwedge_{i=1}^{n+m}\ecc{(\psi^{B_i}_{j_i},\bar{X}^{(B_i,j_i)}_\star)}.\end{equation*}
\end{itemize}
Use associativity and commutativity of conjunction to observe that this is a diagram.
\pointinproof{The colimit of $\DELTA$}
Let $\underrightarrow{P}:\underline{\CC}\op\to\Pos$ be the colimit of $\DELTA$ in $\Dott$---we refer again to the proof of \Cref{prop:dir_colim_app} in \Cref{sub:dir_colim} to know the details of how it is computed. The base category is $\underline{\CC}$, since all functors in the 1-arrows of the diagram are identities.

The fibers of the doctrine are defined as
\begin{equation*}\underrightarrow{P}(C)=\faktor{\bigsqcup_{\mathcal{U}\in I}\underline{P}^\mathcal{U}(C)}{\sim},\end{equation*}
where $[\ecc{(\alpha,\bar{Y})},\mathcal{U}]\in\underrightarrow{P}(C)$ for some $\ecc{(\alpha,\bar{Y})}\in\underline{P}(C)$ such that $\ecc{(\alpha,\bar{Y})}\leq\underline{P}(!)\bigwedge_{i=1}^n\ecc{(\psi^{B_i}_{j_i},\bar{X}^{(B_i,j_i)}_\star)}$ with a fixed $\mathcal{U}=\{(B_1,j_1),\dots,(B_n,j_n)\}\in I$. Here $[\ecc{(\alpha,\bar{Y})},\mathcal{U}]=[\ecc{(\beta,\bar{Z})},\mathcal{V}]$, for $\ecc{(\beta,\bar{Z})}\in\underline{P}(C)$, $\ecc{(\beta,\bar{Z})}\leq\underline{P}(!)\bigwedge_{r=1}^m\ecc{(\psi^{D_r}_{l_r},\bar{X}^{(D_r,l_r)}_\star)}$ with a fixed $\mathcal{V}=\{(D_1,l_1),\dots,(D_m,l_m)\}\in I$, if there exists a $\mathcal{W}=\{(A_1,q_1),\dots,(A_z,q_z)\}\supseteq\mathcal{U,V}$ such that in $\underline{P}(C)$ we have
\begin{equation*}\ecc{(\alpha,\bar{Y})}\land\underline{P}(!)\bigwedge_{k=1}^z\ecc{(\psi^{A_k}_{q_k},\bar{X}^{(A_k,q_k)}_\star)}=\ecc{(\beta,\bar{Z})}\land\underline{P}(!)\bigwedge_{k=1}^z\ecc{(\psi^{A_k}_{q_k},\bar{X}^{(A_k,q_k)}_\star)}.\end{equation*}
This assignment appropriately extends to arrows in $\underline{\CC}$. Also in this case, every doctrine and doctrine morphism in the image of the diagram $\DELTA$ are implicational and existential---this follows from \Cref{prop:ax}---hence by \Cref{prop:addit_struct} the doctrine $\underrightarrow{P}$ is implicational and existential too.

Call $(\id{},\underrightarrow{f}):\underline{P}\to\underrightarrow{P}$ the map in the colimit starting from $\DELTA(\emptyset)$: a component of the natural transformation $\underrightarrow{\mathfrak{f}}_A$ sends $\ecc{(\alpha,\bar{X})}\in \underline{P}(A)$, for some $\alpha\in P(\prod_{i=1}^n X_a\times A)$, into $[\ecc{(\alpha,\bar{X})},\emptyset]\in\underrightarrow{P}(A)$.
\begin{remark}\label{rmk:recap}
We revise in a single diagram the two constructions we did above:
\begin{center}
\begin{tikzcd}
\CC\op\arrow[r,"{\underline{F}}\op"]\arrow[dr,"P"',""{name=L},bend right]&{\underline{\CC}}\op\arrow[r,"{\id{}}\op"]\arrow[d,""'{name=C},"\underline{P}"{name=M}]&\underline{\CC}\op\arrow[dl, "\underrightarrow{P}",""'{name=R}, bend left]\\
&\Pos
\arrow[rightarrow,"\underline{\mathfrak{f}}"near end,"\cdot"', from=L, to=C, bend left=10]
\arrow[rightarrow,"{\underrightarrow{\mathfrak{f}}}"near start,"\cdot"', from=M, to=R, bend left=10]
\end{tikzcd}
\end{center}
The doctrine $P:\CC\op\to\Pos$ has a small base category, and it is implicational and existential.

Call the composition $(\id{},\underrightarrow{\mathfrak{f}})\circ(\underline{F},\underline{\mathfrak{f}})=(F,\mathfrak{f})$, so that both $F$ and $\mathfrak{f}$ take the corresponding equivalence classes:
\begin{align*}F:\CC\to\underline{\CC},\qquad&(f:A\to B)\mapsto(\ecc{(f,\emptyset)}:A\dashrightarrow B)\\
\mathfrak{f}_A:P(A)\to\underrightarrow{P}(A),\qquad&\alpha\mapsto[\ecc{(\alpha,\emptyset)},\emptyset].\end{align*}
This morphism preserves implicational and existential structure because both $(\underline{F},\underline{\mathfrak{f}})$ and $(\id{},\underrightarrow{\mathfrak{f}})$ do---this follows again by \Cref{prop:addit_struct}.
\end{remark}

%%%%%%%%%%%%%%%%%%%%%%%%%%%%%%%%%%%%%%%%%%%%%%%%%%%%%%%%%%%%%%%
\subsection{Richness of $\protect\underrightarrow{P}$ and weak universal property}\label{sub:weak_univ_prop}
We next show the main result of the section, which is the proof of the fact that the doctrine $\underrightarrow{P}$ defined in \Cref{sect:dir_colim2} is rich; moreover, the morphism $(F,\mathfrak{f}):P\to\underrightarrow{P}$ satisfies a weak universal property, meaning that any other implicational existential morphism $P\to R$ with $R$ rich factors through $(F,\mathfrak{f})$.
\begin{theorem}
Let $P$ be an implicational existential doctrine with a small base category. Then the doctrine $\underrightarrow{P}$ is rich.
\end{theorem}
\begin{proof}
Given $[\ecc{(\varphi,\bar{Y})},\mathcal{U}]\in\underrightarrow{P}(B)$, we will find an arrow $\ecc{(c,\bar{Z})}:\tmn\dashrightarrow B$ such that
\begin{equation*}\underrightarrow{\exists}^B_\tmn[\ecc{(\varphi,\bar{Y})},\mathcal{U}]\leq\underrightarrow{P}(\ecc{(c,\bar{Z})})[\ecc{(\varphi,\bar{Y})},\mathcal{U}].\end{equation*}

Note that $[\ecc{(\varphi,\bar{Y})},\mathcal{U}]=[\ecc{(\varphi,\bar{Y})},\emptyset]$: indeed taking $\mathcal{U}\supseteq\mathcal{U},\emptyset$ we have in $\underline{P}(B)$
\begin{equation*}\ecc{(\varphi,\bar{Y})}\land\underline{P}(!)\bigwedge_{i=1}^n\ecc{(\psi^{B_i}_{j_i},\bar{X}^{(B_i,j_i)}_\star)}=\ecc{(\varphi,\bar{Y})}.\end{equation*}

Moreover, since $\ecc{(\varphi,\bar{Y})}\in\underline{P}(B)$ in particular $\ecc{(\varphi,\bar{Y})}=\ecc{(\varphi^B_j,\bar{X}^{(B,j)})}$ for some $j\in\Lambda$, with
\begin{equation*}\varphi^{B}_j\in P(\prod_{a=1}^{n^{(B,j)}}X_a^{(B,j)}\times B).\end{equation*}
First of all compute $\underrightarrow{\exists}^B_\tmn[\ecc{(\varphi,\bar{Y})},\mathcal{U}]=\underrightarrow{\exists}^B_\tmn[\ecc{(\varphi^B_j,\bar{X}^{(B,j)})},\emptyset]=[\underline{\exists}^B_\tmn\ecc{(\varphi^B_j,\bar{X}^{(B,j)})},\emptyset]$.
Then let $c\coloneqq \id{B}$ and $\bar{Z}\coloneqq (B,d^B_j)$, so we can consider $\ecc{(\id{B},{(B,d^B_j)})}:\tmn\dashrightarrow B$ the equivalence class of the identity
\begin{equation*}\id{B}:B\to B.\end{equation*}
We then compute
\begin{equation*}\underrightarrow{P}(\ecc{(\id{B},{(B,d^B_j)})})[\ecc{(\varphi,\bar{Y})},\mathcal{U}]=\underrightarrow{P}(\ecc{(\id{B},{(B,d^B_j)})})[\ecc{(\varphi^B_j,\bar{X}^{(B,j)})},\emptyset]=[\ecc{(\varphi^B_j,\bar{X}^{(B,j)}_\star)},\emptyset],\end{equation*}
thus, in $\underrightarrow{P}(\tmn)$ we have
\begin{equation*}\underrightarrow{\exists}^B_\tmn[\ecc{(\varphi,\bar{Y})},\mathcal{U}]\leq\underrightarrow{P}(\ecc{(\id{B},{(B,d^B_j)})})[\ecc{(\varphi,\bar{Y})},\mathcal{U}]\end{equation*}
if and only if
\begin{equation*}[\ecc{(\top,\emptyset)},\emptyset]\leq\underrightarrow{\exists}^B_\tmn[\ecc{(\varphi,\bar{Y})},\mathcal{U}]\longrightarrow\underrightarrow{P}(\ecc{(c,{(B,d^B_j)})})[\ecc{(\varphi,\bar{Y})},\mathcal{U}],\end{equation*}
i.e.\
\begin{equation*}[\ecc{(\top,\emptyset)},\emptyset]\leq[\underline{\exists}^B_\tmn\ecc{(\varphi^B_j,\bar{X}^{(B,j)})},\emptyset]\longrightarrow[\ecc{(\varphi^B_j,\bar{X}^{(B,j)}_\star)},\emptyset];\end{equation*}
but then compute the implication in $\underline{P}(\tmn)$ as seen in \eqref{eq:comp_impl} to get
\begin{equation*}[\ecc{(\top,\emptyset)},\emptyset]\leq[\ecc{(\psi^B_{j},\bar{X}^{(B,j)}_\star)} ,\emptyset]\end{equation*}
which holds since $[\ecc{(\psi^B_{j},\bar{X}^{(B,j)}_\star)} ,\emptyset]$ is the top element of $\underrightarrow{P}(\tmn)$ by definition: take $\{(B,j)\}\supseteq\emptyset$ and observe that in $\underline{P}(\tmn)$:
\begin{equation*}\ecc{(\top,\emptyset)}\land\ecc{(\psi^B_{j},\bar{X}^{(B,j)}_\star)}=\ecc{(\psi^B_{j},\bar{X}^{(B,j)}_\star)}\land\ecc{(\psi^B_{j},\bar{X}^{(B,j)}_\star)}.\end{equation*}
This concludes the proof that $\underrightarrow{P}$ is rich.
\end{proof}
\begin{theorem}\label{thm:weak_univ_prop}
Let $P:\CC\op\to\Pos$ be an implicational existential doctrine with a small base category. The 1-arrow $(F,\mathfrak f):P\to \underrightarrow{P}$ is implicational existential and it is such that $\underrightarrow{P}$ is rich, and it is weakly universal with respect to this property, i.e.\ for any implicational existential morphism $(H,\mathfrak{h}):P\to R$ where $R:\mathbb{D}^{\op}\to\Pos$ is an implicational existential rich doctrine, there exists an implicational existential 1-arrow $(G,\mathfrak g):\underrightarrow{P}\to R$ such that $(G,\mathfrak g)(F,\mathfrak f)=(H,\mathfrak h)$.

Moreover, if $P$, $R$ and $(H,\mathfrak h)$ are respectively bounded, universal, elementary, then such $(G,\mathfrak g)$ is respectively bounded, universal, elementary.
\[\begin{tikzcd}
	P && R \\
	& \underrightarrow{P}
	\arrow[from=1-1, to=1-3,"{(H,\mathfrak{h})}"]
	\arrow[from=1-1, to=2-2,"{(F,\mathfrak{f})}"']
	\arrow[from=2-2, to=1-3,"{(G,\mathfrak{g})}"', dashed]
\end{tikzcd}\]
\end{theorem}
\begin{proof}
We rewrite the two colimit diagrams:
\[\begin{tikzcd}
	&& {P_{\Pi X_a}} \\
	P &&&& {\underline{P}} \\
	&& {P_{\Pi Y_b}}
	\arrow["{(F_{\bar{X}},\mathfrak{f}_{\bar{X}})}", from=2-1, to=1-3]
	\arrow["{(F_{\bar{Y}},\mathfrak{f}_{\bar{Y}})}"', from=2-1, to=3-3]
	\arrow["{(F'_{\bar{Y}},\mathfrak{f}'_{\bar{Y}})}"', from=3-3, to=2-5]
	\arrow["{(F_{\bar{X}\bar{Y}},\mathfrak{f}_{\bar{X}\bar{Y}})}"{description, pos=0.6}, from=1-3, to=3-3]
	\arrow["{(\underline{F},\underline{\mathfrak{f}})}"{description, pos=0.3}, bend left=15, from=2-1, to=2-5, crossing over]
	\arrow["{(F'_{\bar{X}},\mathfrak{f}'_{\bar{X}})}", from=1-3, to=2-5]
\end{tikzcd}\]
\[\begin{tikzcd}
	&& {\underline{P}^{\mathcal{U}}} \\
	\underline{P} &&&& {\underrightarrow{P}} \\
	&& {\underline{P}^{\mathcal{V}}}
	\arrow["{(\id{},\mathfrak{f}_{\mathcal{U}})}", from=2-1, to=1-3]
	\arrow["{(\id{},\mathfrak{f}_{\mathcal{V}})}"', from=2-1, to=3-3]
	\arrow["{(\id{},\mathfrak{q}_{\mathcal{V}})}"', from=3-3, to=2-5]
	\arrow["{{(\id{},\mathfrak{f}_{\mathcal{UV}})}}"{description, pos=0.6}, from=1-3, to=3-3]
	\arrow["{(\id{},\underrightarrow{\mathfrak{f}})}"{description, pos=0.3}, bend left=15, from=2-1, to=2-5, crossing over]
	\arrow["{(\id{},\mathfrak{q}_{\mathcal{U}})}", from=1-3, to=2-5]
\end{tikzcd}.\]

Using the universal properties of the colimit diagrams, and the universal properties of the arrows $(F_{\bar{X}},\mathfrak{f}_{\bar{X}}):P\to P_{\Pi X_a}$ and $(\id{},\mathfrak{f}_{\mathcal{U}}):\underline{P}\to\underline{P}^{\mathcal{U}}$, we know that defining a doctrine morphism $(G,\mathfrak g):\underrightarrow{P}\to R$ is equivalent to defining a doctrine morphism $(G,\mathfrak s):\underline{P}\to R$ such that $\mathfrak{s}:\underline{P}\xrightarrow{\cdot} RG$ maps each $\ecc{(\psi^{B}_{j},\bar{X}^{(B,j)}_\star)}\in \underline{P}(\tmn)$ to the top element of $R(\tmn_\ct{D})$. However, defining $(G,\mathfrak s):\underline{P}\to R$ is equivalent to choosing a doctrine morphism $P\to R$---let it be $(H,\mathfrak h):P\to R$---and a choice of a constant $\tmn_\ct{D}\to HY$ in $\ct{D}$ for every $(Y,\lambda)\in\text{ob}\CC\times\Lambda$. If we manage to do so, we will have a diagram as below, where all triangles commute:
\[\begin{tikzcd}[column sep=scriptsize]
	P && {\underline{P}} && {\underrightarrow{P}} && R
	\arrow["{(\id{},\underrightarrow{\mathfrak f})}"'xshift=-3pt, from=1-3, to=1-5]
	\arrow["{(G,\mathfrak g)}"', dashed, from=1-5, to=1-7]
	\arrow["{(F,\mathfrak f)}"', curve={height=24pt}, from=1-1, to=1-5]
	\arrow["{(G,\mathfrak s)}"{description, pos=0.3}, curve={height=-18pt}, dashed, from=1-3, to=1-7]
	\arrow["{(H,\mathfrak h)}", curve={height=-30pt}, from=1-1, to=1-7]
	\arrow["{(\underline{F},\underline{\mathfrak f})}"', from=1-1, to=1-3]
\end{tikzcd}.\]
So to recap, our goal is to find a suitable choice of constants in the base category of $R$, such that the induced doctrine morphism $(G,\mathfrak s):\underline{P}\to R$ maps each $\ecc{(\psi^{B}_{j},\bar{X}^{(B,j)}_\star)}$ to the top element of $R(\tmn_\ct{D})$.

First of all, fix a well-ordering of $\text{ob}\CC$, and consider the lexicographic order on $\text{ob}\CC\times\Lambda$.
Recall that for any given object $B$ we have $\underline{P}(B)=\left\{\ecc{(\varphi^{B}_j,\bar{X}^{(B,j)})}\right\}_{j\in\Lambda}$ where
\begin{equation*}\varphi^{B}_j\in P(\prod_{a=1}^{n^{(B,j)}}X_a^{(B,j)}\times B).\end{equation*}

We start from $(B_0,0)$: consider $\varphi^{B_0}_0\in P(\prod_{a=1}^{n^{(B_0,0)}}X_a^{(B_0,0)}\times B_0)$ which is used to define the last entry of the list $\bar{X}_\star^{(B_0,0)}$ in \Cref{sect:label}. Take $\mathfrak{h}_{\Pi X_a\times B_0}\varphi^{B_0}_0\in R(\prod HX^{(B_0,0)}_a\times H B_0)$, hence there exists a constant in $\mathbb{D}$---which is actually a list of constants
\begin{equation*}c^{(B_0,0)}=\ple{c_{(X_1^{(B_0,0)},x_1^{(B_0,0)})},\dots,c_{(X_{n^{(B_0,0)}}^{(B_0,0)},x_{n^{(B_0,0)}}^{(B_0,0)})},c_{(B_0,d_0^{B_0})}}:\tmn_\ct{D}\to\prod_{a=1}^{n^{(B_0,0)}}HX_a^{(B_0,0)}\times HB_0\end{equation*}
such that
\begin{equation*}\exists^{\Pi HX_a\times HB_0}_{\tmn_\ct{D}}\mathfrak{h}_{\Pi X_a\times B_0}\varphi^{B_0}_0\leq R(c^{(B_0,0)})\mathfrak{h}_{\Pi X_a\times B_0}\varphi^{B_0}_0\end{equation*}
by using the richness property of $R$.
This defines an assignment $(Y,\lambda)\mapsto (c_{(Y,\lambda)}:\tmn_\ct{D}\to HY)$ for some pair $(Y,\lambda)$---the ones of the kind $(X_a^{(B_0,0)},x_a^{(B_0,0)})$ for $a=1,\dots,n^{(B_0,0)}$, and of the kind $(B_0,d^{B_0}_0)$: our goal is to extend this to every pair of such kinds.
Consider now $(B,j)>(B_0,0)$---i.e.\ $B>B_0$ in $\text{ob}\CC$, or $B=B_0$ and $j>0$---, and take $\varphi^{B}_j\in P(\prod_{a=1}^{n^{(B,j)}}X_a^{(B,j)}\times B)$. Take all the pairs $(X_b^{(B,j)},x_b^{(B,j)})$ that have already appeared as subscripts in the components of some $c^{(A,i)}$ for some $(A,i)<(B,j)$. Their indexes form a subset $K^{(B,j)}\subseteq\{1,\dots,n^{(B,j)}\}$.
Evaluate the element $\mathfrak{h}_{\Pi X_a\times B}\varphi^{B}_j\in R(\prod_{a=1}^{n^{(B,j)}}HX_a^{(B,j)}\times HB)$ in the corresponding constants:
\begin{equation}\label{eq:Rdots}
\begin{tikzcd}[row sep=small]
	{R(\ple{\pr1,\dots,c_{(X_b^{(B,j)},x_b^{(B,j)})},\dots,\pr{n^{(B,j)}},\pr{n^{(B,j)}+1}})\mathfrak{h}_{\Pi X_a\times B}\varphi^{B}_j} & {} \\
	{R(HX_1^{(B,j)}\times\dots\times \widehat{HX_b^{(B,j)}}\times\dots HX_{n^{(B,j)}}^{(B,j)}\times HB)}
	\arrow["{\text{\rotatebox{90}{$\ni$}}}"{description}, draw=none, from=1-1, to=2-1]
\end{tikzcd}\end{equation}
where each $\widehat{HX_b^{(B,j)}}$ for $b\in K^{(B,j)}$ is the terminal object $\tmn_\ct{D}$. Let
\begin{equation*}\widehat{\prod}HX_a^{(B,j)}=\prod_{a\notin K^{(B,j)}}HX_a^{(B,j)},\end{equation*}
and observe that there exists a canonical isomorphism
\begin{equation*}
\omega_{(B,j)}:\widehat{\prod}HX_a^{(B,j)}\times HB\longrightarrow HX_1^{(B,j)}\times\dots\times \widehat{HX_b^{(B,j)}}\times\dots HX_{n^{(B,j)}}^{(B,j)}\times HB.
\end{equation*}
So now there exists a list of constants
\begin{equation*}c^{(B,j)}=\ple{\dots,c_{(X_a^{(B,j)},x_a^{(B,j)})},\dots,c_{(B,d_j^{B})}}:\tmn_\ct{D}\to\widehat{\prod}HX_a^{(B,j)}\times HB\end{equation*}
such that
\begin{equation}\label{eq:wit}
\begin{aligned}\exists^{\widehat{\Pi}HX_a\times HB}_{\tmn_\ct{D}} R(\omega_{(B,j)})R(\ple{\pr1,\dots,c_{(X_b^{(B,j)},x_b^{(B,j)})},\dots,\pr{n^{(B,j)}+1}})\mathfrak{h}_{\Pi X_a\times B}\varphi^{B}_j \\
\leq R(c^{(B,j)})R(\omega_{(B,j)})R(\ple{\pr1,\dots,c_{(X_b^{(B,j)},x_b^{(B,j)})},\dots,\pr{n^{(B,j)}+1}})\mathfrak{h}_{\Pi X_a\times B}\varphi^{B}_j\end{aligned}\end{equation}
by using again the richness property of $R$. Note that the projections and constants that appear in the argument of $R$ here in \eqref{eq:wit} are the same that appear in \eqref{eq:Rdots}.

In this way, we are able to define $c_{(Y,\lambda)}\colon\tmn_\ct{D}\to HY$ for all pairs $(Y,\lambda)$ of the kind $(X_a^{(B,j)},x_a^{(B,j)})$ with $a\notin K^{(B,j)}$, and of the kind $(B,d^B_j)$.
Once completed the assignments given by all pairs $(B,j)\in\mathrm{ob}\CC\times\Lambda$, extend then the assignment $(Y,\lambda)\mapsto c_{(Y,\lambda)}$ to all the remaining pairs by choosing any constant $c_{(Y,\lambda)}:\tmn_\ct{D}\to HY$. To do so, recall that since $R$ is rich, for any object $D$ in $\ct{D}$ there exists a map $\tmn_\ct{D}\to D$. This choice of constants defines a unique $(G,\mathfrak s):\underline{P}\to R$ such that $(G,\mathfrak s)(\underline{F},\underline f)=(H,\mathfrak h)$, and such that for any $\bar{X}=\big((X_1,x_1),\dots,(X_n,x_n)\big)\in J$ we have
\[G([\id{\Pi X_a},\bar{X}]:\tmn\dashrightarrow\prod_{a=1}^nX_a)=(\ple{c_{(X_1,x_1)},\dots,c_{(X_n,x_n)}}:\tmn_\ct{D}\to\prod_{a=1}^nHX_a).\]

Now, consider the $\underline{\CC}$-arrow
\begin{equation*}\ecc{(\id{},\bar{X}^{(B,j)}_\star)}:\tmn\dashrightarrow \prod_{a=1}^{n^{(B,j)}}X_a^{(B,j)}\times B,\end{equation*}
equivalence class of the identity arrow in $\CC$
\begin{equation*}\id{} :\prod_{a=1}^{n^{(B,j)}}X_a^{(B,j)}\times B\to \prod_{a=1}^{n^{(B,j)}}X_a^{(B,j)}\times B\end{equation*}
with respect to the list $\bar{X}^{(B,j)}_\star=\big((X_1^{(B,j)},x_1^{(B,j)}),\dots,(X_{n^{(B,j)}}^{(B,j)},x_{n^{(B,j)}}^{(B,j)}),(B,d^B_j)\big)$.
The reindexing in $\underline{P}$ along this map is the evaluation in the corresponding new constants added by the colimit.
Compute now $\mathfrak{s}_{\tmn}\ecc{(\psi^{B}_{j},\bar{X}^{(B,j)}_\star)}$, using the naturality of $\mathfrak{s}$ and the commutativity of the triangle $(H,\mathfrak{h})=(G,\mathfrak{s})\circ(\underline{F},\underline{\mathfrak{f}})$:
\begin{align*}\mathfrak{s}_{\tmn}\ecc{(\psi^{B}_{j},\bar{X}^{(B,j)}_\star)}&=\mathfrak{s}_{\tmn}\underline{P}\left(\ecc{(\id{},\bar{X}^{(B,j)}_\star)}\right)\ecc{(\psi^{B}_{j},\emptyset)}\\
&=RG\left(\ecc{(\id{},\bar{X}^{(B,j)}_\star)}\right)\mathfrak{s}_{\Pi X_a\times B}\ecc{(\psi^{B}_{j},\emptyset)}\\
&=R(\ple{c_{(X_1^{(B,j)},x_1^{(B,j)})},\dots,c_{(X_{n^{(B,j)}}^{(B,j)},x_{n^{(B,j)}}^{(B,j)})},c_{(B,d_j^{B})}})\mathfrak{h}_{\Pi X_a\times B}\psi^B_j.\end{align*}
Note that we removed the superscripts in the objects of the natural transformation to lighten the notation. We also write $\overline c$ for the list of $\mathbb{D}$-constants $\ple{c_{(X_1^{(B,j)},x_1^{(B,j)})},\dots,c_{(X_{n^{(B,j)}}^{(B,j)},x_{n^{(B,j)}}^{(B,j)})},c_{(B,d_j^{B})}}$ for simplicity, and recall that
\begin{equation*}\psi^B_j= P(\ple{\pr1,\dots,\pr{n^{(B,j)}}})\exists^{B}_{\Pi X_a^{(B,j)}}\varphi^{B}_j\longrightarrow\varphi^{B}_j.\end{equation*}
So $\top\leq\mathfrak{s}_{\tmn}\ecc{(\psi^{B}_{j},\bar{X}^{(B,j)}_\star)}$ if and only if
\begin{equation*}R(\overline c)\mathfrak{h}_{\Pi X_a\times B}P(\ple{\pr1,\dots,\pr{n^{(B,j)}}})\exists^{B}_{\Pi X_a}\varphi^{B}_j\leq R(\overline c)\mathfrak{h}_{\Pi X_a\times B}\varphi^{B}_j;\end{equation*}
using naturality of $\mathfrak{h}$ and the fact that $H$ preserves products, and then the fact that $\mathfrak{h}$ preserves the existential quantifier, we get
\begin{align*}R(\overline c)\mathfrak{h}_{\Pi X_a\times B}P(\ple{\pr1,\dots,\pr{n^{(B,j)}}})\exists^{B}_{\Pi X_a}\varphi^{B}_j=&R(\overline c)R(\ple{\pr1,\dots,\pr{n^{(B,j)}}})\mathfrak{h}_{\Pi X_a}\exists^{B}_{\Pi X_a}\varphi^{B}_j\\
=&R(\overline c)R(\ple{\pr1,\dots,\pr{n^{(B,j)}}})\exists^{HB}_{\Pi HX_a}\mathfrak{h}_{\Pi X_a\times B}\varphi^{B}_j,\end{align*}
so we need to prove
\begin{equation}\label{eq:claim}R(\ple{c_{(X_1^{(B,j)},x_1^{(B,j)})},\dots,c_{(X_{n^{(B,j)}}^{(B,j)},x_{n^{(B,j)}}^{(B,j)})}})\exists^{HB}_{\Pi HX_a}\mathfrak{h}_{\Pi X_a\times B}\varphi^{B}_j\leq R(\overline c)\mathfrak{h}_{\Pi X_a\times B}\varphi^{B}_j.\end{equation}
Observe that in the right-hand side of \eqref{eq:claim} we have exactly the same element of the right-hand side of \eqref{eq:wit}.
\[\begin{tikzcd}
	\tmn_\ct{D} \\
	{\widehat{\prod}HX_a^{(B,j)}\times HB} \\
	{HX_1^{(B,j)}\times\dots\times \widehat{HX_b^{(B,j)}}\times\dots HX_{n^{(B,j)}}^{(B,j)}\times HB} \\
	{HX_1^{(B,j)}\times\dots\times {HX_b^{(B,j)}}\times\dots HX_{n^{(B,j)}}^{(B,j)}\times HB}
	\arrow["{c^{(B,j)}=\ple{\dots,c_{(X_a^{(B,j)},x_a^{(B,j)})},\dots,c_{(B,d_j^{B})}}}", from=1-1, to=2-1]
	\arrow["{\omega_{(B,j)}}","\rotatebox{-90}{$\sim$}"', from=2-1, to=3-1]
	\arrow["{\ple{\pr1,\dots,c_{(X_b^{(B,j)},x_b^{(B,j)})},\dots,\pr{n^{(B,j)}+1}}}", from=3-1, to=4-1]
	\arrow["{\overline c}"', curve={height=170pt}, from=1-1, to=4-1]
\end{tikzcd}\]

So it is enough to prove
\begin{align*}&R(\ple{c_{(X_1^{(B,j)},x_1^{(B,j)})},\dots,c_{(X_{n^{(B,j)}}^{(B,j)},x_{n^{(B,j)}}^{(B,j)})}})\exists^{HB}_{\Pi HX_a}\mathfrak{h}_{\Pi X_a\times B}\varphi^{B}_j\\
&\leq\exists^{\widehat{\Pi}HX_a\times HB}_{\tmn_\ct{D}} R(\omega_{(B,j)})R(\ple{\pr1,\dots,c_{(X_b^{(B,j)},x_b^{(B,j)})},\dots,\pr{n^{(B,j)}}}\times\id{HB})\mathfrak{h}_{\Pi X_a\times B}\varphi^{B}_j.\end{align*}

Write $\tau$ for the list $\ple{\pr1,\dots,c_{(X_b^{(B,j)},x_b^{(B,j)})},\dots,\pr{n^{(B,j)}}}$, so that
\begin{equation*}\ple{\pr1,\dots,c_{(X_b^{(B,j)},x_b^{(B,j)})},\dots,\pr{n^{(B,j)}+1}}=\tau\times\id{HB},\end{equation*}
write $\sigma$ for every component except for the last one for the map $c^{(B,j)}$, so that $c^{(B,j)}=\ple{\sigma,c_{(B,d^B_j)}}$, and write $\omega'$ for the canonical isomorphism
\begin{equation*}\widehat{\prod}HX_a^{(B,j)}\longrightarrow HX_1^{(B,j)}\times\dots\times \widehat{HX_b^{(B,j)}}\times\dots HX_{n^{(B,j)}}^{(B,j)},\end{equation*}
so that $\omega_{(B,j)}=\omega'\times\id{HB}$. In particular $\overline c=(\tau\times\id{HB})(\omega'\times\id{HB})c^{(B,j)}$, so we can rewrite  the list $\ple{c_{(X_1^{(B,j)},x_1^{(B,j)})},\dots,c_{(X_{n^{(B,j)}}^{(B,j)},x_{n^{(B,j)}}^{(B,j)})}}$ as follows:

\begin{align*}
\ple{c_{(X_1^{(B,j)},x_1^{(B,j)})},\dots,c_{(X_{n^{(B,j)}}^{(B,j)},x_{n^{(B,j)}}^{(B,j)})}}&=\ple{\pr1,\dots,\pr{n^{(B,j)}}}\overline{c}\\
&=\ple{\pr1,\dots,\pr{n^{(B,j)}}}(\tau\times\id{HB})(\omega'\times\id{HB})c^{(B,j)}\\
&=\ple{\pr1,\dots,\pr{n^{(B,j)}}}(\tau\omega'\times\id{HB})c^{(B,j)}\\&=\tau\omega'\sigma.
\end{align*}

So now we have:
\begin{align*}R(\ple{c_{(X_1^{(B,j)},x_1^{(B,j)})},\dots,c_{(X_{n^{(B,j)}}^{(B,j)},x_{n^{(B,j)}}^{(B,j)})}})\exists^{HB}_{\Pi HX_a}&=R(\sigma)R(\tau\omega')\exists^{HB}_{\Pi HX_a}\\
&=R(\sigma)\exists^{HB}_{\widehat{\Pi}HX_a}R(\tau\omega'\times\id{B});\end{align*}
hence, we are left to prove that $R(\sigma)\exists^{HB}_{\widehat{\Pi}HX_a}\leq\exists^{\widehat{\Pi}HX_a\times HB}_{\tmn_\ct{D}}$.

Now, since $\exists^{\widehat{\Pi}HX_a\times HB}_{\tmn_\ct{D}}=\exists^{\widehat{\Pi}HX_a}_{\tmn_\ct{D}}\exists^{HB}_{\widehat{\Pi}HX_a}$, we should prove $R(\sigma)\leq\exists^{\widehat{\Pi}HX_a}_{\tmn_\ct{D}}$, but this holds since $\sigma:\tmn_\ct{D}\to\widehat{\Pi}HX_a$ and we can apply $R(\sigma)$ to the unit $\id{R({\widehat{\Pi}HX_a})}\leq R(!_{\widehat{\Pi}HX_a})\exists^{\widehat{\Pi}HX_a}_{\tmn_\ct{D}}$.

Since we proved that $\top\leq\mathfrak{s}_{\tmn}\ecc{(\psi^{B}_{j},\bar{X}^{(B,j)}_\star)}$ for all $(B,j)$, we can define a unique $(G,\mathfrak g):\underrightarrow{P}\to R$ such that $(G,\mathfrak g)(\id{},\underrightarrow{\mathfrak f})=(G,\mathfrak s)$, hence in particular $(G,\mathfrak g)(F,\mathfrak{f})=(H,\mathfrak h)$.
To conclude, observe that since we defined $(G,\mathfrak{g})$ through directed colimits and constructions that add constants and axioms, implicational and existential structure are preserved by $(G,\mathfrak g)$; moreover, if $R$ has as additional structure any between bottom element, universal quantifier, elementary structure, preserved by $(H,\mathfrak{h})$, then also $(G,\mathfrak{g})$ does.
\end{proof}
We conclude the subsection extending \Cref{thm:weak_univ_prop} to 2-arrows.
\begin{proposition}
Let $P:\CC\op\to\Pos$ be an implicational existential doctrine with a small base category. Consider the 1-arrow $(F,\mathfrak f):P\to \underrightarrow{P}$, and let $(H,\mathfrak{h}):P\to R$ be an implicational existential morphism where $R:\mathbb{D}^{\op}\to\Pos$ is an implicational existential rich doctrine and let $(G,\mathfrak g):\underrightarrow{P}\to R$ be an implicational existential 1-arrow such that $(G,\mathfrak g)(F,\mathfrak f)=(H,\mathfrak h)$. Then the precomposition with $(F,\mathfrak f)$ induces an equivalence between the coslice categories
\begin{equation*}\blank\circ(F,\mathfrak f):(G,\mathfrak{g})\downarrow\Dott_{\land,\top,\to,\exists}(\underrightarrow{P},R)\longrightarrow(H,\mathfrak{h})\downarrow\Dott_{\land,\top,\to,\exists}(P,R).\end{equation*}
\end{proposition}
\begin{proof}
Take any two objects $\gamma:(G,\mathfrak{g})\to(M,\mathfrak m),\mu:(G,\mathfrak{g})\to(N,\mathfrak n)\in(G,\mathfrak{g})\downarrow\Dott_{\land,\top,\to,\exists}(\underrightarrow{P},R)$, for some $(M,\mathfrak m),(N,\mathfrak n):\underrightarrow{P}\to R$; then take an arrow $\delta:\gamma\to \mu$. Since the functor $F$ acts as the identity on objects, precomposition with $F$ applied to the natural transformations $\gamma,\mu$ and $\delta$ is the identity:
\[\begin{tikzcd}
	{(G,\mathfrak{g})\downarrow\Dott_{\land,\top,\to,\exists}(\underrightarrow{P},R)} &&& {(H,\mathfrak{h})\downarrow\Dott_{\land,\top,\to,\exists}(P,R)}
	\arrow["{\blank\circ(F,\mathfrak f)}", from=1-1, to=1-4]
\end{tikzcd}\]
\[\begin{tikzcd}
	& {(G,\mathfrak g)} &&& {(H,\mathfrak h)} \\
	\\
	{(M,\mathfrak m)} && {(N,\mathfrak n)} & {(M,\mathfrak m)(F,\mathfrak f)} && {(N,\mathfrak n)(F,\mathfrak f)}
	\arrow["\gamma"', from=1-2, to=3-1]
	\arrow["\mu", from=1-2, to=3-3]
	\arrow["\delta"', from=3-1, to=3-3]
	\arrow["\gamma"', from=1-5, to=3-4]
	\arrow["\mu", from=1-5, to=3-6]
	\arrow["\delta"', from=3-4, to=3-6]
\end{tikzcd}.\]
In particular, the faithfulness of the precomposition functor follows trivially.
We show that the functor is essentially surjective.

Take a 2-arrow $\gamma:(H,\mathfrak{h})\to(K,\mathfrak{k})$ where $(K,\mathfrak{k}):P\to R$ is an implicational existential morphism. We want to find a morphism $(M,\mathfrak{m}):\underrightarrow{P}\to R$ and a 2-arrow $(G,\mathfrak{g})\to(M,\mathfrak{m})$, where $(M,\mathfrak{m})$ makes the triangle with $(K,\mathfrak{k})$ commute.
\[\begin{tikzcd}
	P && R \\
	& {\underrightarrow{P}}
	\arrow[""{name=0, anchor=center, inner sep=0}, "{(K,\mathfrak{k})}"', shift right=1, bend left=14, from=1-1, to=1-3]
	\arrow[""{name=1, anchor=center, inner sep=0}, "{(H,\mathfrak{h})}", shift left=4, bend left=14, from=1-1, to=1-3]
	\arrow["{(F,\mathfrak{f})}"', bend right=14, from=1-1, to=2-2]
	\arrow[""{name=2, anchor=center, inner sep=0}, "{(G,\mathfrak{g})}"{pos=0.3}, shift left=2, bend right=14, from=2-2, to=1-3]
	\arrow[""{name=3, anchor=center, inner sep=0}, "{(M,\mathfrak{m})}"'{pos=0.3}, shift right=3, bend right=14, dotted, from=2-2, to=1-3]
	\arrow[shorten <=2pt, shorten >=2pt, Rightarrow, dotted, from=2, to=3]
	\arrow["\gamma", shorten <=1pt, shorten >=1pt, Rightarrow, from=1, to=0]
\end{tikzcd}\]
Recall that $(G,\mathfrak{g})$ is uniquely determined by $(H,\mathfrak h)$ and a choice of $c_{(X,x)}:\tmn_\ct{D}\to HX$ for each $(X,x)\in J$. Moreover, having a 2-arrow $\gamma$ means that we have a natural transformation $\gamma:H\xrightarrow{\cdot}K$ such that $\mathfrak{h}_X\leq R(\gamma_X)\mathfrak{k}_X$ for all $X\in\text{ob}\CC$.
To define $(M,\mathfrak{m})$, we look for a constant $d_{(X,x)}:\tmn_\ct{D}\to KX$ for any $(X,x)\in J$ such that the corresponding induced map $\underline{P}\to R$ maps each $\ecc{(\psi^{B}_{j},\bar{X}^{(B,j)}_\star)}\in \underline{P}(\tmn)$ in the top element of $R(\tmn_\ct{D})$. Define $d_{(X,x)}\coloneqq \gamma_X\cdot c_{(X,x)}$, and then we check that in $R(\tmn_\ct{D})$
\begin{equation*}\top\leq R(\ple{d_{(X_1^{(B,j)},x_1^{(B,j)})},\dots,d_{(X_{n^{(B,j)}}^{(B,j)},x_{n^{(B,j)}}^{(B,j)})},d_{(B,d_j^{B})}})\mathfrak{k}_{\Pi X_a\times B}\psi^B_j.\end{equation*}
By using naturality of $\gamma$ and the fact that both $H$ and $K$ preserve products, we get the following commutative triangle
\[\begin{tikzcd}
	\tmn_\ct{D} && {\prod KX_a\times KB} \\
	& {\prod HX_a\times HB}
	\arrow["{\ple{\dots,d_{(X_i^{(B,j)},x_i^{(B,j)})},\dots,d_{(B,d_j^{B})}}}", from=1-1, to=1-3]
	\arrow["{\ple{\dots,c_{(X_i^{(B,j)},x_i^{(B,j)})},\dots,c_{(B,d_j^{B})}}}"', from=1-1, to=2-2]
	\arrow["{\gamma_{\Pi X_a\times B}}"', from=2-2, to=1-3]
\end{tikzcd}.\]
Then, using the definition of $c_{(X,x)}$'s and the fact that $\gamma$ is a 2-arrow we have:
\begin{align*}\top\leq& R(\ple{c_{(X_1^{(B,j)},x_1^{(B,j)})},\dots,c_{(X_{n^{(B,j)}}^{(B,j)},x_{n^{(B,j)}}^{(B,j)})},c_{(B,d_j^{B})}})\mathfrak{h}_{\Pi X_a\times B}\psi^B_j\\
\leq& R(\ple{c_{(X_1^{(B,j)},x_1^{(B,j)})},\dots,c_{(X_{n^{(B,j)}}^{(B,j)},x_{n^{(B,j)}}^{(B,j)})},c_{(B,d_j^{B})}})R(\gamma_{\Pi X_a\times B})\mathfrak{k}_{\Pi X_a\times B}\psi^B_j\end{align*}
as claimed, so we defined a morphism $(M,\mathfrak m)$ such that $(M,\mathfrak m)(F,\mathfrak f)=(K,\mathfrak k)$.

To conclude essential surjectivity, we show that $\gamma$ is actually a 2-arrow also between $(G,\mathfrak{g})$ and $(M,\mathfrak{m})$.
Take any $\underline{\CC}$-arrow $\ecc{(f,\bar{X})}:A\dashrightarrow B$, where $f:\prod_{a=1}^n X_a\times A\to B$ is a $\CC$-arrow and $\bar{X}=\big((X_1,x_1),\dots,(X_n,x_n)\big)$ is a list in $J$. Naturality means that the following square commutes:
\[\begin{tikzcd}
	GA & MA \\
	GB & MB
	\arrow["{G\ecc{(f,\bar{X})}}"', from=1-1, to=2-1]
	\arrow["{\gamma_A}"', from=1-1, to=1-2]
	\arrow["{M\ecc{(f,\bar{X})}}", from=1-2, to=2-2]
	\arrow["{\gamma_B}", from=2-1, to=2-2]
\end{tikzcd}.\]
Observe that the $\mathbb{D}$-arrow $\gamma_A:HA\to KA$ is indeed an arrow from $GA$ to $MA$, because the functors $G$ and $M$ act like $H$ and $K$ on objects respectively.
Use now the definition of $G\ecc{(f,\bar{X})}$ and $M\ecc{(f,\bar{X})}$, so that we need to prove the commutativity of the outer rectangle:
\[\begin{tikzcd}
	HA & KA \\
	{\prod HX_a\times HA} & {\prod KX_a\times KA} \\
	HB & KB
	\arrow["{\gamma_A}", from=1-1, to=1-2]
	\arrow["{\ple{\overline{c}\cdot !,\id{HA}}}"', from=1-1, to=2-1]
	\arrow["{H(f)}"', from=2-1, to=3-1]
	\arrow["{\gamma_B}", from=3-1, to=3-2]
	\arrow["{\ple{\overline{d}\cdot !,\id{KA}}}", from=1-2, to=2-2]
	\arrow["{K(f)}", from=2-2, to=3-2]
	\arrow["{\gamma_{\Pi X_a\times A}}", from=2-1, to=2-2,dotted]
\end{tikzcd}\]
where $\overline{c}=\ple{c_{(X_1,x_1)},\dots,c_{(X_n,x_n)}}$ and symilarly $\overline{d}=\ple{d_{(X_1,x_1)},\dots,d_{(X_n,x_n)}}$.
The rectangle can be easily divided into two commutative squares: the lower one is clearly commutative by naturality of $\gamma$, while the upper one is commutative too since $\gamma_{\Pi X_a\times A}=\Pi\gamma_{X_a}\times\gamma_A$ and $\ple{\overline{d}\cdot !,\id{KA}}=(\Pi\gamma_{X_a}\times\id{KA})\ple{\overline{c}\cdot !,\id{KA}}$. So we get $\gamma:G\xrightarrow{\cdot}M$, as claimed.

At last, we show that it is indeed a 2-arrow: take any $[\ecc{(\alpha,\bar{X})},\mathcal{U}]\in\underrightarrow{P}(A)$ for some element $\alpha\in P(\prod_{a=1}^nX_a\times A)$ and $\bar{X}=\big((X_1,x_1),\dots,(X_n,x_n)\big)\in J$, we prove that in $R(GA)$
\begin{equation*}\mathfrak{g}_A[\ecc{(\alpha,\bar{X})},\mathcal{U}]\leq R(\gamma_A)\mathfrak{m}_A[\ecc{(\alpha,\bar{X})},\mathcal{U}].\end{equation*}
Using the same notation we used above for $\overline{c}$ and $\overline{d}$, we compute:
\begin{align*}\mathfrak{g}_A[\ecc{(\alpha,\bar{X})},\mathcal{U}]&=\mathfrak{g}_A[\ecc{(\alpha,\bar{X})},\emptyset]\\
&=R(\ple{\overline{c}\cdot!,\id{HA}})\mathfrak{h}_{\Pi X_a\times A}\alpha\\
&\leq R(\ple{\overline{c}\cdot!,\id{HA}})R(\gamma_{\Pi X_a\times A})\mathfrak{k}_{\Pi X_a\times A}\alpha\\
&= R(\gamma_A)R(\ple{\overline{d}\cdot !,\id{KA}})\mathfrak{k}_{\Pi X_a\times A}\alpha\\
&=R(\gamma_A)\mathfrak{m}_A[\ecc{(\alpha,\bar{X})},\mathcal{U}]\end{align*}
as claimed.

It is left to prove that the functor $\blank\circ(F,\mathfrak f)$ is a full functor between the coslice categories.

Suppose to have $\gamma:(G,\mathfrak{g})\to(M,\mathfrak m),\mu:(G,\mathfrak{g})\to(N,\mathfrak n)\in(G,\mathfrak{g})\downarrow\Dott_{\land,\top,\to,\exists}(\underrightarrow{P},R)$, for some $(M,\mathfrak m),(N,\mathfrak n):\underrightarrow{P}\to R$. Moreover, let $\delta:(M,\mathfrak m)(F,\mathfrak f)\to(N,\mathfrak n)(F,\mathfrak f)$ be a 2-arrow making the triangle on the right commute.
\[\begin{tikzcd}
	& {(G,\mathfrak g)} &&& {(H,\mathfrak h)} \\
	\\
	{(M,\mathfrak m)} && {(N,\mathfrak n)} & {(M,\mathfrak m)(F,\mathfrak f)} && {(N,\mathfrak n)(F,\mathfrak f)}
	\arrow["\gamma"', from=1-2, to=3-1]
	\arrow["\mu", from=1-2, to=3-3]
	\arrow["\delta"',dashed, from=3-1, to=3-3]
	\arrow["\gamma"', from=1-5, to=3-4]
	\arrow["\mu", from=1-5, to=3-6]
	\arrow["\delta"', from=3-4, to=3-6]
\end{tikzcd}\]

We prove that $\delta$ is also a 2-arrow between $(M,\mathfrak{m})$ and $(N,\mathfrak{n})$. Similarly to what we did before, define for any $(X,x)\in J$ the \ct{D}-arrows $d_{(X,x)}:=M\ecc{({\id{X},(X,x)})}:\tmn_\ct{D}\to MX$ and $e_{(X,x)}:=N\ecc{({\id{X},(X,x)})}:\tmn_\ct{D}\to NX$. Apply naturality of $\gamma:G\to M$ and $\mu:G\to N$ to the arrow $\ecc{({\id{X},(X,x)})}$ to obtain respectively $\gamma_Xc_{(X,x)}=d_{(X,x)}$ and $\mu_Xc_{(X,x)}=e_{(X,x)}$. However, since $\delta_X\cdot\gamma_X=\mu_X$, we get
\begin{equation}\label{eq:dd=e}\delta_X d_{(X,x)}=e_{(X,x)}.\end{equation}
Now fix a $\underline{\CC}$-arrow $\ecc{(f,\bar{X})}:A\dashrightarrow B$, for a $\CC$-arrow $f:\prod_{a=1}^n X_a\times A\to B$ and for a list $\bar{X}=\big((X_1,x_1),\dots,(X_n,x_n)\big)$ in $J$.
Moreover we write $\overline{d}=\ple{d_{(X_1,x_1)},\dots,d_{(X_n,x_n)}}$ and symilarly $\overline{e}=\ple{e_{(X_1,x_1)},\dots,e_{(X_n,x_n)}}$. Naturality of $\delta:M\to N$ means that the following square commutes:

\[\begin{tikzcd}
	MA & NA \\
	{\prod MX_a\times MA} & {\prod NX_a\times NA} \\
	MB & NB
	\arrow["{\delta_A}", from=1-1, to=1-2]
	\arrow["{\ple{\overline{d}\cdot !,\id{MA}}}"', from=1-1, to=2-1]
	\arrow["{MF(f)}"', from=2-1, to=3-1]
	\arrow["{\delta_B}", from=3-1, to=3-2]
	\arrow["{\ple{\overline{e}\cdot !,\id{NA}}}", from=1-2, to=2-2]
	\arrow["{NF(f)}", from=2-2, to=3-2]
	\arrow["{\delta_{\Pi X_a\times A}}", from=2-1, to=2-2,dotted]
\end{tikzcd}.\]
Commutativity of the lower square follows from naturality of $\delta:MF\to NF$, while the upper square commutes if and only if $\delta_{X_i}d_{(X_i,x_i)}=e_{(X_i,x_i)}$, but this follows from \eqref{eq:dd=e}. This concludes the proof.\end{proof}

%%%%%%%%%%%%%%%%%%%%%%%%%%%%%%%%%%%%%%%%%%%%%%%%%%%%%%%%
\subsection{Consistency of $\protect\underrightarrow{P}$}\label{sect:coher}
\begin{definition}
A doctrine $R:\ct{D}\op\to\Pos$ is \emph{consistent} if there exists a pair $a,b\in R(\tmn)$ such that $a\nleq b$. Moreover, $R$ is \emph{two-valued} if it is consistent and there exists a pair $a,b\in R(\tmn)$ such that $a\nleq b$ and for all $c\in R(\tmn)$ one has $a\leq c$ or $b\leq c$.
\end{definition}
\begin{center}
\fbox{\parbox{.5\textwidth}{\raggedright From now on, $P:\CC\op\to\Pos$ is a fixed bounded implicational existential consistent doctrine, with a small base category, unless otherwise specified.}}
\end{center}
Our goal is to show that the new doctrine $\underrightarrow{P}$ is consistent: we must be careful not to collapse fibers of $\underrightarrow{P}$ to the trivial poset.
\begin{lemma}
If $R:\ct{D}\op\to\Pos$ is a bounded doctrine. Then the following are equivalent:
\begin{enumerate}[label=\roman*.]
\item\label{i:1_const} $R(\tmn)\neq\{\star\}$;
\item\label{i:2_const} $\top_{\tmn}\nleq\bot_{\tmn}$;
\item\label{i:3_const} $R$ is consistent;
\item\label{i:4_const} $R$ is two-valued.
\end{enumerate}
\end{lemma}
\begin{proof}
$\eqref{i:1_const}\Rightarrow\eqref{i:2_const}$ If $\top_{\tmn}\leq\bot_{\tmn}$, then for all $a\in R(\tmn)$ we have $\bot_{\tmn}\leq a\leq\top_{\tmn}\leq\bot_{\tmn}$, hence for all $a$ we have $a=\bot_{\tmn}$, hence $R(\tmn)$ is a singleton.

$\eqref{i:2_const}\Rightarrow\eqref{i:1_const}$ Trivial.

$\eqref{i:4_const}\Rightarrow\eqref{i:3_const}$ By definition.

$\eqref{i:3_const}\Rightarrow\eqref{i:2_const}$ If $\top_{\tmn}\leq\bot_{\tmn}$, then for all $a,b\in R(\tmn)$ we have $a\leq\top_{\tmn}\leq\bot_{\tmn}\leq b$, hence $R$ cannot be consistent.

$\eqref{i:2_const}\Rightarrow\eqref{i:4_const}$ Take $a=\top_{\tmn}$ and $b=\bot_{\tmn}$ and observe that for all $c\in R(\tmn)$ we have $b=\bot_{\tmn}\leq c$.
\end{proof}
\begin{remark}
Let $R$ be a bounded existential doctrine. If $R$ is consistent and rich, then each of its fibers is non-trivial---i.e.\ it is not a singleton.
Indeed, suppose $R(D)=\{\bot_D=\top_D\}$ for some object $D$ in the base category. Then there exists a $d:\tmn\to D$ such that $\exists^D_\tmn\top_D=R(d)\top_D=\exists^D_\tmn\bot_D=R(d)\bot_D$, in particular $\top_\tmn=\bot_\tmn$, which is absurd since $R$ is consistent.
\end{remark}
We want to find the conditions making $\underrightarrow{P}$ a consistent doctrine as well. Using the lemma above, we want $[\ecc{(\top,\emptyset)},\emptyset]\nleq[\ecc{(\bot,\emptyset)},\emptyset]$ in $\underrightarrow{P}(\tmn)$.

However, $[\ecc{(\top,\emptyset)},\emptyset]\leq[\ecc{(\bot,\emptyset)},\emptyset]$ if and only if there exists $\mathcal{U}=\{(B_1,j_1),\dots,(B_n,j_n)\}\in I$ such that
\begin{equation}\label{eq:abs}
\bigwedge_{i=1}^q\ecc{(\psi^{B_i}_{j_i},\bar{X}^{(B_i,j_i)}_\star)}\leq\ecc{(\bot,\emptyset)} \text{ in }\underline{P}(\tmn).
\end{equation}
We want to prove this to be a contradiction by induction on $q$.
If $q=0$, we get $\ecc{(\top,\emptyset)}\leq\ecc{(\bot,\emptyset)}$, i.e.\ there exists $\bar{Y}=\big((Y_1,y_1),\dots,(Y_m,y_m)\big)\in J$ such that in $P(\prod_{b=1}^mY_b)$
\begin{equation*}P(!_{\Pi Y_b})(\top)\leq P(!_{\Pi Y_b})(\bot),\end{equation*}
i.e.\ $\top\leq\bot$ in $P(\prod_{b=1}^mY_b)$. It follows from this that a stronger requirement on $P$ is needed: not only $P(\tmn)$ must not be a singleton, but also each $P(A)$ must not be a singleton, for every $A\in\text{ob}\CC$. Otherwise, $\underline{P}(\tmn)$ is trivial, hence also $\underrightarrow{P}(\tmn)$ is trivial.
So, from now on we suppose that $P$ has bottom element and has each $P(A)$ non-trivial.
\begin{center}
\fbox{\parbox{.5\textwidth}{\raggedright From now on, $P:\CC\op\to\Pos$ is a fixed bounded implicational existential doctrine, with non-trivial fibers, and with a small base category, unless otherwise specified.}}
\end{center}
With this additional assumption, we get a contradiction in the case $q=0$.
Suppose now $\eqref{eq:abs}$ to be a contradiction for $q$; we will take the rest of the section to understanding when it is the case that also $q+1$ gives a contradiction.
Suppose
\begin{equation*}\bigwedge_{i=1}^{q+1}\ecc{(\psi^{B_i}_{j_i},\bar{X}^{(B_i,j_i)}_\star)}\leq\ecc{(\bot,\emptyset)}\text{ in }\underline{P}(\tmn),\end{equation*}
\begin{equation*}\text{i.e.\ }\bigwedge_{i=1}^q\ecc{(\psi^{B_i}_{j_i},\bar{X}^{(B_i,j_i)}_\star)}\land\ecc{(\psi^{B_{q+1}}_{j_{q+1}},\bar{X}^{(B_{q+1},j_{q+1})}_\star)}\leq\ecc{(\bot,\emptyset)}\text{ in }\underline{P}(\tmn).\end{equation*}
For the sake of simplicity we write $\psi$ instead of $\psi^{B_{q+1}}_{j_{q+1}}$. Moreover, up to a permutation of the indices $i=1,\dots,q+1$, we can suppose that $d^{B_{q+1}}_{j_{q+1}}\geq d^{B_{i}}_{j_{i}}$ for $i=1,\dots,q$.

Compute $\bigwedge_{i=1}^q\ecc{(\psi^{B_i}_{j_i},\bar{X}^{(B_i,j_i)}_\star)}$ as the class of some $\theta$ paired with a list $\bar{T}$ of $J$ with entries in
\begin{equation*}\mathcal{F}\coloneqq\bigcup_{i=1}^q\left\{\left(X_a^{(B_i,j_i)},x_a^{(B_i,j_i)}\right)\right\}_{a=1}^{n^{(B_i,j_i)}}\cup\bigcup_{i=1}^q\left\{\left(B_i,d^{B_i}_{j_i}\right)\right\}.\end{equation*}
Then call
\begin{equation*}\mathcal{G}\coloneqq\left\{\left(X_a^{(B_{q+1},j_{q+1})},x_a^{(B_{q+1},j_{q+1})}\right)\right\}_{a=1}^{n^{(B_{q+1},j_{q+1})}};\end{equation*}
and rename the pairs:
\begin{align*}&\mathcal{F}\cap\mathcal{G}=\left\{\left(Z_b,z_b\right)\right\}_{b=1}^{\overline{b}},\\
&\mathcal{F}\smallsetminus(\mathcal{F}\cap\mathcal{G})=\left\{\left(W_c,w_c\right)\right\}_{c=1}^{\overline{c}},\\
&\mathcal{G}\smallsetminus(\mathcal{F}\cap\mathcal{G})=\left\{\left(V_e,v_e\right)\right\}_{e=1}^{\overline{e}}.\end{align*}
Observe that $(B_{q+1},d^{B_{q+1}}_{j_{q+1}})\notin\mathcal{G}\cup\mathcal{F}$: it does not belong to $\mathcal{G}$ by definition of $d^{B_{q+1}}_{j_{q+1}}$, it is different from all the pairs $(B_i,d^{B_i}_{j_i})$ for $i=1,\dots,q$ since we are taking the conjunction of $q+1$ formulae by assumption, and it is different from all the pairs $(X_a^{(B_i,j_i)},x_a^{(B_i,j_i)})$ for $i=1,\dots,q$ and $a=1,\dots,{n^{(B_i,j_i)}}$ since $d^{B_{q+1}}_{j_{q+1}}\geq d^{B_{i}}_{j_{i}}>x_a^{(B_i,j_i)}$ for $i=1,\dots,q$ and $a=1,\dots,n^{(B_i,j_i)}$.

From now on, we write $(B,d)$ instead of $(B_{q+1},d^{B_{q+1}}_{j_{q+1}})$ in order to lighten the notation. We compute $\ecc{(\theta,\bar{T})}\land\ecc{(\psi,\bar{X}^{(B,j_{q+1})}_\star)}$ as the equivalence class of an element in
\begin{equation*}\overunderbraces{&\br{2}{\mathcal{F}}}{P(&\prod W_c\times&\prod Z_b&\times\prod V_e&\times B)}{&&\br{2}{\mathcal{G}}}\end{equation*}
paired with the list
\begin{equation*}\bar{S}=\big(\dots,(W_c,w_c),\dots,(Z_b,z_b),\dots,(V_e,v_e),\dots,(B,d)\big).\end{equation*}
We can assume $\theta\in P(\Pi{W_c}\times\Pi{Z_b})$ and
\begin{equation*}\ecc{(\psi',({\dots,(Z_b,z_b),\dots,(V_e,v_e),\dots,{(B,d)})})}=\ecc{(\psi,\bar{X}^{(B,j_{q+1})}_\star)}\end{equation*}
where $\psi'\in P(\Pi{Z_b}\times\Pi{V_e}\times B)$ is a reindexing along a suitable permutation of $\psi$. We can do so by recalling that
\begin{equation*}\mathcal{G}=\left\{\left(X_a^{(B_{q+1},j_{q+1})},x_a^{(B_{q+1},j_{q+1})}\right)\right\}_{a=1}^{n^{(B_{q+1},j_{q+1})}}=\left\{\left(Z_b,z_b\right)\right\}_{b=1}^{\overline{b}}\cup\left\{\left(V_e,v_e\right)\right\}_{e=1}^{\overline{e}}.\end{equation*}
Then
\begin{equation*}\ecc{(\theta,\bar{T})}\land\ecc{(\psi,\bar{X}^{(B,j_{q+1})}_\star)}=\ecc{(P(\ple{\pr1,\pr2})\theta\land P(\ple{\pr2,\pr3,\pr4})\psi',\bar{S})}\in\underline{P}(\tmn).\end{equation*}
Then $\ecc{(\theta,\bar{T})}\land\ecc{(\psi,\bar{X}^{(B,j_{q+1})}_\star)}\leq\ecc{(\bot,\emptyset)}$ if and only if there exists a set $\left\{\left(Y_h,y_h\right)\right\}_{h=1}^{\overline{h}}$, disjoint from $\mathcal{F}\cup\mathcal{G}\cup\left\{\left(B,d\right)\right\}$ such that in $P(\Pi W_c\times\Pi Z_b\times\Pi V_e\times B\times \Pi Y_h)$ one has
\begin{equation*}P(\ple{\pr1,\pr2})\theta\land P(\ple{\pr2,\pr3,\pr4})\psi'\leq\bot\end{equation*}
if and only if in $P(\Pi W_c\times\Pi Z_b\times\Pi V_e\times\Pi Y_h\times B)$ one has
\begin{equation*}P(\ple{\pr1,\pr2})\theta\land P(\ple{\pr2,\pr3,\pr5})\psi'\leq\bot=P(\ple{\pr1,\pr2,\pr3,\pr4})\bot\end{equation*}
if and only if, using $\exists^B_{\Pi W\times\Pi Z\times\Pi V\times\Pi Y}\dashv P(\ple{\pr1,\pr2,\pr3,\pr4})$, in $P(\Pi W_c\times\Pi Z_b\times\Pi V_e\times\Pi Y_h)$ one has
\begin{equation*}\exists^B_{\Pi W\times\Pi Z\times\Pi V\times\Pi Y}\left(P(\ple{\pr1,\pr2})\theta\land P(\ple{\pr2,\pr3,\pr5})\psi'\right)\leq\bot;\end{equation*}
then use Frobenius reciprocity, and note that $P(\ple{\pr1,\pr2})=P(\ple{\pr1,\pr2,\pr3,\pr4})P(\ple{\pr1,\pr2})$ as the composition of the projections from $\Pi W_c\times\Pi Z_b\times\Pi V_e\times\Pi Y_h\times B$ to $\Pi W_c\times\Pi Z_b\times\Pi V_e\times\Pi Y_h$ to $\Pi W_c\times\Pi Z_b$ in order to get
\begin{equation*}\exists^B_{\Pi W\times\Pi Z\times\Pi V\times\Pi Y}P(\ple{\pr2,\pr3,\pr5})\psi'\land P(\ple{\pr1,\pr2})\theta\leq\bot.\end{equation*}
\begin{claim}\label{claim:wit}
$\top\leq\exists^B_{\Pi W\times\Pi Z\times\Pi V\times\Pi Y}P(\ple{\pr2,\pr3,\pr5})\psi'.$
\end{claim}
If this is the case, then we get $P(\ple{\pr1,\pr2})\theta\leq\bot$, hence we have
\begin{align*}[P(\ple{\pr1,\pr2})\theta,(\dots,(W_c,w_c),\dots,(Z_b,z_b),\dots,(V_e,v_e),\dots,(Y_h,y_h),\dots)]&=\ecc{(\theta,\bar{T})}\\
&=\bigwedge_{i=1}^q\ecc{(\psi^{B_i}_{j_i},\bar{X}^{(B_i,j_i)}_\star)}\\
&\leq\ecc{(\bot,\emptyset)}
\end{align*}
which is $\eqref{eq:abs}$, a contradiction for our inductive hypothesis.

Now recall the definition of
\begin{equation*}\psi=\psi^{B_{q+1}}_{j_{q+1}}=P(\pr1)\exists^{B}_{\Pi X_a^{(B_{q+1},j_{q+1})}}\varphi^{B_{q+1}}_{j_{q+1}}\longrightarrow\varphi^{B_{q+1}}_{j_{q+1}}.\end{equation*}

Using the same permutation that defines $\psi'$ and Beck-Chevalley condition, \Cref{claim:wit} becomes equivalent to
\begin{equation*}\top\leq\exists^B_{\Pi W\times\Pi X\times\Pi Y}P(\ple{\pr2,\pr4})\psi.\end{equation*}
We then compute
\begin{align*}\exists^B_{\Pi W\times\Pi X\times\Pi Y}P(\ple{\pr2,\pr4})\psi&=\exists^B_{\Pi W\times\Pi X\times\Pi Y}P(\ple{\pr2,\pr3,\pr4})P(\ple{\pr1,\pr3})\psi\\
&=P(\ple{\pr2,\pr3})\exists^B_{\Pi X\times\Pi Y}P(\ple{\pr1,\pr3})\psi,\end{align*}
so it is sufficient to prove $\top\leq\exists^B_{\Pi X\times\Pi Y}P(\ple{\pr1,\pr3})\psi$. Substituting $\psi$ with its definition, omitting superscripts and subscripts of $\varphi^{B_{q+1}}_{j_{q+1}}$ and $X_a^{(B_{q+1},j_{q+1})}$ we want to prove the following
\begin{claim}\label{claim}
$\top\leq\exists^B_{\Pi X\times\Pi Y}\Big(P(\pr1)\exists^B_{\Pi X}\varphi\longrightarrow P(\ple{\pr1,\pr3})\varphi\Big)$ in $P(\Pi X\times\Pi Y)$.
\end{claim}
For the proof of this claim we suppose for now that $P$ is Boolean. We show later in that we can remove this additional assumption and still prove consistency of $\underrightarrow{P}$.
\begin{center}
\fbox{\parbox{.5\textwidth}{\raggedright From now on, $P:\CC\op\to\Pos$ is a fixed Boolean existential doctrine, with non-trivial fibers, and with a small base category, unless otherwise specified.}}
\end{center}
The doctrine $P$ is Boolean, so we can suppose that
\begin{equation*}\top=\left(\exists^B_{\Pi X\times\Pi Y}P(\ple{\pr1,\pr3})\varphi\right)\lor\left(\lnot\exists^B_{\Pi X\times\Pi Y}P(\ple{\pr1,\pr3})\varphi\right).\end{equation*}
Then, use Beck-Chevalley condition to write $P(\ple{\pr1,\pr2})\exists^B_{\Pi X\times\Pi Y}P(\ple{\pr1,\pr3})\varphi$ instead of $ P(\pr1)\exists^B_{\Pi X}\varphi$.

Hence now it is sufficient to prove both
\begin{equation}\label{eq:a}\exists^B_{\Pi X\times\Pi Y}P(\ple{\pr1,\pr3})\varphi\leq\exists^B_{\Pi X\times\Pi Y}\big(P(\ple{\pr1,\pr2})\exists^B_{\Pi X\times\Pi Y}P(\ple{\pr1,\pr3})\varphi\to P(\ple{\pr1,\pr3})\varphi\big)
\end{equation}
and
\begin{equation}\label{eq:b}\lnot\exists^B_{\Pi X\times\Pi Y}P(\ple{\pr1,\pr3})\varphi\leq\exists^B_{\Pi X\times\Pi Y}\big(P(\ple{\pr1,\pr2})\exists^B_{\Pi X\times\Pi Y}P(\ple{\pr1,\pr3})\varphi\to P(\ple{\pr1,\pr3})\varphi\big),
\end{equation}
so that the \Cref{claim} follows by taking the join of $\eqref{eq:a}$ and $\eqref{eq:b}$.

To prove $\eqref{eq:a}$ it is sufficient to see that
\begin{equation}\label{eq:e}P(\ple{\pr1,\pr3})\varphi\leq P(\ple{\pr1,\pr2})\exists^B_{\Pi X\times\Pi Y}P(\ple{\pr1,\pr3})\varphi\to P(\ple{\pr1,\pr3})\varphi\end{equation}
if and only if
\begin{equation*}P(\ple{\pr1,\pr3})\varphi\land P(\ple{\pr1,\pr2})\exists^B_{\Pi X\times\Pi Y}P(\ple{\pr1,\pr3})\varphi\leq P(\ple{\pr1,\pr3})\varphi,\end{equation*}
which is trivially verified; then get $\eqref{eq:a}$ by applying $\exists^B_{\Pi X\times\Pi Y}$ to both sides of $\eqref{eq:e}$.

Now write $\varphi'$ instead of $P(\ple{\pr1,\pr3})\varphi$, and we prove $\eqref{eq:b}$ by showing first
\begin{equation}\label{eq:c}\lnot\exists^B_{\Pi X\times\Pi Y}\varphi'\leq\exists^B_{\Pi X\times\Pi Y}P(\ple{\pr1,\pr2})\lnot\exists^B_{\Pi X\times\Pi Y}\varphi'
\end{equation}
and then
\begin{equation}\label{eq:d}\exists^B_{\Pi X\times\Pi Y}P(\ple{\pr1,\pr2})\lnot\exists^B_{\Pi X\times\Pi Y}\varphi'\leq\exists^B_{\Pi X\times\Pi Y}\big(P(\ple{\pr1,\pr2})\exists^B_{\Pi X\times\Pi Y}\varphi'\to \varphi'\big).
\end{equation}
The proof of $\eqref{eq:d}$ is quite immediate: observe that in general in a Boolean algebra we have $\lnot\alpha\leq\alpha\to\beta$---if and only if $\bot=\lnot\alpha\land\alpha\leq\beta$---, hence take $\alpha=P(\ple{\pr1,\pr2})\exists^B_{\Pi X\times\Pi Y}\varphi'$, $\beta=\varphi'$ and apply $\exists^B_{\Pi X\times\Pi Y}$ to get $\eqref{eq:d}$.

To conclude, we show that given $\gamma\in P(\Pi X\times\Pi Y)$ we have $\gamma\leq\exists^B_{\Pi X\times\Pi Y}P(\ple{\pr1,\pr2})\gamma$, so that we get $\eqref{eq:c}$ by taking $\gamma=\lnot\exists^B_{\Pi X\times\Pi Y}\varphi'$.
To do so, we need to look at the set $\left\{\left(Y_h,y_h\right)\right\}_{h=1}^{\overline{h}}$ defined above. We can suppose that one of the $Y_h$'s is actually the object $B$---in which case the associated ordinal $y_h$ is different from $d$. If this is not the case, we add the element $(B,k)$ to $\left\{\left(Y_h,y_h\right)\right\}_{h=1}^{\overline{h}}$ for some ordinal $k\in\Lambda$ that does not appear in any second entry of $(B,\lambda)$ belonging to $\mathcal{F}\cup\mathcal{G}\cup\left\{\left(B,d\right)\right\}$---note that such new pair does not belong to $\left\{\left(Y_h,y_h\right)\right\}_{h=1}^{\overline{h}}$: if it did, we did not have to add it to such set. So, up to a permutation of indices and up to a change of $\overline{h}$ with $\overline{h}+1$, we can suppose that in the set $\left\{\left(Y_h,y_h\right)\right\}_{h=1}^{\overline{h}}$ we have $Y_{\overline{h}}=B$.

So now we look at the adjunction:
\[\begin{tikzcd}[row sep=5pt]
	{P(\prod_{a=1}^{\overline{b}+\overline{e}}X_a\times\prod_{h=1}^{\overline{h}}Y_h)} && {P(\prod_{a=1}^{\overline{b}+\overline{e}}X_a\times\prod_{h=1}^{\overline{h}}Y_h\times B)} \\
	\\
	\\
	{P(\prod_{a=1}^{\overline{b}+\overline{e}}X_a\times\prod_{h=1}^{\overline{h}-1}Y_h\times B)} && {P(\prod_{a=1}^{\overline{b}+\overline{e}}X_a\times\prod_{h=1}^{\overline{h}-1}Y_h\times B\times B)}
	\arrow[""{name=0, anchor=center, inner sep=0}, "{P(\ple{\pr1,\pr2})}", shift left=2, from=1-1, to=1-3]
	\arrow[""{name=1, anchor=center, inner sep=0}, "{\exists^B_{\Pi X\times\Pi Y}}", shift left=2, from=1-3, to=1-1]
	\arrow["{=}"{marking}, draw=none, from=1-1, to=4-1]
	\arrow[draw=none, from=1-3, to=4-3]
	\arrow[""{name=2, anchor=center, inner sep=0}, "{\exists^B_{\Pi X\times\Pi Y}}", shift left=1, from=4-3, to=4-1]
	\arrow[""{name=3, anchor=center, inner sep=0}, "{P(\ple{\pr1,\pr2,\pr3})}", shift left=2, from=4-1, to=4-3]
	\arrow["{=}"{marking}, draw=none, from=1-3, to=4-3]
	\arrow["{P(\ple{\pr1,\pr2,\pr3,\pr3})}", shift left=4, bend left=10, from=4-3, to=4-1]
	\arrow["\dashv"{marking}, Rightarrow, draw=none, from=1, to=0]
	\arrow["\dashv"{marking}, Rightarrow, draw=none, from=2, to=3]
\end{tikzcd}\]
so if we look at our claim in the lower part of the diagram we want that given $\gamma\in P(\Pi X\times\Pi Y)$, then $\gamma\leq\exists^B_{\Pi X\times\Pi Y}P(\ple{\pr1,\pr2,\pr3})\gamma$. Now, consider the unit of the adjunction at the level $P(\ple{\pr1,\pr2,\pr3})\gamma$, hence
\begin{equation*}P(\ple{\pr1,\pr2,\pr3})\gamma\leq P(\ple{\pr1,\pr2,\pr3}) \exists^B_{\Pi X\times\Pi Y}P(\ple{\pr1,\pr2,\pr3})\gamma;\end{equation*}
now, apply $P(\ple{\pr1,\pr2,\pr3,\pr3})$, so we get exactly $\gamma\leq\exists^B_{\Pi X\times\Pi Y}P(\ple{\pr1,\pr2,\pr3})\gamma$ as claimed.

In particular, we proved the following:
\begin{proposition}\label{prop:coher_bool}
Let $P:\CC\op\to\Pos$ be a Boolean existential doctrine such that each fiber is non-trivial, and the base category $\CC$ is small, then the doctrine $\underrightarrow{P}$ is consistent.\end{proposition}

As hinted before, we will soon slightly weaken the assumption that $P$ is Boolean, and prove the consistency of $\underrightarrow{P}$ anyway.
\begin{remark}\label{sub:bool_compl}
Given a primary doctrine $P:\CC\op\to\Pos$, a \emph{topology on $P$} is a primary doctrine endomorphism of the form $(\id{\CC},j):P\to P$, where $j$ is such that for every object $X$ in $\CC$ and every $\alpha\in P(X)$ we have $\alpha\leq j_X(\alpha)$ and $j_Xj_X(\alpha)=j_X(\alpha)$. Given a topology $(\id{\CC},j)$ on $P$, the \emph{doctrine of $j$-closed element of $P$} is the primary doctrine $P_j:\CC\op\to\Pos$ where for every object $X$ we have $P_j(X)=\{\alpha\in P(X)\mid \alpha=j_X(\alpha)\}$, and $P_j$ on arrows acts as the restriction of $P$---see \cite{MaPaRo}, Definitions 3.2 and 3.3. If additionally $P$ is elementary, then so is $P_j$; if $P$ is existential, then so is $P_j$; if $P$ is implicational, then so is $P_j$---see \cite{MaPaRo}, Proposition 3.6.

As a particular case of this, given a bounded implicational doctrine $P:\CC\op\to\Pos$, we can define a topology $(\id{\CC},\lnot\lnot):P\to P$ on $P$. This allows us to define a Boolean doctrine $P_{\lnot\lnot}\colon\CC\op\to\Pos$ and a bounded implicational morphism $(\id{\CC},\lnot\lnot):P\to P_{\lnot\lnot}$. Moreover, if $P$ is existential (resp.\ elementary), then also $P_{\lnot\lnot}$ and $(\id{\CC},\lnot\lnot):P\to P_{\lnot\lnot}$ are existential (resp.\ elementary).
\end{remark}
We can now prove an analogue of \Cref{prop:coher_bool} where we suppose fibers to be bounded implicative inf-semilattices instead of Boolean algebras.
\begin{proposition}\label{prop:coher}
Let $P:\CC\op\to\Pos$ be a bounded existential implicational doctrine such that each fiber is non-trivial, and the base category $\CC$ is small, then the doctrine $\underrightarrow{P}$ is consistent.
\end{proposition}
\begin{proof}
We start from $P$, and we build the Boolean doctrine $P_{\lnot\lnot}:\CC^{\op}\to\Pos$ as in \Cref{sub:bool_compl}. We have the following commutative diagram:
\[\begin{tikzcd}
	P & {P_{\lnot\lnot}} \\
	{\underrightarrow{P}} & {\underrightarrow{P_{\lnot\lnot}}}
	\arrow[from=1-1, to=2-1]
	\arrow[from=1-1, to=1-2]
	\arrow[from=1-2, to=2-2]
	\arrow[from=2-1, to=2-2,dashed]
\end{tikzcd}.\]
The map $P\to\underrightarrow{P}$ is $(F_P,\mathfrak{f}_P)$ defined in \Cref{rmk:recap}, the map $P\to{P_{\lnot\lnot}}$ is $(\id{},\lnot\lnot)$ as in \Cref{sub:bool_compl}, the map ${P_{\lnot\lnot}}\to\underrightarrow{P_{\lnot\lnot}}$ is $(F_{P_{\lnot\lnot}},\mathfrak{f}_{P_{\lnot\lnot}})$ again defined in \Cref{rmk:recap} corresponding to the construction applied to the doctrine $P_{\lnot\lnot}$. Then, use the weak universal property of $P\to\underrightarrow{P}$---see \Cref{thm:weak_univ_prop}: the doctrine $\underrightarrow{P_{\lnot\lnot}}$ is existential, implicational, rich, and the composition of the upper morphism with the one on the right preserves the bounded implicational existential structure because both arrows do; so there exists a map $\underrightarrow{P}\to\underrightarrow{P_{\lnot\lnot}}$ closing the square above and endowed with the structure just mentioned.
Note that all $P_{\lnot\lnot}(X)$ are non-trivial, since the top and bottom elements are computed in $P(X)$, in which these are distinct elements by assumption.
In particular, since $P_{\lnot\lnot}$ is also Boolean, it follows from \Cref{prop:coher_bool} that $\underrightarrow{P_{\lnot\lnot}}$ is consistent. But then, since there exists a map $\underrightarrow{P}\to\underrightarrow{P_{\lnot\lnot}}$ preserving, among others, $\top$ and $\bot$, if $\underrightarrow{P_{\lnot\lnot}}$ is consistent, $\underrightarrow{P}$ must be consistent too.\end{proof}

%%%%%%%%%%%%%%%%%%%%%%%%%%%%%%%%%%%%%%%%%%%%%%%%%%%%%%%%%%%%%%%%%%%
%%%%%%%%%%%%%%%%%%%%%%%%%%%%%%%%%%%%%%%%%%%%%%%%%%%%%%%%%%%%%%%%%%%
\section{A model of a rich doctrine}\label{sect:mod_rch}
The goal of this section is to build a model in $\pws$ of a bounded consistent existential implicational rich doctrine $P$, preserving the bounded existential implicational structure. To achieve this result, we first need a small detour about the quotient of a doctrine over a filter $\nabla\subseteq P(\tmn)$ in the fiber of the terminal object: this notion allows us to define yet a new doctrine $P/\nabla$, with a morphism from $P$ that preserves the bounded existential implicational structure---see \Cref{sub:quot}. In the particular case where $\nabla$ is an ultrafilter, whose existence is granted by the fact that $P$ is consistent, we find a bounded implicational existential model $P/\nabla \to\pws$---see \Cref{prop:mod}.

The existence of this model, together with the results of the previous section, gives the last ingredient for proving Henkin's Theorem for doctrines---see \Cref{thm:mod_ex}.

All of these results are then adapted to the case when the starting doctrine has equality, proving that all the doctrines and morphisms involved are also elementary---see \Cref{prop:mod_eq} and \Cref{thm:mod_ex_eq}.

%%%%%%%%%%%%%%%%%%%%%%%%%%%%%%%%%%%%%%%%%%%%%%%%%%%%%%%%%%%%%%%%%%%
\subsection{The quotient of a doctrine over a filter}\label{sub:quot}\label{subsub:filter}
Filters play a significant role in lattice theory, particularly in the study of Boolean algebras. We present here some results concerning filters and ultrafilters in bounded implicative inf-semilattices. While these proofs are already established in the context of Boolean algebras---see for example \cite{monk} or \cite{BS}---, we demonstrate their adaptability in this weaker framework.

Then, for a given primary doctrine $P$, we define the quotient of the doctrine over a filter in the fiber of the terminal object, and prove that the quotient map preserves many properties of $P$ itself.
\begin{definition}
Let $A$ be an inf-semilattice. A subset $\nabla\subseteq A$ is a \emph{filter} if the following properties hold:
\begin{itemize}
\item $\top\in\nabla$;
\item if $a\in\nabla$ and $a\leq b$, then $b\in\nabla$;
\item if $a,b\in\nabla$, then $a\land b\in\nabla$.
\end{itemize}
A filter $\nabla$ is \emph{proper} if $\nabla\neq A$
\end{definition}
\begin{remark}
In a bounded inf-semilattice, a filter $\nabla$ is proper if and only if $\bot\notin\nabla$.
\end{remark}
\begin{definition}
Let $A$ be a bounded implicative inf-semilattice and $\nabla\subseteq A$ a filter.
\begin{itemize}
\item $\nabla$ is an \emph{ultrafilter} if for all $a\in A$, either $a\in\nabla$ or $\lnot a\in\nabla$, where $\lnot a\coloneqq a\to\bot$.
\item $\nabla$ is a \emph{maximal filter} if it is maximal with respect to the inclusion, meaning that $\nabla\neq A$ and, whenever $\nabla\subsetneqq\nabla'$ where $\nabla'$ is a filter, then $\nabla'=A$.
\end{itemize}
\end{definition}
\begin{lemma}
Let $A$ be an inf-semilattice and $E\subseteq A$. Consider the set
\begin{equation*}F=\{y\in A\mid \text{ there exist }x_1,\dots,x_n\in E\text{ such that }x_1\land\dots\land x_n\leq y\}\cup\{\top\},\end{equation*}
Then $\langle E\rangle=F$, where $\langle E\rangle$ is the filter generated by $E$.
\end{lemma}
\begin{proof}
First of all, observe that $F$ is a filter:
\begin{itemize}
\item $\top\in F$;
\item let $y\in F$ and $z\in A$, $y\leq z$. If $y=\top$, then $z=\top\in F$. Otherwise, take $x_1,\dots,x_n\in E$ such that $x_1\land\dots\land x_n\leq y\leq z$, then also $z\in F$;
\item take $y,z\in F$. If $y=\top$ then $y\land z=z\in F$; similarly if $z=\top$. Otherwise $x_1\land\dots\land x_n\leq y$, $w_1\land\dots\land w_m\leq z$ with $x_1,\dots, x_n,w_1,\dots, w_m\in E$; then $x_1\land\dots\land x_n\land w_1\land\dots\land w_m\leq y\land z$, so that $y\land z\in F$.
\end{itemize}
Then $E\subseteq F$: take $x\in E$, since $x\leq x$, we have $x\in F$. In particular $\langle E\rangle\subseteq F$. To conclude, take $y\in F$. If $y=\top$, then $y\in\langle E\rangle$; otherwise, take $x_1\land\dots\land x_n\leq y$ for some $x_1,\dots,x_n\in E$. Any filter $G\supseteq E$ is such that $x_1\land\dots\land x_n\in G$ and since $x_1\land\dots\land x_n\leq y$, also $y\in G$. Hence $y\in\langle E\rangle$, as claimed.\end{proof}
\begin{lemma}
Let $A$ be a bounded implicative inf-semilattice and $\nabla\subseteq A$ a filter. Then $\nabla$ is a maximal filter if and only if $\nabla$ is an ultrafilter.
\end{lemma}
\begin{proof}
Suppose $\nabla$ is an ultrafilter. Since $\top\in\nabla$, then $\nabla\reflectbox{$\notin$}\lnot\top=\bot$, so $\nabla\neq A$. So take another filter $\nabla\subsetneqq\nabla'$, in particular there exists $y\in\nabla'$ such that $y\notin\nabla$. By assumption $y\to\bot\in\nabla$ and also $y\to\bot\in\nabla'$. Then, since $y\land(y\to\bot)\leq\bot$, $\bot\in\nabla'$, so that $\nabla'=A$.
For the converse, suppose $\nabla$ is a maximal filter. In particular, given $x\in A$, it cannot be the case that both $x,x\to\bot\in\nabla$---otherwise we would have also $\bot\in\nabla$, which would give $\nabla=A$. Suppose that $x\notin\nabla$, we claim that $\lnot x=x\to\bot\in\nabla$. Consider $E=\nabla\cup\{x\}$ and take $\langle E\rangle$. Clearly $\langle E\rangle\supsetneqq\nabla$, since $x\in E$ but $x\notin\nabla$. Hence by assumption $\langle E\rangle=A$. In particular $\lnot x\in A=\langle E\rangle$. If $\lnot x=\top$, then we have $\lnot x\in\nabla$. Otherwise there exist $x_1,\dots,x_n\in\nabla\cup\{x\}$ such that $x_1\land\dots\land x_n\leq \lnot x$. Now, if every $x_i$'s belong to the filter $\nabla$, we get $\lnot x\in\nabla$. Instead, if some $x_i$'s are actually $x$, we can rewrite the inequality as $x\land y\leq\lnot x$ for some $y\in\nabla$. But $x\land y\leq x\to\bot$ if and only if $x\land y\leq\bot$ if and only if $y\leq x\to\bot$, hence again $\lnot x\in\nabla$, as claimed.\end{proof}
\begin{lemma}\label{lemma:ultraf}
Given a proper filter $\nabla$ of a bounded implicative inf-semilattice, there exists an ultrafilter $U\supseteq\nabla$.
\end{lemma}
\begin{proof}
We use Zorn's Lemma. Take $\mathscr{F}$ the set of all proper filters that contain $\nabla$, ordered by inclusion. Clearly $\nabla\in\mathscr{F}$. The upper bound of a chain $\nabla\subseteq\nabla_1\subseteq\dots\subseteq\nabla_n\dots$ is given by the union $\cup_{i\in\mathbb{N}}\nabla_i$. So let $U$ be a maximal element in $\mathscr{F}$. This is a maximal filter: let $W$ be a proper filter containing $U$, in particular it contains $\nabla$, so $W=U$.\end{proof}
\begin{remark}
Observe that \Cref{lemma:ultraf} is a straightforward generalization of the same lemma for Boolean algebras. However, the analogous lemma stating that in a Boolean algebra, if $a\nleq b$ there exists an ultrafilter containing $a$ and not $b$, does not hold in the case of bounded inf-semilattices. Indeed, consider the ordered set $\{0<\frac{1}{2}<1\}$: its only ultrafilter is $\{\frac{1}{2},1\}$, but $1\nleq\frac{1}{2}$.
\end{remark}
Let $P:\CC\op\to\Pos$ be a primary doctrine and $\nabla\subseteq P(\tmn)$ be a filter in the fiber of the terminal object $\tmn$. Define, in each $X\in\text{ob}\CC$ the following preorder: $\alpha\sqsubseteq_\nabla\beta$ if and only if there exists a $\theta\in\nabla$ such that $P(!_X)\theta\land\alpha\leq\beta$ in $P(X)$.

Define a new doctrine $P/\nabla:\CC\op\to\Pos$ as follows: for each object $X$, $P/\nabla(X)$ is the poset reflection of the preorder $\sqsubseteq_\nabla$. In particular we have $[\alpha]=[\beta]$ if and only if there exists $\theta\in\nabla$ such that $P(!_X)\theta\land\alpha= P(!_X)\theta\land\beta$.
For every $\CC$-arrow $f:X\to Y$ it is easily shown that $P/\nabla(f)[\alpha]\coloneqq[P(f)\alpha]$ for $[\alpha]\in P/\nabla(X)$ is well-defined.

Note that the quotient map of each $P(X)$ is a monotone function.
Call for each object $X$, $\mathfrak{q}_X$ the quotient map: $\mathfrak{q}_X(\alpha)=[\alpha]\in P/\nabla(X)$ for a given $\alpha\in P(X)$; then $(\id{\CC},\mathfrak{q})$ is a doctrine morphism. Indeed, to prove that $\mathfrak{q}$ is a natural transformation, take $f:X\to Y$ and observe that:
\begin{equation*}\mathfrak{q}_X P(f)\alpha=[P(f)\alpha]=P/\nabla(f)[\alpha]=P/\nabla(f)\mathfrak{q}_Y(\alpha).\end{equation*}

Moreover, it can be easily shown that $P/\nabla$ is primary, with top and meet of $P/\nabla(X)$ computed as in $P(X)$, and that the quotient $(\id{\CC},\mathfrak{q})$ is a morphism of primary doctrines.
\begin{proposition}
Let $P:\CC\op\to\Pos$ be a primary doctrine and $\nabla\subseteq P(\tmn)$ be a filter. The 1-arrow $(\id{\CC},\mathfrak q):P\to P/\nabla$ is such that $\top\leq\mathfrak{q}_\tmn(\theta)$ in $P/\nabla(\tmn)$ for all $\theta\in\nabla$, and it is universal with respect to this property, i.e.\ for any primary 1-arrow $(G,\mathfrak{g}):P\to R$, where $R:\ct{D}\op\to\Pos$ is a primary doctrine, such that $\top\leq \mathfrak{g}_\tmn(\theta)$ in $R(\tmn_\ct{D})$ for all $\theta\in\nabla$, there exists a unique up to a unique natural isomorphism primary 1-arrow $(G',\mathfrak{g}'):P/\nabla\to R$ such that $(G',\mathfrak{g}')\circ(\id{\CC},\mathfrak q)=(G,\mathfrak{g})$.
\end{proposition}
\begin{proof}
At first, observe that any $\theta\in\nabla$ is sent to the top element of $P/\nabla(\tmn)$: indeed, consider $\theta\in \nabla$ itself to observe that $\theta\land\top_\tmn\leq\theta$, to that $[\top_\tmn]\leq[\theta]$.
We now show the universal property. First of all, since $G'\id{\CC}=G$, we observe that $G'=G:\CC\to\ct{D}$.
Then we show that for any fixed $\CC$-object $X$, the function $\mathfrak g _X:P(X)\to R(GX)$ factors through the quotient $\mathfrak q _X$, defining $\mathfrak{g}'_X([\alpha])=\mathfrak g_X(\alpha)$. To prove that this is well-defined, take $\alpha\sqsubseteq_\nabla\beta$ in $P(X)$, i.e.\ $P(!_X)(\theta)\land\alpha\leq\beta$. Then apply $\mathfrak g_X$ to get $\mathfrak g_XP(!_X)(\theta)\land\mathfrak g_X\alpha\leq\mathfrak g_X\beta$ in $R(GX)$. However $\mathfrak g_XP(!_X)(\theta)=R(!_{GX})\mathfrak g_\tmn(\theta)=\top_{GX}$, hence $\mathfrak g_X(\alpha)\leq\mathfrak g_X(\beta)$. As a result, we obtain a well-defined monotone function $\mathfrak g'_X:P/\nabla(X)\to R(GX)$ such that $\mathfrak g'_X\mathfrak q_X=\mathfrak g_X$---and it is also unique. Since $\mathfrak g_X$ preserves finite meets, and finite meets in $P/\nabla$ are computed as in $P$, it follows that $\mathfrak g'_X$ preserves finite meets. Moreover, we can use naturality of $\mathfrak g$ to show that $\mathfrak g':P/\nabla\to RG\op$ defines a natural transformation. In particular, $(G,\mathfrak g')$ is a primary 1-arrow such that $(G,\mathfrak{g}')\circ(\id{\CC},\mathfrak q)=(G,\mathfrak{g})$, and it is unique with respect to this property, as claimed.\end{proof}
\begin{remark}
In the proposition above, taking the particular case where the filter is $\nabla=\uparrow\varphi=\{\alpha\in P(\tmn)\mid\alpha\geq\varphi\}$ for some $\varphi\in P(\tmn)$, the universal property is the same seen in \Cref{coroll:ax}. It follows that there exists an isomorphism between the primary doctrines $P/\uparrow\varphi$ and $P_\varphi$.
\end{remark}
In the following lemma, we show that if $P$ has some additional structure, then $P/\nabla$ has them as well, and the structure is preserved by the quotient morphism.
\begin{lemma}\label{lemma:pres_filter}
Let $P$ be a primary doctrine, $\nabla\subseteq P(\tmn)$ be a filter and $P/\nabla$ be the quotient doctrine.
\begin{enumerate}[label=(\roman*)]
\item If $P$ is bounded, then the doctrine $P/\nabla$ and the 1-arrow $(\id{\CC},\mathfrak q)$ are bounded.
\item If $P$ is implicational, then the doctrine $P/\nabla$ and the 1-arrow $(\id{\CC},\mathfrak q)$ are implicational.
\item If $P$ is elementary, then the doctrine $P/\nabla$ and the 1-arrow $(\id{\CC},\mathfrak q)$ are elementary.
\item If $P$ is existential, then the doctrine $P/\nabla$ and the 1-arrow $(\id{\CC},\mathfrak q)$ are existential.
\item If $P$ is universal, then the doctrine $P/\nabla$ and the 1-arrow $(\id{\CC},\mathfrak q)$ are universal.
\end{enumerate}
\end{lemma}
\begin{proof}
\begin{enumerate}[label=(\roman*)]
\item We show that $\mathfrak{q}_X(\bot_X)=[\bot_X]\leq[\alpha]$ in $P/\nabla(X)$ for all $[\alpha]\in P/\nabla(X)$, but this holds since $P(!_X)\top_\tmn\land\bot_X=\bot_X\leq\alpha$ in $P(X)$. Naturality of the bottom element follows from naturality of $\mathfrak{q}$ and of the bottom in $P$. The quotient $(\id{\CC},\mathfrak{q})$ trivially preserves the bottom element.
\item We show that $\mathfrak{q}_X(\beta\to\gamma)=[\beta\to\gamma]=[\beta]\to[\gamma]$ in $P/\nabla(X)$ for all $[\beta],[\gamma]\in P/\nabla(X)$.
Suppose $[\alpha]\land[\beta]\leq[\gamma]$, if and only if there exists $\theta\in\nabla$ such that $P(!_X)\theta\land\alpha\land\beta\leq\gamma$ in $P(X)$, if and only if there exists $\theta\in\nabla$ such that $P(!_X)\theta\land\alpha\leq\beta\to\gamma$ in $P(X)$, if and only if $[\alpha]\leq[\beta\to\gamma]$, i.e.\ $[\beta\to\gamma]=[\beta]\to[\gamma]$ in $P/\nabla(X)$.
Naturality again follows from naturality of $\mathfrak{q}$ and of the bottom in $P$. The quotient $(\id{\CC},\mathfrak{q})$ preserves implication.
\item Consider the elementary doctrine $P$, and define $\underline{\delta}_A:=\mathfrak{q}_{A\times A}(\delta_A)=[\delta_A]$ in $P/\nabla(A\times A)$. This is trivially the fibered equality on $A$ for the doctrine $P/\nabla$ (simply take the quotient of the three inequalities in \Cref{def:elem}). So the doctrine $P/\nabla$ is elementary, and the quotient is a morphism of elementary doctrines.
\item Consider the existential doctrine $P$, with left adjoint $\exists^B_A\dashv P(\pr1)$ for any projection $\pr1:A\times B\to A$ in $\CC$. 

We show that $\mathfrak{q}_A(\exists^B_A\alpha)=[\exists^B_A\alpha]=\underline{\exists}^B_A[\alpha]$ in $P/\nabla(A)$ for all $[\alpha]\in P/\nabla(A\times B)$ defines the existential quantifier for the quotient doctrine $P/\nabla$. To show that $\underline{\exists}^B_A$ is well-defined on the quotients, suppose $\alpha\sqsubseteq_\nabla\beta$, for some $\alpha,\beta\in P(A\times B)$, i.e.\ there exists $\theta\in\nabla$ such that $P(!_{A\times B})\theta\land\alpha\leq\beta$ in $P(A\times B)$; then $\exists^B_A(P(!_{A\times B})\theta\land\alpha)=\exists^B_A(P(\pr1)P(!_{A})\theta\land\alpha)=\exists^B_A\alpha\land P(!_{A})\theta\leq\exists^B_A\beta$ in $P(A)$ by using Frobenius reciprocity, i.e.\ $[\exists^B_A\alpha]\leq[\exists^B_A\beta]$, so $\underline{\exists}^B_A[\alpha]=[\exists^B_A\alpha]$ is well-defined. This is the left adjoint to the reindexing along the first projection: take $[\alpha]\in P/\nabla(A\times B)$ and $[\gamma]\in P/\nabla(A)$, then $\underline{\exists}^B_A[\alpha]\leq[\gamma]$ if and only if there exists $\theta\in\nabla$ such that $P(!_A)\theta\land\exists^B_A\alpha\leq\gamma$ in $P(A)$, but $P(!_A)\theta\land\exists^B_A\alpha=\exists^B_A(\alpha\land P(\pr1)P(!_A)\theta)=\exists^B_A(\alpha\land P(!_{A\times B})\theta)$ by Frobenius reciprocity, hence if and only if there exists $\theta\in\nabla$ such that $\alpha\land P(!_{A\times B})\theta\leq P(\pr1)\gamma$ in $P(A\times B)$, if and only if $[\alpha]\leq P/\nabla(\pr1)[\gamma]$, as claimed.

Beck-Chevalley condition for $\underline{\exists}$ and Frobenious reciprocity follow from the same properties of $\exists$ in $P$.

So the doctrine $P/\nabla$ is existential, and the quotient is an existential doctrine morphism.

\item Consider the universal doctrine $P$, with right adjoint $P(\pr1)\dashv\forall^B_A$ for any projection $\pr1:A\times B\to A$ in $\CC$. 

We show that $\mathfrak{q}_A(\forall^B_A\alpha)=[\forall^B_A\alpha]=\underline{\forall}^B_A[\alpha]$ in $P/\nabla(A)$ for all $[\alpha]\in P/\nabla(A\times B)$ defines the existential quantifier for the quotient doctrine $P/\nabla$. To show that $\underline{\forall}^B_A$ is well-defined on the quotients, suppose $\alpha\sqsubseteq_\nabla\beta$, for some $\alpha,\beta\in P(A\times B)$, i.e.\ there exists $\theta\in\nabla$ such that $P(!_{A\times B})\theta\land\alpha\leq\beta$ in $P(A\times B)$; then $P(!_{A})\theta\land\forall^B_A\alpha\leq\forall^B_AP(\pr1)P(!_{A})\theta\land\forall^B_A\alpha=\forall^B_A(P(!_{A\times B})\theta\land\alpha)\leq\forall^B_A\beta$ in $P(A)$ by using the unity of the adjunction and the fact that right adjoint preserve limits---hence meets too---, i.e.\ $[\forall^B_A\alpha]\leq[\forall^B_A\beta]$, so $\underline{\forall}^B_A[\alpha]=[\forall^B_A\alpha]$ is well-defined. This is the right adjoint to the reindexing along the first projection: take $[\alpha]\in P/\nabla(A\times B)$ and $[\gamma]\in P/\nabla(A)$, then $[\gamma]\leq\underline{\forall}^B_A[\alpha]$ if and only if there exists $\theta\in\nabla$ such that $P(!_A)\theta\land\gamma\leq\forall^B_A\alpha$ in $P(A)$, if and only if there exists $\theta\in\nabla$ such that $P(!_{A\times B})\theta\land P(\pr1)\gamma\leq\alpha$ in $P(A\times B)$, if and only if $P/\nabla(\pr1)[\gamma]\leq[\alpha]$, as claimed.

Beck-Chevalley condition for $\underline{\forall}$ follows from the same property of $\forall$ in $P$.

So the doctrine $P/\nabla$ is universal, and the quotient is a universal doctrine morphism.\qedhere\end{enumerate}\end{proof}

%%%%%%%%%%%%%%%%%%%%%%%%%%%%%%%%%%%%%%%%%%%%%%%%%%%%%%%%%%%%%%%%%%%
\subsection{Definition of a model}\label{sub:def_of_mod}

Let $P:\CC\op\to\Pos$ be a bounded consistent existential implicational rich doctrine. Let $\nabla\subseteq P(\tmn)$ be an ultrafilter and $P/\nabla:\CC\op\to\Pos$ the quotient doctrine. Such ultrafilter exists since $\top\neq\bot$ in $P(\tmn)$, and we can take an extension of the proper filter $\{\top\}$---see \Cref{lemma:ultraf}. By \Cref{lemma:pres_filter}, the doctrine $P/\nabla$ is again bounded existential implicational, and all of these structures are preserved by the quotient morphism $(\id{\CC},\mathfrak{q}):P\to P/\nabla$.

We now build a model of $P/\nabla$ in the doctrine $\pws :\Set_{\ast}\op\to\Pos$, meaning a doctrine morphism $(\Gamma,\mathfrak{g}):P/\nabla\to\pws $. Also, this model preserves the bounded existential implicational structure. Define $\Gamma\coloneqq\text{Hom}_\CC(\tmn,\blank):\CC\to\Set_{\ast}$. It is well-defined since $P$ is rich, and this clearly preserves the products. Then, define for a given $X\in\text{ob}\CC$, $\mathfrak{g}_X:P/\nabla(X)\to\pws (\text{Hom}_\CC(\tmn,X))$:
\begin{equation*}
\begin{split}
\mathfrak{g}_X[\varphi]&=\{c:\tmn\to X\mid[\top]\leq P/\nabla(c)[\varphi]\}\\
&=\{c:\tmn\to X\mid[\top]\leq [P(c)\varphi]\}\\
&=\{c:\tmn\to X\mid P(c)\varphi\in\nabla\}.
\end{split}
\end{equation*}
\begin{proposition}\label{prop:mod}
Let $P$ be a bounded consistent implicational existential rich doctrine, let $\nabla\subseteq P(\tmn)$ be an ultrafilter, and let $P/\nabla$ be the quotient doctrine. Then the pair $(\Gamma,\mathfrak{g})$, where $\Gamma=\text{Hom}_\CC(\tmn,\blank)$ and $\mathfrak g_X[\varphi]=\{c:\tmn\to X\mid P(c)\varphi\in\nabla\}$ for any object $X$ and any $[\varphi]\in P/\nabla(X)$ is a bounded existential implicational morphism.
\end{proposition}
\begin{proof}
{\bf $\mathfrak{g}_X$ is monotone:}
Suppose $[\varphi]\leq[\psi]$ in $P/\nabla(X)$, i.e.\ there exists $\theta\in\nabla$ such that $P(!_X)\theta\leq \varphi\to\psi$; we show that $\mathfrak{g}_X[\varphi]\subseteq\mathfrak{g}_X[\psi]$.  Let $c:\tmn\to X$ be an arrow in $\CC$ such that $P(c)\varphi\in\nabla$. Apply $P(c)$ to the inequality above and get $\theta\leq P(c)(\varphi\to\psi)$; so $P(c)(\varphi\to\psi)\in\nabla$. Then, $P(c)\varphi\land P(c)(\varphi\to\psi)\leq P(c)\psi\in\nabla$, i.e.\ $c\in\mathfrak{g}_X[\psi]$.
\pointinproof{$\mathfrak{g}_X$ is a natural transformation}
Take $f:X\to Y$ an arrow in $\CC$. We want to show that the following diagram commutes:
\begin{center}
\begin{tikzcd}
Y&P/\nabla(Y)\arrow[r,"\mathfrak{g}_Y"]\arrow[d,"P/\nabla(f)"]& \pws(\text{Hom}_\CC(\tmn,Y))\arrow[d,"{(f\circ\blank)^{-1}}"]\\
X\arrow[u,"f"]&P/\nabla(X)\arrow[r,"\mathfrak{g}_X"]& \pws(\text{Hom}_\CC(\tmn,X))
\end{tikzcd}.
\end{center}
Consider $c:\tmn\to X$; $c\in\mathfrak{g}_X P/\nabla(f)[\varphi]$ if and only if $P(c)P(f)\varphi\in\nabla$. On the other hand, $c\in(f\circ\blank)^{-1}\mathfrak{g}_Y[\varphi]$ if and only if $fc\in\mathfrak{g}_Y[\varphi]$ if and only if $P(fc)\varphi\in\nabla$.

In particular, $(\text{Hom}_{\CC}(\tmn,\blank),\mathfrak{g})$ is a morphism of doctrines. We now prove that all the other properties are preserved.
\pointinproof{$\mathfrak{g}_X$ preserves top and bottom elements} We observe that
\begin{equation*}\mathfrak{g}_X[\top_X]=\{c:\tmn\to X\mid P(c)\top_X\in\nabla\}=\text{Hom}_\CC(\tmn,X),\end{equation*}
since $P(c)\top_X=\top_\tmn\in\nabla$ for any $c$. Moreover,
\begin{equation*}\mathfrak{g}_X[\bot_X]=\{c:\tmn\to X\mid P(c)\bot_X\in\nabla\}=\emptyset,\end{equation*}
since $P(c)\bot_X=\bot_\tmn\notin\nabla$ for any $c$.
\pointinproof{$\mathfrak{g}_X$ preserves meets} We have
\begin{align*}\mathfrak{g}_X([\varphi]\land[\psi])&=\mathfrak{g}_X([\varphi\land\psi])\\
&=\{c:\tmn\to X\mid P(c)\varphi\land P(c)\psi\in\nabla\}\\
&=\{c:\tmn\to X\mid P(c)\varphi\in\nabla \text{ and } P(c)\psi\in\nabla\}\\
&=\mathfrak{g}_X[\varphi]\cap\mathfrak{g}_X[\psi].\end{align*}
\pointinproof{$\mathfrak{g}_X$ preserves implication} We have
\begin{align*}\mathfrak{g}_X([\varphi]\to[\psi])&=\mathfrak{g}_X([\varphi\to\psi])=\{c:\tmn\to X\mid P(c)\varphi\to P(c)\psi\in\nabla\}\quad\quad\text{and}\\
\mathfrak{g}_X[\varphi]\Rightarrow\mathfrak{g}_X[\psi]&=\{c:\tmn\to X\mid P(c)\psi\in\nabla\}\cup\{c:\tmn\to X\mid P(c)\varphi\notin\nabla\}.\end{align*}
We show that the two sets coincide. First of all, suppose $c:\tmn\to X$ be such that $P(c)\varphi\to P(c)\psi\in\nabla$; then consider $P(c)\varphi$. If $P(c)\varphi\in\nabla$, we get $P(c)\varphi\land(P(c)\varphi\to P(c)\psi)\leq P(c)\psi\in\nabla$; otherwise, $P(c)\varphi\notin\nabla$. In both cases $c\in \mathfrak{g}_X[\varphi]\Rightarrow\mathfrak{g}_X[\psi]$. For the converse, take at first $c$ such that $P(c)\psi\in\nabla$. Since $P(c)\psi\leq P(c)\varphi\to P(c)\psi$, we get $P(c)\varphi\to P(c)\psi\in\nabla$. Then, take $c$ such that $P(c)\varphi\notin\nabla$; since $\nabla$ is an ultrafilter, $P(c)\varphi\to\bot\in\nabla$. But then, $P(c)\varphi\to\bot\leq P(c)\varphi\to P(c)\psi$ since $P(c)\varphi\to(\blank)$ is monotone; so $P(c)\varphi\to P(c)\psi\in\nabla$.
\pointinproof{$\mathfrak{g}_X$ preserves existential quantifier}
Recall that, given a function between two sets $h:A\to B$, the left adjoint to the preimage $h^{-1}:\pws(B)\to\pws(A)$ acts on any subset of $A$ as the image $\exists_h=h:\pws(A)\to\pws(B)$.

So now we show $\exists_{\pr1\circ\blank}\mathfrak{g}_{X\times Y}[\varphi]=\mathfrak{g}_X\underline{\exists}^Y_X[\varphi]$ for any pair $X,Y$ of objects in $\CC$.
First of all, observe that the inclusion $(\subseteq)$ holds if and only if $\mathfrak{g}_{X\times Y}[\varphi]\subseteq(\pr1\circ\blank)^{-1}\mathfrak{g}_X[\exists^Y_X\varphi]$ but
\begin{equation*}(\pr1\circ\blank)^{-1}\mathfrak{g}_X[\exists^Y_X\varphi]=\pws(\pr1\circ\blank)\mathfrak{g}_X[\exists^Y_X\varphi]=\mathfrak{g}_{X\times Y}P/\nabla(\pr1)[\exists^Y_X\varphi]\end{equation*}
and $[\varphi]\leq P/\nabla(\pr1)[\exists^Y_X\varphi]$. Concerning the converse, observe that
\begin{align*}\exists_{\pr1\circ\blank}\mathfrak{g}_{X\times Y}[\varphi]&=\{c:\tmn\to X\mid \text{there exists }d:\tmn\to Y\text{ such that }\ple{c,d}\in\mathfrak{g}_{X\times Y}[\varphi]\}\\
&=\{c:\tmn\to X\mid \text{there exists }d:\tmn\to Y\text{ such that }P(\ple{c,d})\varphi\in\nabla\}.\end{align*}
Then take $c:\tmn\to X$ such that $P(c)\exists^Y_X\varphi=\exists^Y_\tmn P(\ple{c!,\id{Y}})\varphi\in\nabla$. Since $P$ is rich, we can take $d:\tmn\to Y$ such that
\begin{equation*}\exists^Y_\tmn P(\ple{c!,\id{Y}})\varphi=P(d)P(\ple{c!,\id{Y}})\varphi=P(\ple{c,d})\varphi,\end{equation*}
so that $c\in\exists_{\pr1\circ\blank}\mathfrak{g}_{X\times Y}[\varphi]$.\end{proof}
\begin{remark}
Observe that in the proof above, we used the assumptions of consistency for the existence of the ultrafilter, the fact that the filter is an ultrafilter to prove that the model is implicational, and richness to prove that the model is existential. 
\end{remark}
\begin{example}{\bf A counterexample to universality.}\label{ex:real}
We prove that in general, if we add the universal quantifier to our structure, it is not necessarily preserved by the model we defined above. We will consider a slight change of the domain in the realizability doctrine, defined in \cite{hjp}: $R:\Set_{\ast}\op\to\Pos$ takes value from the opposite category of non-empty sets. For each non-empty set $I$, define the following preorder in $\mathscr{P}(\mathbb{N})^I=\{p:I\to\mathscr{P}(\mathbb{N})\}$: we say that $p\leq q$ if there exists a partial recursive function $\varphi:\mathbb{N}\dashrightarrow\mathbb{N}$ such that for all $i\in I$ the restriction $\varphi_{\vert{p(i)}}:p(i)\to q(i)$ is a total function; reflexivity is witnessed by the identity $\id{\mathbb{N}}$, while transitivity can be proved by taking the composition of the two partial functions as witness. Then, define $R(I)$ to be the poset reflection of this preorder. The reindexing along a function $\alpha: J\to I$ is given by the precomposition $\blank\circ\alpha: R(I)\to R(J)$; note that if $p\leq q$ in $\mathscr{P}(\mathbb{N})^I$ is witnessed by $\varphi:\mathbb{N}\dashrightarrow\mathbb{N}$, also $p\alpha\leq q\alpha$ in $\mathscr{P}(\mathbb{N})^J$ is again witnessed by $\varphi$.
\pointinproof{{$R$ is primary}}
First of all observe that in each $R(I)$, the constant function $T_I:I\to\mathscr{P}(\mathbb{N})$ sending each $i\in I$ to $\mathbb{N}$ is the top element: take any other $p:I\to\mathscr{P}(\mathbb{N})$ and consider $\id{\mathbb{N}}$, so that the inclusion ${\id{\mathbb{N}}}_{\vert{p(i)}}:p(i)\to\mathbb{N}$ is a total function for every $i\in I$, giving $p\leq T_I$. Moreover, for any $\alpha: J\to I$, precomposition $T_I\alpha=T_J$ is again the constant function to the element $\mathbb{N}$, so the top element is preserved by reindexing. Then, for any $p,q:I\to\mathscr{P}(\mathbb{N})$, define for each $i\in I$, $(p\land q)(i)\coloneqq \{\ple{a,b}\in\mathbb{N}\mid a\in p(i), b\in q(i)\}$; here $\ple{\blank,\blank}:\mathbb{N}\times\mathbb{N}\xleftrightarrow{\sim}\mathbb{N}:\ple{\pi_1,\pi_2}$ are Cantor's pairing and unpairing functions. The inequalities $p\land q\leq p$ and $p\land q\leq q$ are witnessed by the (total) functions $\pi_1:\mathbb{N}\to\mathbb{N}$ and $\pi_2:\mathbb{N}\to\mathbb{N}$ respectively. Suppose now $r\leq p$ and $r\leq q$, with given recursive functions $\varphi$ and $\psi$; then define $\ple{\varphi,\psi}:\mathbb{N}\dashrightarrow\mathbb{N}$ whose domain is the intersection of the domains of $\varphi$ and $\psi$, sending $n\in\text{dom}\varphi\cap\text{dom}\psi$ to $\ple{\varphi(n),\psi(n)}$, so that $\ple{\varphi,\psi}$ is partial recursive and witnesses $r\leq p\land q$. As before, take $\alpha: J\to I$: for any $j\in J$ we have $(p\land q)(\alpha(j))=\{\ple{a,b}\in\mathbb{N}\mid a\in p\alpha(j), b\in q\alpha(j)\}=(p\alpha\land q\alpha)(j)$, so the meet is preserved by reindexings, hence $R$ is a primary doctrine.
\pointinproof{{$R$ has bottom elements}}
In each $R(I)$, the constant function $B_I:I\to\mathscr{P}(\mathbb{N})$ sending each $i\in I$ to $\emptyset$ is the bottom element: take any other $p:I\to\mathscr{P}(\mathbb{N})$ and consider $\id{\mathbb{N}}$, so that the inclusion ${\id{\mathbb{N}}}_{\vert{\emptyset}}:\emptyset\to p(i)$ is a total function for every $i\in I$, giving $B_I\leq p$. Moreover, for any $\alpha: J\to I$, precomposition $B_I\alpha=B_J$ is again the constant function to the element $\emptyset$, so the bottom element is preserved by reindexing.
\pointinproof{{$R$ is implicational}}
For any $p,q:I\to\mathscr{P}(\mathbb{N})$, define for each $i\in I$, $(p\to q)(i)$ as the set $\{e\in\mathbb{N}\mid e \text{ encodes a partial recursive function }\theta:\mathbb{N}\dashrightarrow\mathbb{N}\text{ such that }\theta\text{ maps }p(i)\text{ in }q(i)\}$. 
To prove that this is indeed the implication in $R(I)$, take $r\in R(I)$ and suppose $r\land p\leq q$, if and only if there exists $\varphi:\mathbb{N}\dashrightarrow\mathbb{N}$ such that for every $i\in I$, $\varphi_{\vert(r\land p)(i)}:(r\land p)(i)\to q(i)$ is a total function. For a given $n\in\mathbb{N}$, we can consider the partial function $\varphi(\ple{n,\blank}):\mathbb{N}\dashrightarrow\mathbb{N}$, $m\mapsto\varphi(\ple{n,m})$ when it exists; define $\psi:\mathbb{N}\to\mathbb{N}$ the (total) function that maps $n$ to the natural number that encodes $\varphi(\ple{n,\blank})$. For each $i\in I$, the restriction $\psi_{\vert r(i)}$ is defined over all $r(i)$, and its image is in $(p\to q)(i)$, proving $r\leq p\to q$: indeed, take $n\in r(i)$, then $\psi(n)\in(p\to q)(i)$ if and only if $\varphi(\ple{n,\blank})$ maps $p(i)$ to $q(i)$, but if we take any $m\in p(i)$, then $\ple{n,m}\in(r\land p)(i)$, so that $\varphi(\ple{n,m})\in q(i)$. Now, to prove the converse, suppose $r\leq p\to q$, if and only if there exists $\psi:\mathbb{N}\dashrightarrow\mathbb{N}$ such that for every $i\in I$, $\psi_{\vert r(i)}:r(i)\to(p\to q)(i)$ is a total function. For any $k\in\mathbb{N}$, recall that $k=\ple{n,m}$ where $n=\pi_1(k)$ and $m=\pi_2(k)$; if $\psi(n)$ exist, call $\theta_n:\mathbb{N}\dashrightarrow\mathbb{N}$ the partial function encoded by the natural number $\psi(n)$. Define $\varphi:\mathbb{N}\dashrightarrow\mathbb{N}$ such that $\ple{n,m}\mapsto \theta_n(m)$ whenever both $\psi(n)$ and $\theta_n(m)$ are defined. For each $i\in I$, the restriction $\varphi_{\vert(r\land p)(i)}$ is defined over all $(r\land p)(i)$, and its image is in $q(i)$, proving $r\land p\leq q$: indeed, take $k=\ple{n,m}\in(r\land p)(i)$, hence $n\in r(i)$ and $m\in p(i)$; then $\psi(n)$ is defined and belongs to $(p\to q)(i)$, hence encodes a partial recursive function $\theta_n$ that maps $p(i)$ to $q(i)$. Since $m\in p(i)$, we have $\varphi(k)=\theta_n(m)\in q(i)$, as claimed.

Take then $\alpha: J\to I$: for any $j\in J$ we have on the one hand $(R(\alpha)(p\to q))(j)=(p\to q)(\alpha(j))=\{e\in\mathbb{N}\mid e \text{ encodes a partial recursive function }\theta:\mathbb{N}\dashrightarrow\mathbb{N}\text{ such that }\theta\text{ maps }p(\alpha(j))\text{ in }q(\alpha(j))\}$, and on the other hand
\begin{align*}(R(\alpha)(p)\to R(\alpha)(q))(j)=\{d\in\mathbb{N}\mid \, & d \text{ encodes a partial recursive function }\tau:\mathbb{N}\dashrightarrow\mathbb{N}\\
&\text{ such that }\tau\text{ maps }R(\alpha)(p)(j)\text{ in }R(\alpha)(q)(j)\},\end{align*}
so the implication is preserved by reindexings, hence $R$ is an implicational doctrine.
\pointinproof{{$R$ is existential}}
For each pair of non-empty sets $I,J$, consider $\pr1:I\times J\to I$ and define $\exists^J_I:R(I\times J)\to R(I)$ that maps a function $q:I\times J\to\mathscr{P}(\mathbb{N})$ to $\exists^J_Iq:I\to\mathscr{P}(\mathbb{N})$, $(\exists^J_Iq)(i)=\bigcup_{j\in J}q(i,j)$. This is the left adjoint to $R(\pr1)$: $\exists^J_Iq\leq p$ if and only if there exists $\varphi:\mathbb{N}\dashrightarrow\mathbb{N}$ such that for all $i\in I$, $\varphi_{\vert{(\exists^J_Iq)(i)}}:\bigcup_{j\in J}q(i,j)\to p(i)$ is a total function, if and only if there exists $\varphi:\mathbb{N}\dashrightarrow\mathbb{N}$ such that for all $i\in I$ and $j\in J$, $\varphi_{\vert{q(i,j)}}:q(i,j)\to p(i)$ is a total function, if and only if $q\leq R(\pr1)p$.

To show Beck-Chevalley condition, take a function $\alpha: K\to I$: for any $q\in R(I\times J)$ and $k\in K$,
\begin{equation*}R(\alpha)(\exists^J_Iq)(k)=(\exists^J_Iq)(\alpha(k))=\bigcup_{j\in J}q(\alpha(k),j)\end{equation*}
and also
\begin{equation*}(\exists^J_KR(\alpha\times\id{J})q)(k)=\bigcup_{j\in J}(R(\alpha\times\id{J})q)(k,j)=\bigcup_{j\in J}q(\alpha(k),j)\end{equation*}
so that $R(\alpha)(\exists^J_Iq)=\exists^J_KR(\alpha\times\id{J})q$, as claimed.

To show Frobenius reciprocity, for any $q\in R(I\times J)$, $p\in R(I)$, and $i\in I$
\begin{equation*}\exists^J_I(q\land R(\pr1)p)(i)=\bigcup_{j\in J}((q\land R(\pr1)p)(i,j))=\bigcup_{j\in J}\{\ple{a,b}\in\mathbb{N}\mid a\in q(i,j),b\in p(i)\}\end{equation*}
and also
\begin{equation*}(\exists^J_Iq\land p)(i)=\{\ple{a,b}\in\mathbb{N}\mid a\in \bigcup_{j\in J}q(i,j),b\in p(i)\}\end{equation*}
so that $\exists^J_I(q\land R(\pr1)p)=\exists^J_Iq\land p$, hence Frobenius reciprocity holds.
\pointinproof{{$R$ is consistent}}
Take $R(\{\star\})=\mathscr{P}(\mathbb{N})$; $T_{\{\star\}}\nleq B_{\{\star\}}$ since for every partial recursive function $\varphi:\mathbb{N}\dashrightarrow\mathbb{N}$ is it not the case that $\varphi_{\vert{\mathbb{N}}}:\mathbb{N}\to\emptyset$ can be defined.
\pointinproof{{$R$ is rich}}
Take any $q\in R(J)$ for a non-empty set $J$, we then look for a function $\overline{c}:\{\star\}\to J$, hence an element $c=\overline{c}(\star)\in J$, such that $\exists^J_{\{\star\}}q\leq R(\overline{c})q$, i.e.\ such that there exists a partial recursive function $\varphi:\mathbb{N}\dashrightarrow\mathbb{N}$ such that $\varphi_{\vert{\bigcup_{j\in J}q(j)}}:\bigcup_{j\in J}q(j)\to q(c)$ is a total function. Here is the point where the usual realizability doctrine defined over $\Set$ does not satisfy the needed assumption, and we need to remove the empty set from the base category. If $\bigcup_{j\in J}q(j)=\emptyset$, choose any $c\in J$ and $\varphi=\id{\mathbb{N}}$, so that ${\id{\mathbb{N}}}_{\vert{\emptyset}}:\emptyset\to q(c)$ is a total function, as claimed. On the other hand, if $\bigcup_{j\in J}q(j)\neq\emptyset$, there exist $n\in\mathbb{N}$ and $c\in J$ such that $n\in q(c)$; choose $\varphi:\mathbb{N}\to\mathbb{N}$ to be the constant function to $n$, so that the restriction $\varphi_{\vert{\bigcup_{j\in J}q(j)}}:\bigcup_{j\in J}q(j)\to q(c)$ is a total function, again as wanted.
\pointinproof{{$R$ is universal}}
For each pair of non-empty sets $I,J$, consider $\pr1:I\times J\to I$ and define $\forall^J_I:R(I\times J)\to R(I)$ that maps a function $q:I\times J\to\mathscr{P}(\mathbb{N})$ to $\forall^J_Iq:I\to\mathscr{P}(\mathbb{N})$, $(\forall^J_Iq)(i)=\bigcap_{j\in J}q(i,j)$. This is the right adjoint to $R(\pr1)$: $p\leq \forall^J_Iq$ if and only if there exists $\varphi:\mathbb{N}\dashrightarrow\mathbb{N}$ such that for all $i\in I$, $\varphi_{\vert{p(i)}}:p(i)\to \bigcap_{j\in J}q(i,j)$ is a total function, if and only if there exists $\varphi:\mathbb{N}\dashrightarrow\mathbb{N}$ such that for all $i\in I$ and $j\in J$, $\varphi_{\vert{p(i)}}:p(i)\to q(i,j)$ is a total function, if and only if $q\leq R(\pr1)p$.
To show Beck-Chevalley condition, take a function $\alpha: K\to I$: for any $q\in R(I\times J)$ and $k\in K$,
\begin{equation*}R(\alpha)(\forall^J_Iq)(k)=(\forall^J_Iq)(\alpha(k))=\bigcap_{j\in J}q(\alpha(k),j)\end{equation*}
and also
\begin{equation*}(\forall^J_KR(\alpha\times\id{J})q)(k)=\bigcap_{j\in J}(R(\alpha\times\id{J})q)(k,j)=\bigcap_{j\in J}q(\alpha(k),j)\end{equation*}
so that $R(\alpha)(\forall^J_Iq)=\forall^J_KR(\alpha\times\id{J})q$, as claimed.
\pointinproof{{Universal quantifier not preserved---expanding the cofinite sets}} Our next goal is to find an ultrafilter $\nabla\subseteq R(\{\star\})=\mathscr{P}(\mathbb{N})$ such that the morphism we built in \Cref{prop:mod} $(\Gamma,\mathfrak{g}):R/\nabla\to\pws $ does not preserve the universal quantifier: in particular we will find a non-empty set $J$ and a $q\in R(J)$ such that $\forall_{!_J}\mathfrak{g}_{J}[q]\not\subseteq\mathfrak{g}_{\{\star\}}\underline{\forall}^J_{\{\star\}}[q]$.
Recall that, given a function between two sets $h:A\to B$, the right adjoint to the preimage $h^{-1}:\pws(B)\to\pws(A)$ sends a subset $S$ of $A$ to the set $\forall_hS\coloneqq\{b\in B\mid h^{-1}(b)\subseteq S\}$. In our case, we have $\forall_{!_J}\mathfrak{g}_{J}[q]\neq\emptyset$ if and only if $J\subseteq \mathfrak{g}_J[q]=\{j\in J\mid q(j)\in\nabla\}$. Then, observe that $\mathfrak{g}_{\{\star\}}[\forall^J_{\{\star\}}q]\neq\emptyset$ if and only if $\forall^J_{\{\star\}}q\in\nabla$.
\[\begin{tikzcd}
	{R(J)} & {R/\nabla(J)} & {\mathscr{P}(J)} \\
	{R(\{\star\})} & {R/\nabla(\{\star\})} & {\mathscr{P}(\{\star\})}
	\arrow["{\forall^J_{\{\star\}}}"', from=1-1, to=2-1]
	\arrow["{\mathfrak{q}_J}", from=1-1, to=1-2]
	\arrow["{\mathfrak{g}_J}", from=1-2, to=1-3]
	\arrow["{\mathfrak{q}_{\{\star\}}}"', from=2-1, to=2-2]
	\arrow["{\underline{\forall}^J_{\{\star\}}}"', from=1-2, to=2-2]
	\arrow["{\mathfrak{g}_{\{\star\}}}"', from=2-2, to=2-3]
	\arrow["{\forall_{!_J}}", from=1-3, to=2-3]
\end{tikzcd}\]
Suppose $\nabla\subseteq\mathscr{P}(\mathbb{N})$ is an ultrafilter that contains all cofinite sets of $\mathbb{N}$; then take $J\coloneqq\mathbb{N}$ and $q:\mathbb{N}\to\mathscr{P}(\mathbb{N})$ such that $q(n)\coloneqq\mathbb{N}\setminus\{n\}$. We show that for all $j\in J$, $q(j)\in\nabla$, but $\forall^J_{\{\star\}}q\notin\nabla$, so that $\forall_{!_J}\mathfrak{g}_{J}[q]\not\subseteq\mathfrak{g}_{\{\star\}}\underline{\forall}^J_{\{\star\}}[q]$. Since $q(j)$ is clearly cofinite for every $j$, each $q(j)\in\nabla$; then compute $\forall^J_{\{\star\}}q=\bigcap_{j\in J}q(j)=\bigcap_{n\in \mathbb{N}}\mathbb{N}\setminus\{n\}=\emptyset\notin\nabla$.
To conclude our proof, we need to show the existence of an ultrafilter over $\mathscr{P}(\mathbb{N})$ that contains every cofinite set. It is enough to prove that the filter generated by cofinite sets is a proper filter---i.e.\ does not contain the bottom element. Take the filter $F=\ple{\mathscr{C}}$ where $\mathscr{C}$ is the set of all cofinite sets of $\mathbb{N}$ and suppose that it contains the bottom element. Recall from above that the bottom is $\emptyset$ and the meet of two subsets $A,B$ of $\mathbb{N}$ is computed as $A\land B=\{\ple{a,b}\in\mathbb{N}\mid a \in A,b\in B\}$. Note that if $A$ and $B$ are cofinite, $A\land B$ is not in general cofinite, hence $\mathscr{C}$ is not a filter, as it is instead by taking the intersection as meet. However, suppose that $A\land B\leq\emptyset$ for a given pair $A,B\subseteq\mathbb{N}$, i.e.\ there exists a partial recursive function $\varphi:\mathbb{N}\dashrightarrow\mathbb{N}$ such that $\varphi_{\vert A\land B}:A\land B\to\emptyset$ is total, hence $A\land B=\emptyset$. In particular, it follows that at least one between $A$ and $B$ must be the empty set: if both $A\neq\emptyset$ and $B\neq\emptyset$, we can take $a\in A$ and $b\in B$, so that $\ple{a,b}\in A\land B\neq\emptyset$. Having noticed this, if it were the case that $\emptyset\in F$, there would exist $A_1,\dots,A_n\in\mathscr{C}$ such that $((A_1\land A_2)\land\dots\land A_{n-1})\land A_n\leq\bot$, so that one between $((A_1\land A_2)\land\dots\land A_{n-1})$ and $A_n$ would be the empty set; since $A_n\in\mathscr{C}$, we must have $((A_1\land A_2)\land\dots\land A_{n-1})=\emptyset$; by induction we get to a contradiction, so $\emptyset\notin F$, hence $F$ is a proper filter.
\end{example}
\begin{remark}
Suppose that the starting doctrine $P:\CC\op\to\Pos$ in \Cref{prop:mod} is also Boolean, meaning that we have the additional condition that $\lnot\lnot$ is the identity on each $P(X)$. Then, in particular, also $P/\nabla$ is a Boolean algebra, since the quotient preserves both implication and bottom element. Under this assumption, we obtain that the model $(\Gamma,\mathfrak{g})$ is Boolean. In particular, since the morphism is existential and Boolean, it is also universal.
\end{remark}
A little more work must be done in general if the starting doctrine is also elementary---in addition to the bounded implicational existential rich structure---and we want the model to preserve the elementary structure. So this time we define a morphism $(\Omega,\mathfrak{h}):P/\nabla\to\pws $ preserving the bounded elementary existential implicational structure. Define for each object $X$ the following equivalence relation $\sim^X_\nabla$ on $\text{Hom}_\CC(\tmn,X)$: given $c,d:\tmn\to X$, se say that $c\sim^X_\nabla d$ if and only if $P(\ple{c,d})\delta_X\in\nabla$.
\begin{itemize}
\item Reflexivity: $P(\ple{c,c})\delta_X=P(c)P(\Delta_X)\delta_X\geq P(c)\top_X=\top_\tmn\in\nabla$, so $c\sim^X_\nabla c$;
\item symmetry: suppose $P(\ple{c,d})\delta_X\in\nabla$, then since $\nabla$ is a filter
\begin{equation*}P(\ple{c,d})\delta_X\leq P(\ple{c,d})P(\ple{\pr2,\pr1})\delta_X=P(\ple{d,c})\delta_X\in\nabla,\end{equation*}
this follows from the fact that we have $\delta_X\leq P(\ple{\pr2,\pr1})\delta_X$. Indeed, using \ref{i:3-equality}.\ and \ref{i:2-equality}.\ in \Cref{def:elem}, we have in $P(X\times X\times X\times X)$
\begin{align}
&P(\ple{\pr1,\pr2})\delta_X\land P(\ple{\pr1,\pr3})\delta_X\land P(\ple{\pr2,\pr4})\delta_X\notag\\
&\leq P(\ple{\pr1,\pr2})\delta_X\land \delta_{X\times X}\label{eq:symm}\\
&\leq P(\ple{\pr3,\pr4})\delta_X.\notag
\end{align} 
Taking the reindexing along $P(\ple{\pr1,\pr1,\pr2,\pr1})$, we obtain that $\delta_X\leq P(\ple{\pr2,\pr1})\delta_X$ in $P(X\times X)$, as claimed;
\item transitivity: suppose $c\sim^X_\nabla d$ (hence also $d\sim^X_\nabla c$ by symmetry) and $d\sim^X_\nabla a$, then apply $P(\ple{d,d,c,a})$ to the outmost inequality in \eqref{eq:symm} to get $\top_\tmn\land P(\ple{d,c})\delta_X\land P(\ple{d,a})\delta_X\leq P(\ple{c,a})\delta_X$, hence $c\sim^X_\nabla a$.
\end{itemize}

Given $f:X\to Y$, post-composition $f\circ\blank:\text{Hom}_\CC(\tmn, X)\to\text{Hom}_\CC(\tmn, Y)$ is well-defined on the quotients: take $c\sim^X_\nabla d$ for some $c,d:\tmn\to X$, i.e.\ $P(\ple{c,d})\delta_X\in\nabla$, we show that $fc\sim^Y_\nabla fd$.
Since $P(\ple{\pr1,\pr1})P(f\times f)\delta_Y\land\delta_X\leq P(f\times f)\delta_Y$---see \Cref{lemma:aeq}---and $(f\times f)\ple{\pr1,\pr1}=\Delta_Yf\pr1$ we obtain $\delta_X\leq P(f\times f)\delta_Y$. Applying $P(\ple{c,d})$ we get $P(\ple{c,d})\delta_X\leq P(\ple{c,d})P(f\times f)\delta_Y=P(\ple{fc,fd})\delta_Y$, so that $fc\sim^Y_\nabla fd$ as claimed. Hence, we can define the functor
\begin{equation*}\Omega\coloneqq\text{Hom}_\CC(\tmn,\blank)/{\sim^{(\blank)}_\nabla}:\CC\to\Set_{\ast}.\end{equation*}
This preserves the products: take $a,c:\tmn\to X$ and $b,d:\tmn\to Y$, we have $\ple{a,b}\sim^{X\times Y}_\nabla\ple{c,d}$ if and only if $P(\ple{a,b,c,d})\delta_{X\times Y}\in\nabla$. Applying $P(\ple{a,b,c,d})$ to the equality $P(\ple{\pr1,\pr3})\delta_X\land P(\ple{\pr2,\pr4})\delta_Y= \delta_{X\times Y}$---see property \ref{i:3-equality}.\ in \Cref{def:elem} and \Cref{rmk:aeq}---, we get
\begin{equation*}P(\ple{a,c})\delta_X\land P(\ple{b,d})\delta_Y= P(\ple{a,b,c,d})\delta_{X\times Y},\end{equation*}
so that $P(\ple{a,b,c,d})\delta_{X\times Y}\in\nabla$ if and only if both
\begin{equation*}P(\ple{a,c})\delta_X\in\nabla\text{ and }P(\ple{b,d})\delta_Y\in\nabla,\end{equation*}
if and only if $a\sim^X_\nabla c$ and $b\sim^Y_\nabla d$; so we proved that
\begin{equation*}\text{Hom}_\CC(\tmn,X\times Y)/{\sim^{X\times Y}_\nabla}=\text{Hom}_\CC(\tmn,X)/{\sim^{X}_\nabla}\times \text{Hom}_\CC(\tmn, Y)/{\sim^{ Y}_\nabla}.\end{equation*}

Then, define for a given $X\in\text{ob}\CC$, $\mathfrak{h}_X:P/\nabla(X)\to\pws(\text{Hom}_\CC(\tmn,X)/{\sim^{X}_\nabla})$:
\begin{equation*}\mathfrak{h}_X[\varphi]=\{[c:\tmn\to X]\mid P(c)\varphi\in\nabla\}.\end{equation*}
This is well-defined, since whenever $c\sim^X_\nabla d$ and $[c]\in\mathfrak{h}_X[\varphi]$ we can apply $P(\ple{c,d})$ to the inequality $\delta_X\land P(\pr1)\varphi\leq P(\pr2)\varphi$ to get $P(c)\varphi\to P(d)\varphi\in\nabla$, and hence $P(d)\varphi\in\nabla$.
\begin{proposition}\label{prop:mod_eq}
Let $P$ be a bounded consistent implicational elementary existential rich doctrine, let $\nabla\subseteq P(\tmn)$ be an ultrafilter, and let $P/\nabla$ be the quotient doctrine. Then the pair $(\Omega,\mathfrak{h})$, where $\Omega\coloneqq\text{Hom}_\CC(\tmn,\blank)/{\sim^{(\blank)}_\nabla}$ and $\mathfrak h_X[\varphi]=\{[c:\tmn\to X]\mid P(c)\varphi\in\nabla\}$ for any object $X$ and any $[\varphi]\in P/\nabla(X)$ is a bounded elementary existential implicational morphism.
\end{proposition}
\begin{proof}
All proofs from \Cref{prop:mod} can be rearranged in this scenario to prove that $(\Omega,\mathfrak{h})$ is a morphism of doctrines, preserving bounded implicational existential structure. The last thing left to prove is that $(\Omega,\mathfrak{h})$ preserves the fibered equality:
\begin{equation*}\mathfrak{h}_{A\times A}(\underline{\delta}_A)=\mathfrak{h}_{A\times A}([{\delta}_A])=\{([c:\tmn\to A],[d:\tmn\to A])\in\Omega A\times \Omega A\mid P(\ple{c,d})\delta_A\in\nabla\}=\Delta_{\Omega A}.\qedhere\end{equation*}\end{proof}
We now have all the ingredients to generalize Henkin's Theorem.
\begin{theorem}\label{thm:mod_ex}
Let $P$ be a bounded existential implicational doctrine, with non-trivial fibers and with a small base category. Then there exists a bounded existential implicational model of $P$ in the doctrine of subsets $\pws :\Set_{\ast}\op\to\Pos$.
\end{theorem}
\begin{proof}
Do the construction in \Cref{rmk:recap} to get a morphism $(F,\mathfrak f):P\to\underrightarrow{P}$ that preserves bounded implicational existential structure; moreover by \Cref{prop:coher} the doctrine $\underrightarrow{P}$ is consistent. So $\underrightarrow{P}$ is an existential, bounded, implicational doctrine, consistent and rich, then we can choose an ultrafilter $\nabla\subseteq\underrightarrow{P}(\tmn)$ and take the quotient over it, and then the model $(\Gamma,\mathfrak g)$ of such quotient. The composition
\begin{equation*}P\xrightarrow{(F,\mathfrak f)}\underrightarrow{P}\xrightarrow{(\id{},\mathfrak q)}\underrightarrow{P}/\nabla\xrightarrow{(\Gamma,\mathfrak g)}\pws \end{equation*}
is a model of $P$, preserving all said structure.\end{proof}
\begin{theorem}\label{thm:mod_ex_eq}
Let $P$ be a bounded elementary existential implicational doctrine, such that each of its fibers is non-trivial and with a small base category. Then there exists a bounded elementary existential implicational model of $P$ in the doctrine of subsets $\pws :\Set_{\ast}\op\to\Pos$.
\end{theorem}
\begin{proof}
Do as above but take $(\Omega,\mathfrak h)$ instead of $(\Gamma,\mathfrak g)$.\end{proof}
%%%%%%%%%%%%%%%%%%%%%%%%%%%%%%%%%%%%%%%%%%%%%%%%%%%%%%%%%%%%%%%%%%%%
%%%%%%%%%%%%%%%%%%%%%%%%%%%%%%%%%%%%%%%%%%%%%%%%%%%%%%%%%%%%%%%%%%%
%%%%%%%%%%%%%%%%%%%%%%%%%%%%%%%%%%%%%%%%%%%%%%%%%%%%%%%%%%%%%%%%%%%
\appendix
%%%%%%%%%%%%%%%%%%%%%%%%%%%%%%%%%%%%%%%%%%%%%%%%%%%%%%%%%%%%%%%%%%%
%%%%%%%%%%%%%%%%%%%%%%%%%%%%%%%%%%%%%%%%%%%%%%%%%%%%%%%%%%%%%%%%%%%
\section{Existence of directed colimits in $\Dott$}\label{sub:dir_colim}
This appendix is devoted to the construction of direct colimits in the category $\Dott$. We demonstrate that this construction preserves many properties, which are crucial for our work in \Cref{sect:dir_colim,sect:dir_colim2}. Specifically, we use these results to verify that two constructions we introduce respect all the needed structures of the starting doctrine.

While some of the results in this section are well-known, such as how directed colimits are computed in categories like $\Cat$ or $\Pos$, we present them here in detail in order to compute how additional structure is preserved.
\begin{proposition}[\Cref{sect:dir_col}, \Cref{prop:dir_colim}]\label{prop:dir_colim_app}
The category $\Dott$ has colimits over directed preorders.
\end{proposition}
\begin{proof}
We begin by considering a directed preorder $I$, so that for each $i,j\in I$ there exists a $k\in I$ such that $k\geq i,j$. Then suppose to have a diagram over this preorder, i.e.\ a functor $D:I\to\Dott$. In particular, for all $i\in I$ we have $P_i\coloneqq D(i):\CC_i\op\to\Pos$, and for all $i\leq k$ a morphism $(F_{ik},\mathfrak{f}_{ik}):P_i\to P_k$ where $F_{ik}:\CC_i\to\CC_k$ is a functor preserving finite products and $\mathfrak{f}_{ik}:P_i\xrightarrow{\cdot} P_kF_{ik}\op$ is a natural transformation. Moreover, we ask for $(F_{ii},\mathfrak{f}_{ii})$ to be the identity on $P_i$, and for $(F_{jk},\mathfrak{f}_{jk})\circ(F_{ij},\mathfrak{f}_{ij})=(F_{ik},\mathfrak{f}_{ik})$ whenever $i\leq j\leq k$.

Our goal is to define a suitable doctrine $P_\bullet:\CC_\bullet\op\to\Pos$, and then show that it is the colimit over $I$.
\pointinproof{The base category $\CC_\bullet$}
The base category $\CC_\bullet$ is the colimit over $I$ in $\Cat$ of the diagram given by $\CC_i$'s and $F_{ij}$'s. Objects are classes of objects from any $\CC_i$, identified as follows:
\begin{equation*}\text{ob}\CC_\bullet=\faktor{\bigsqcup_{i\in I}\CC_i}{\sim},\end{equation*}
where two objects $A_{(i)},B_{(j)}$ in $\CC_i$ and $\CC_j$ respectively are such that $A_{(i)}\sim B_{(j)}$ if and only if there exists $k\geq i,j$ such that $F_{ik}A_{(i)}=F_{jk}B_{(j)}$ in $\CC_k$.
Then for any pair of objects $[A_{(i)}], [B_{(j)}]$ we have as morphisms:
\begin{equation*}\text{Hom}_{\CC_\bullet}\big([A_{(i)}], [B_{(j)}]\big)=\faktor{\bigsqcup_{k\geq i,j}\text{Hom}_{\CC_k}\big(F_{ik}A_{(i)}, F_{jk}B_{(j)}\big)}{\sim}\end{equation*}
where $(f_k:F_{ik}A_{(i)}\to F_{jk}B_{(j)})\sim(f_{k'}:F_{ik'}A_{(i)}\to F_{jk'}B_{(j)})$ if and only if there exists $h\geq k,k'$ such that $F_{kh}f_k=F_{k'h}f_{k'}$ in $\CC_h$.
This is well-defined: suppose $i\leq l$ and $j\leq m$, so that $[A_{(i)}]=[F_{il}A_{(i)}]$ and $[B_{(j)}]=[F_{jm}B_{(j)}]$, we want to show that the inclusion
\begin{equation*}\bigsqcup_{n\geq l,m}\text{Hom}_{\CC_{n}}\big({F_{ln}F_{il}}A_{(i)}, {F_{mn}F_{jm}}B_{(j)}\big)\hookrightarrow \bigsqcup_{k\geq i,j}\text{Hom}_{\CC_k}\big(F_{ik}A_{(i)}, F_{jk}B_{(j)}\big)\end{equation*}
becomes a bijection on the corresponding quotients:
\[\begin{tikzcd}
	\bigsqcup_{n\geq l,m}\text{Hom}_{\CC_{n}}\big(F_{in}A_{(i)}, F_{jn}B_{(j)}\big) & \bigsqcup_{k\geq i,j}\text{Hom}_{\CC_k}\big(F_{ik}A_{(i)}, F_{jk}B_{(j)}\big) \\
	\faktor{\bigsqcup_{n\geq l,m}\text{Hom}_{\CC_{n}}\big(F_{in}A_{(i)}, F_{jn}B_{(j)}\big)}{\sim} & \faktor{\bigsqcup_{k\geq i,j}\text{Hom}_{\CC_k}\big(F_{ik}A_{(i)}, F_{jk}B_{(j)}\big)}{\sim}
	\arrow[hook, from=1-1, to=1-2]
	\arrow[from=1-1, to=2-1]
	\arrow[dotted, from=2-1, to=2-2]
	\arrow[from=1-2, to=2-2]
\end{tikzcd}\]
Take $f_{n_1},f_{n_2}$, with $n_s\geq l,m$, and $f_{n_s}:F_{i{n_s}}A_{(i)}\to F_{j{n_s}}B_{(j)}$ for $s=1,2$. It follows from the definition that $f_{n_1}\sim f_{n_2}$ as arrows seen in the union on the left if and only if $f_{n_1}\sim f_{n_2}$ seen in the union on the right, so that the dotted arrow is both well-defined and injective. This arrow is also surjective: consider $f_{k}:F_{ik}A_{(i)}\to F_{jk}B_{(j)}$ for some $k\geq i,l$ and take $n\geq k,l,m$; then clearly $[f_k]=[F_{kn}f_k]$, with $F_{kn}f_k$ belonging to the union on the left.
To conclude, since the preorder is directed one can show the isomorphism between such quotients of unions also in the general case $i\nleq l$ or $j\nleq m$.

Composition in ${\CC_\bullet}$ between two composable arrows
\begin{equation*}[A_{(i)}]\xrightarrow{[f_k]}[B_{(j)}]\xrightarrow{[f_{k'}]}[C_{(l)}],\end{equation*}
where $f_k:F_{ik}A_{(i)}\to F_{jk}B_{(j)}$ and $f_{k'}:F_{jk'}B_{(j)}\to F_{lk'}C_{(l)}$, is $[f_{k'}]\circ[f_k]=[F_{k'h}f_{k'}\circ F_{kh}f_k]$ for a given $h\geq k,k'$. This is clearly well-defined on the choice of $h$, and on the representative of $f_k$ and $f_{k'}$.
\pointinproof{Finite products in $\CC_\bullet$} The category $\CC_\bullet$ has binary products, defined in the obvious way: take objects $[A_{(i)}],[B_{(j)}]$ and call $[A_{(i)}]\underrightarrow{\times}[B_{(j)}]\coloneqq[F_{ik}A_{(i)}\times F_{jk}B_{(j)}]$, having as projections the classes of projections from $F_{ik}A_{(i)}\times F_{jk}B_{(j)}$ in $\CC_k$ for some $k\geq i,l$---note that $[F_{ik}A_{(i)}]=[A_{(i)}]$ and similarly for the other object, so the codomains of projections make sense in the diagram below. Such class of objects is well-defined because the $F_{\star*}$'s preserve products.
To see that it is indeed a product consider the diagram:
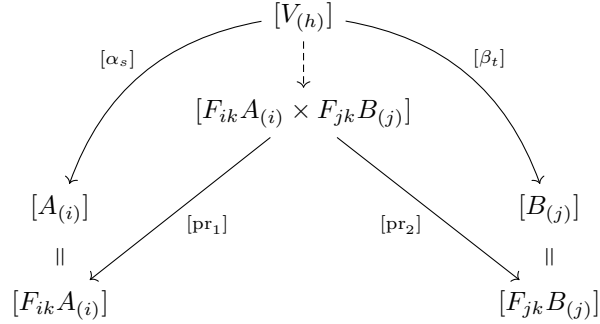
\[\begin{tikzcd}
	& {[V_{(h)}]} \\
	{} & {[F_{ik}A_{(i)}\times F_{jk}B_{(j)}]} \\
	{[A_{(i)}]} && {[B_{(j)}]} \\
	{[F_{ik}A_{(i)}]} && {[F_{jk}B_{(j)}]}
	\arrow["{[\pr1]}", from=2-2, to=4-1]
	\arrow[dashed, from=1-2, to=2-2]
	\arrow["{=}"{marking}, draw=none, from=4-3, to=3-3]
	\arrow["{[\alpha_s]}"', bend right, from=1-2, to=3-1]
	\arrow["{=}"{marking}, draw=none, from=4-1, to=3-1]
	\arrow["{[\beta_t]}", bend left, from=1-2, to=3-3]
	\arrow["{[\pr2]}"', from=2-2, to=4-3]
\end{tikzcd}\]
where $\alpha_s:F_{hs}V_{(h)}\to F_{is}A_{(i)}$, $\beta_t:F_{ht}V_{(h)}\to F_{jt}B_{(j)}$, for some $s\geq h,i$ and $t\geq h,j$. Now let $m\geq i,j,k,h,s,t$ and consider the diagram in $\CC_m$:
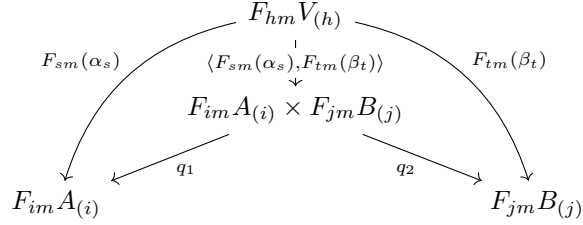
\[\begin{tikzcd}
	& {F_{hm}V_{(h)}} \\
	{} & {F_{im}A_{(i)}\times F_{jm}B_{(j)}} \\
	{F_{im}A_{(i)}} && {F_{jm}B_{(j)}}
	\arrow["{\ple{ F_{sm}(\alpha_s),F_{tm}(\beta_t)}}"description, from=1-2, to=2-2]
	\arrow["{F_{sm}(\alpha_s)}"', bend right, from=1-2, to=3-1]
	\arrow["{F_{tm}(\beta_t)}", bend left, from=1-2, to=3-3]
	\arrow["{q_1}", from=2-2, to=3-1]
	\arrow["{q_2}"', from=2-2, to=3-3]
\end{tikzcd}\]
Clearly $[\ple{ F_{sm}(\alpha_s),F_{tm}(\beta_t)}]$ makes the diagram in $\CC_\bullet$ commute. Now, to prove uniqueness, take $\psi_n=\ple{{\psi_n}_1,{\psi_n}_2}:F_{hn}V_{(h)}\to F_{in}A_{(i)}\times F_{jn}B_{(j)}$ for some $n\geq h,k$, such that $[{\psi_n}_1]=[\alpha_s]$ and $[{\psi_n}_2]=[\beta_t]$. Then, there exists $r\geq n,s,t$ such that $F_{nr}({\psi_n}_1)=F_{sr}(\alpha_s)$ and $F_{nr}({\psi_n}_2)=F_{tr}(\beta_t)$; in particular
\begin{equation*}[\psi_n]=[\ple{ F_{nr}({\psi_n}_1),F_{nr}({\psi_n}_2)}].\end{equation*}
Finally, take $u\geq r,m$:
\begin{align*}F_{mu}(\ple{ F_{sm}(\alpha_s),F_{tm}(\beta_t)})&=\ple{ F_{su}(\alpha_s),F_{tu}(\beta_t)}\\
&=\ple{ F_{ru}F_{sr}(\alpha_s),F_{ru}F_{tr}(\beta_t)}\\
&=F_{ru}\ple{ F_{nr}({\psi_n}_1),F_{nr}({\psi_n}_2)},\end{align*}
i.e.\ $[\ple{ F_{sm}(\alpha_s),F_{tm}(\beta_t)}]=[\psi_n]$.

In order to conclude the argument about the existence of finite products, observe that if $\tmn_i$ is a terminal object in $\CC_i$, then $[\tmn_i]$ is a terminal object in $\CC_\bullet$: take an object $[B_{(j)}]$, $k\geq i,j$ and consider the unique map $!_{F_{jk}B_{(j)}}:F_{jk}B_{(j)}\to \tmn_k$ in $\CC_k$. Then $[!_{F_{jk}B_{(j)}}]$ is a map from $[B_{(j)}]$ to $[\tmn_i]=[F_{ik}\tmn_i]=[\tmn_k]$. We show uniqueness by considering a map $[u_h]:[B_{(j)}]\to[\tmn_i]$ for some $u_h:F_{jh}B_{(j)}\to F_{ih}\tmn_i$: then $u_h=\,!_{F_{jh}B_{(j)}}$ in $\CC_h$. Taking $l\geq k,h$ we get $[!_{F_{jk}B_{(j)}}]=[F_{kl}(!_{F_{jk}B_{(j)}})]=[!_{F_{jl}B_{(j)}}]=[F_{hl}(!_{F_{jh}B_{(j)}})]=[!_{F_{jh}B_{(j)}}]=[u_h]$.
\pointinproof{$P_\bullet$ on objects} Now that we built a suitable base category with finite products, we define the doctrine $P_\bullet$.
For an object $[A_{(i)}]$, we take:
\begin{equation*}P_\bullet([A_{(i)}])=\faktor{\bigsqcup_{k\geq i}P_k(F_{ik}A_{(i)})}{\sim}\end{equation*}
where $a_{k_1}\sim a_{k_2}$, with $a_{k_s}\in P_{k_s}(F_{ik_s}A_{(i)})$ for $s=1,2$, if and only if there exists $j\geq k_1,k_2$ such that
\begin{equation*}\big(\mathfrak{f}_{k_1j}\big)_{F_{ik_1}A_{(i)}}(a_{k_1})=\big(\mathfrak{f}_{k_2j}\big)_{F_{ik_2}A_{(i)}}(a_{k_2})\text{ in }P_j(F_{ij}A_{(i)}).\end{equation*}
This is well-defined on the choice of the representative of $[A_{(i)}]$: in a similar way to what we did above defining arrows in $\CC_\bullet$, we prove that the dotted arrow induced by the inclusion is bijective, in the case $l\geq i$.
\[\begin{tikzcd}
	\bigsqcup_{k\geq l}P_k(F_{ik}A_{(i)}) & \bigsqcup_{n\geq i}P_n(F_{in}A_{(i)}) \\
	\faktor{\bigsqcup_{k\geq l}P_k(F_{ik}A_{(i)})}{\sim} & \faktor{\bigsqcup_{n\geq i}P_n(F_{in}A_{(i)})}{\sim}
	\arrow[hook, from=1-1, to=1-2]
	\arrow[from=1-1, to=2-1]
	\arrow[dotted, from=2-1, to=2-2]
	\arrow[from=1-2, to=2-2]
\end{tikzcd}\]
Take $a_{h_1}, a_{h_2}$ for $h_1, h_2\geq l$, then $a_{h_1}\sim a_{h_2}$ on the left if and only if they are equivalent on the right, hence well-definition and injectivity of the function follows. Surjectivity also follows easily: take $[b_m]$ for some $m\geq i$, and let $u\geq m,l$. Then $b_m\sim\big(\mathfrak{f}_{mu}\big)_{F_{im}A_{(i)}}(b_m)\in P_u(F_{iu}A_{(i)})$ as wanted.
If we fix $A_{(i)}$, we observe that $P_\bullet([A_{(i)}])$ is a directed colimit in $\Pos$ on the diagram defined over elements of $I$ greater or equal to $i$. An element $j\geq i$ is sent to $P_j(F_{ij}A_{(i)})$, and for any $j\leq k$ we have the monotone function $\big(\mathfrak{f}_{jk}\big)_{F_{ij}A_{(i)}}:P_j(F_{ij}A_{(i)})\to P_k(F_{ik}A_{(i)})$. Hence we defined a poset for each object of $\CC_\bullet$.
\pointinproof{$P_\bullet$ on arrows} Take a $\CC_\bullet$-arrow $[f]:[A_{(i)}]\to[B_{(j)}]$ for some $f:F_{ik}A_{(i)}\to F_{jk}B_{(j)}\in\text{arr}\CC_k, k\geq i,j$.
\[\begin{tikzcd}[row sep=5pt]
	{\CC_\bullet\op} & \Pos \\
	{[B_{(j)}]} & {P_\bullet([B_{(j)}])} & {P_\bullet([F_{jk}B_{(j)}])} \\
	\\
	\\
	\\
	{[A_{(i)}]} & {P_\bullet([A_{(i)}])} & {P_\bullet([F_{ik}A_{(i)}])}
	\arrow[""{name=0, anchor=center, inner sep=0}, "{P_\bullet([f])}", from=2-2, to=6-2]
	\arrow["{=}"{description}, draw=none, from=2-2, to=2-3]
	\arrow["{=}"{description}, draw=none, from=6-2, to=6-3]
	\arrow[""{name=1, anchor=center, inner sep=0}, "{[f]}", from=6-1, to=2-1]
	\arrow[from=1-1, to=1-2]
\end{tikzcd}\]
For any given $l\geq k$ we have $F_{kl}(f):F_{il}A_{(i)}\to F_{jl}B_{(j)}\in\text{arr}\CC_l$ and
\begin{equation*}P_l(F_{kl}(f)):P_l(F_{jl}B_{(j)})\to P_l(F_{il}A_{(i)}).\end{equation*}
Since $P_\bullet([F_{jk}B_{(j)}])=\faktor{\bigsqcup_{l\geq k}P_l(F_{jl}B_{(j)})}{\sim}$, we prove that the map
\begin{equation*}\bigsqcup_{l\geq k}P_l(F_{jl}B_{(j)})\longrightarrow\faktor{\bigsqcup_{m\geq k}P_m(F_{im}A_{(i)})}{\sim}\end{equation*}
sending any $\beta_l$ in $[P_l(F_{kl}(f))\beta_l]$ is well-defined on the quotient, hence defining a map from $P_\bullet([B_{(j)}])$ to $P_\bullet([A_{(i)}])$. Take $l'\geq l$---then, the statement for any $h\geq k$ follows---, so that $\beta_l\sim\big(\mathfrak{f}_{ll'}\big)_{F_{jl}B_{(j)}}\beta_l\in P_{l'}(F_{jl'}B_{(j)})$ and
\begin{equation*}\big(\mathfrak{f}_{ll'}\big)_{F_{jl}B_{(j)}}\beta_l\mapsto [P_{l'}(F_{kl'}(f))\big(\mathfrak{f}_{ll'}\big)_{F_{jl}B_{(j)}}\beta_l].\end{equation*}
We now use the naturality of $\mathfrak{f}_{ll'}$ and get:
\begin{equation*}[P_{l'}(F_{kl'}(f))\big(\mathfrak{f}_{ll'}\big)_{F_{jl}B_{(j)}}\beta_l]=[\big(\mathfrak{f}_{ll'}\big)_{F_{il}A_{(i)}}P_l(F_{kl}(f))\beta_l]=[P_l(F_{kl}(f))\beta_l]\end{equation*}
as claimed.

The following step is to prove that the definition of $P_\bullet([f])$ does not depend on the representative of $[f]$. Take $k'\geq k$, then $[f]=[F_{kk'}(f)]$, with $F_{kk'}(f):F_{ik'}A_{(i)}\to F_{jk'}B_{(j)}$. Hence we have for any $\beta_{l'}\in P_{l'}(F_{jl'}B_{(j)})$, $l'\geq k'$
\begin{equation*}[\beta_{l'}]\mapsto [P_{l'}({F_{k'l'}F_{kk'}}(f))\beta_{l'}]\end{equation*}
but ${F_{k'l'}F_{kk'}}=F_{kl'}$, the two maps act in the same way from $P_\bullet([B_{(j)}])$ to $P_\bullet([A_{(i)}])$.

It follows from the fact that $P_\bullet([f])$ is defined on any suitable $k'\geq k$ and that both $[\blank]$---in any $P_\bullet([C_{h}])$---and $P_{k'}(F_{kk'}(f))$ preserve the order, that $P_\bullet([f])$ preserves the order; moreover, also functoriality comes easily. Hence $P_\bullet:\CC_\bullet\op\to\Pos$ is indeed a doctrine.
\pointinproof{A universal cocone into $P_\bullet$} Now, for any $i\in I$, define the 1-arrow $(F_i,\mathfrak{f}_i):P_i\to P_\bullet$ in $\Dott$ as follows:
\begin{center}
\begin{tikzcd}
\CC_i^{\text{op}}\arrow[rr,"{F_i}^{\text{op}}"] \arrow[dr,"P_i"' ,""{name=L}]&&{\CC_\bullet}^{\text{op}}\arrow[dl,"P_\bullet" ,""'{name=R}]\\
&\Pos\arrow[rightarrow,"\mathfrak{f}_i","\cdot"', from=L, to=R, bend left=10]
\end{tikzcd}.
\end{center}
The functor $F_i$ is the quotient map, sending $f:A_{(i)}\to B_{(i)}$ to $[f]:[A_{(i)}]\to[B_{(i)}]$; observe that by construction such functors preserve finite products. Similarly $\mathfrak{f}_i:P_i\xrightarrow{\cdot}P_\bullet F_i\op$ is the quotient map on every object of $\CC_i$:
\begin{equation*}\big(\mathfrak{f}_i\big)_{A_{(i)}}:P_i(A_{(i)})\to P_\bullet([A_{(i)}])\text{ is defined by the assignment }\alpha_i\mapsto[\alpha_i].\end{equation*}
Such functions are clearly order-preserving. It follows trivially from the definition of $P_\bullet$ on arrows that $\mathfrak{f}_i$ is a natural transformation.
Now, to check that it is indeed a cocone, take $i\leq k$: we want $(F_k,\mathfrak{f}_k)\circ(F_{ik},\mathfrak{f}_{ik})=(F_i,\mathfrak{f}_i)$.
\begin{center}
\begin{tikzcd}
\CC_i\op\arrow[r,"{F_{ik}}\op"]\arrow[dr,"P_i"',""{name=L},bend right]&{\CC_k}\op\arrow[r,"{F_k}\op"]\arrow[d,""'{name=C},"P_k"{name=M}]&\CC_\bullet\op\arrow[dl, "P_\bullet",""'{name=R}, bend left]\\
&\Pos
\arrow[rightarrow,"\mathfrak{f}_{ik}"near end,"\cdot"', from=L, to=C, bend left=10]
\arrow[rightarrow,"\mathfrak{f}_k"near start,"\cdot"', from=M, to=R, bend left=10]
\end{tikzcd}
\end{center}
Concerning the functors between the base categories, observe that the composition
\[\begin{tikzcd}[column sep=15pt]
	{A_{(i)}} & {F_{ik}A_{(i)}} & {[F_{ik}A_{(i)}]} & {[A_{(i)}]} \\
	{B_{(i)}} & {F_{ik}B_{(i)}} & {[F_{ik}B_{(i)}]} & {[B_{(i)}]}
	\arrow[""{name=0, anchor=center, inner sep=0}, "f", from=1-1, to=2-1]
	\arrow[""{name=1, anchor=center, inner sep=0}, "{[F_{ik}(f)]}", from=1-3, to=2-3]
	\arrow["{=}"{description}, draw=none, from=1-3, to=1-4]
	\arrow["{=}"{description}, draw=none, from=2-3, to=2-4]
	\arrow[""{name=2, anchor=center, inner sep=0}, "{F_{ik}(f)}", from=1-2, to=2-2]
	\arrow[""{name=3, anchor=center, inner sep=0}, "{[f]}"', from=1-4, to=2-4]
	\arrow[shorten <=25pt, shorten >=8pt, maps to, from=2, to=1]
	\arrow["{=}"{description, pos=0.7}, Rightarrow, draw=none, from=1, to=3]
	\arrow[shorten <=11pt, shorten >=11pt, maps to, from=0, to=2]
\end{tikzcd}\]
is indeed $F_i$.
Then, for any $\alpha_i\in P_i(A_{(i)})$, we have:
\begin{equation*}\big(\mathfrak{f}_k\circ\mathfrak{f}_{ik}\big)_{A_{(i)}}\alpha_i=\big(\mathfrak{f}_k\big)_{F_{ik}A_{(i)}}\big(\mathfrak{f}_{ik}\big)_{A_{(i)}}\alpha_i=[\big(\mathfrak{f}_{ik}\big)_{A_{(i)}}\alpha_i]=[\alpha_i]=\big(\mathfrak{f}_i\big)_{A_{(i)}}\alpha_i,\end{equation*}
so that $\mathfrak{f}_k\circ\mathfrak{f}_{ik}=\mathfrak{f}_i$.

Suppose we have another cocone, i.e.\ any doctrine $R:\ct{D}\op\to\Pos$ that comes with a family of 1-arrows $\{(G_i,\mathfrak{g}_i):P_i\to R\}_{i\in I}$ such that for any $i\leq k$ one has $(G_k,\mathfrak{g}_k)\circ(F_{ik},\mathfrak{f}_{ik})=(G_i,\mathfrak{g}_i)$ we look for a unique 1-arrow $(G,\mathfrak{g}):P_\bullet\to R$ such that $(G,\mathfrak{g})\circ(F_i,\mathfrak{f}_i)=(G_i,\mathfrak{g}_i)$ for all $i\in I$.

In order to define $G:\CC_\bullet\to\ct{D}$, take any $[f]:[A_{(i)}]\to[B_{(j)}]$ with $f:F_{ik}A_{(i)}\to F_{jk}B_{(j)}$ for some $k\geq i,j$ and send it to $G_k(f):G_iA_{(i)}\to G_jB_{(j)}$. This is well-defined because of the commutativity properties of the cocone. Similarly we define $\mathfrak{g}:P_\bullet\xrightarrow{\cdot} RG\op$: for a given object $[A_{(i)}]$, we take
\begin{equation*}\mathfrak{g}_{[A_{(i)}]}:P_\bullet([A_{(i)}])\to RG_iA_{(i)},\text{ such that }[\alpha_k]\mapsto\big(\mathfrak{g}_k\big)_{F_{ik}A_{(i)}}\alpha_k\end{equation*}
for any $\alpha_k\in P_k(F_{ik}A_{(i)}), k\geq i$. This is well-defined on both $[\alpha_k]$ and $[A_{(i)}]$ again from the properties of the cocone.
Naturality of $\mathfrak{g}$ is also easy to see: given an arrow $[f]:[A_{(i)}]\to[B_{(j)}]$ we compute both $RG([f])\mathfrak{g}_{[B_{(j)}]}$ and $\mathfrak{g}_{[A_{(i)}]}P_\bullet([f])$ on a given $[\beta_l]\in P_\bullet([B_{(j)}])$:
\begin{align*}RG([f])\mathfrak{g}_{[B_{(j)}]}[\beta_l]&=RG_k(f)\big(\mathfrak{g}_l\big)_{F_{jl}B_{(j)}}\beta_l=RG_l(F_{kl}(f))\big(\mathfrak{g}_l\big)_{F_{jl}B_{(j)}}\beta_l\\
&=\big(\mathfrak{g}_l\big)_{F_{il}A_{(i)}}P_l(F_{kl}(f))\beta_l=\mathfrak{g}_{[A_{(i)}]}[P_l(F_{kl}(f))\beta_l]=\mathfrak{g}_{[A_{(i)}]}P_\bullet([f])[\beta_l].\end{align*}

Uniqueness is given by the fact that all triangles like the one below must commute.
\[\begin{tikzcd}
	{P_i} && R \\
	& {P_\bullet}
	\arrow["{(F_i,\mathfrak{f}_i)}"', from=1-1, to=2-2]
	\arrow["{(G,\mathfrak{g})}"', from=2-2, to=1-3]
	\arrow["{(G_i,\mathfrak{g}_i)}", from=1-1, to=1-3]
\end{tikzcd}\]\end{proof}

%%%%%%%%%%%%%%%%%%%%%%%%%%%%%%%%%%%%%%%%%%%%%%%%%%%%%%%%%%%%%%%%%%%
\subsection{Additional structure} We now show that many properties are preserved by a directed colimit.
\begin{proposition}[\Cref{sect:dir_col}, \Cref{prop:addit_struct}]
Let $I$ be a directed preorder, let $D:I\to\Dott$ be a diagram, $D(i\leq j)=[(F_{ij},\mathfrak{f}_{ij}):P_i\to P_j]$ for any $i,j\in I$, and let $(P_\bullet,\{(F_i,\mathfrak f _i)\}_{i\in I})$ be the colimit of $D$. Suppose that for every $i,j\in I$, the doctrine $P_i$ and the morphism $(F_{ij},\mathfrak{f}_{ij})$ are primary. Then the doctrine $P_\bullet$ is a primary doctrine, and for every $i\in I$ the morphism $(F_i,\mathfrak f _i)$ is primary.
Moreover, if in a cocone $(R,\{(G_i,\mathfrak g _i)\}_{i\in I})$, $R$ and $(G_i,\mathfrak g _i)$ are primary, then the unique arrow $(G,\mathfrak g):P_\bullet\to R$ defined by the universal property of the colimit is primary.
The same statement holds if we write respectively bounded, with binary joins, implicational, elementary, existential, universal, Heyting, Boolean instead of primary.
\end{proposition}
\begin{proof}{\bf Algebraic properties:}
It is a well-known fact that directed colimit of algebraic structures exists, hence if for all $i\in I$, $P_i$ is endowed with equational structure such as $\land,\top$ or $\lor,\bot$, then these operations are defined also in $P_\bullet$, preserved by $\mathfrak{f}_i$ for all $i\in I$. Such properties are also preserved by reindexing: this can be shown using naturality of $\mathfrak{f}_{ij}$ and the fact that they are preserved by reindexing in each $P_i$ . Moreover, since $\mathfrak{g}$ is defined through $\mathfrak{g}_i$'s, which preserve operations, also $\mathfrak{g}$ preserves them. 
\pointinproof{Implication}
We define for each pair of elements $[\alpha_k],[\beta_{k'}]\in P_\bullet{[A_{(i)}]}$, with $\alpha_k\in F_{ik}A_{(i)}$ and $\beta_{k'}\in F_{ik'}A_{(i)}$ for some $k,k'\geq i$
\begin{equation*}[\alpha_k]\to[\beta_{k'}]\coloneqq[\big(\mathfrak{f}_{kh}\big)_{F_{ik}A_{(i)}}\alpha_k\to\big(\mathfrak{f}_{k'h}\big)_{F_{ik'}A_{(i)}}\beta_{k'}]\end{equation*}
for some $h\geq k,k'$. This is well-defined because every function in $\{\mathfrak{f}_{ij}\}_{i,j\in I}$ preserves implications. Moreover, this is indeed a right adjoint to the binary meet operation:
\begin{equation}\label{eq:impl}[\gamma_{\overline{k}}]\leq[\alpha_k]\to[\beta_{k'}]\text{ in }P_\bullet{[A_{(i)}]}\end{equation}
if and only if there exists $s\geq\overline{k},k,k'$ such that in $P_s(F_{is}A_{(i)})$
\begin{equation*}\big(\mathfrak{f}_{\overline{k}s}\big)_{F_{i\overline{k}}A_{(i)}}\gamma_{\overline{k}}\leq\big(\mathfrak{f}_{ks}\big)_{F_{ik}A_{(i)}}\alpha_k\to\big(\mathfrak{f}_{k's}\big)_{F_{ik'}A_{(i)}}\beta_{k'},\end{equation*}
but this inequality holds if and only if
\begin{equation*}\big(\mathfrak{f}_{\overline{k}s}\big)_{F_{i\overline{k}}A_{(i)}}\gamma_{\overline{k}}\land\big(\mathfrak{f}_{ks}\big)_{F_{ik}A_{(i)}}\alpha_k\leq\big(\mathfrak{f}_{k's}\big)_{F_{ik'}A_{(i)}}\beta_{k'}\end{equation*}
so $\eqref{eq:impl}$ holds if and only if
\begin{equation*}[\gamma_{\overline{k}}]\land[\alpha_k]\leq[\beta_{k'}].\end{equation*}
Now, since $[\alpha_k]\to[\beta_{k'}]$ is computed in a common poset, as in the case of algebraic properties, the implication is preserved by reindexings, $\{\mathfrak{f}_i\}_{i\in I}$ and $\mathfrak{g}$. 
\pointinproof{Elementarity}
For a given object $[A_{(i)}]$, take $[\delta_{A_{(i)}}]\in P_{\bullet}([A_{(i)}\times A_{(i)}])$; we prove that it is the fibered equality on $[A_{(i)}]$.
\begin{enumerate}
\item Since in $P_i(A_{(i)})$ we have $\top_{A_{(i)}}\leq P_i(\Delta_{A_{(i)}})(\delta_{A_{(i)}})$, we have in $P_{\bullet}([A_{(i)}])$ that $\top_{[A_{(i)}]}=[\top_{A_{(i)}}]\leq [P_i(\Delta_{A_{(i)}})(\delta_{A_{(i)}})]=P_\bullet(\Delta_{[A_{(i)}]})([\delta_{A_{(i)}}])$.
\item For any $\alpha_k\in P_k(F_{ik}A_{(i)})$ with $k\geq i$, we want to show $P_\bullet(\pr1)([\alpha_k])\land [\delta_{A_{(i)}}]\leq P_\bullet(\pr2)([\alpha_k])$. Compute $P_\bullet(\pr1)([\alpha_k])\land [\delta_{A_{(i)}}]=[P_k(\pr1)(\alpha_k)\land(\mathfrak{f}_{ik})_{A_{(i)}\times A_{(i)}}(\delta_{A_{(i)}})]=[P_k(\pr1)(\alpha_k)\land\delta_{F_{ik}A_{(i)}}]\leq[P_k(\pr2)(\alpha_k)]=P_\bullet(\pr2)([\alpha_k])$.
\item For any pair of object $A_{(i)}$ in $\CC_i$ and $B_{(j)}$ in $\CC_j$, compute $[A_{(i)}]\underrightarrow{\times}[B_{(j)}]$ as $[F_{ik}A_{(i)}\times F_{jk}B_{(j)}]$ for some $k\geq i,j$. We want to prove that $P_\bullet(\ple{\pr1,\pr3})([\delta_{A_{(i)}}])\land P_\bullet(\ple{\pr2,\pr4})([\delta_{B_{(j)}}])\leq [\delta_{F_{ik}A_{(i)}\times F_{jk}B_{(j)}}]$. However,
\begin{align*}&P_\bullet(\ple{\pr1,\pr3})([\delta_{A_{(i)}}])\land P_\bullet(\ple{\pr2,\pr4})([\delta_{B_{(j)}}])\\
&=[P_k(\ple{\pr1,\pr3})(\mathfrak{f}_{ik})_{A_{(i)}\times A_{(i)}}(\delta_{A_{(i)}})]\land [P_k(\ple{\pr2,\pr4})(\mathfrak{f}_{jk})_{B_{(j)}\times B_{(j)}}(\delta_{B_{(j)}})]\\
&=[P_k(\ple{\pr1,\pr3})(\delta_{F_{ik}A_{(i)}})\land P_k(\ple{\pr2,\pr4})(\delta_{F_{jk}B_{(j)}})]\\
&\leq[\delta_{F_{ik}A_{(i)}\times F_{jk}B_{(j)}}],\end{align*}
as claimed.
\end{enumerate}
Moreover, by definition $\{\mathfrak{f}_i\}_{i\in I}$ and $\mathfrak{g}$ preserve the structure.
\pointinproof{Existentiality and Universality}
Take $[C_{(i)}],[B_{(j)}]$, consider $[\pr1]:[C_{(i)}]\underrightarrow{\times}[B_{(j)}]\to[C_{(i)}]$, where we call $\pr1:F_{ik}C_{(i)}\times F_{jk}B_{(j)}\to F_{ik}C_{(i)}$ the projection in $\CC_k$ for any $k\geq i,j$. Then consider
\begin{equation*}P_\bullet([\pr1]):P_\bullet([C_{(i)}])\to P_\bullet([C_{(i)}]\underrightarrow{\times}[B_{(j)}])\end{equation*}
and define
\begin{equation*}{\exists_\bullet}_{[C_{(i)}]}^{[B_{(j)}]}[\beta_l]\coloneqq[{\exists_l}^{F_{jl}B_{(j)}}_{F_{il}C_{(i)}}\beta_l]\quad\text{ and }\quad{\forall_\bullet}_{[C_{(i)}]}^{[B_{(j)}]}[\beta_l]\coloneqq[{\forall_l}^{F_{jl}B_{(j)}}_{F_{il}C_{(i)}}\beta_l]\end{equation*}
for $\beta_l\in P_l(F_{il}C_{(i)}\times F_{jl}B_{(j)})$, where $\exists_l$ and $\forall_l$ are respectively the existential and universal quantifier for the doctrine $P_l$.

This is well-defined since every map in $\{\mathfrak{f}_{ij}\}_{i,j\in I}$ preserves the structure.
Moreover, one can prove that ${\exists_\bullet}_{[C_{(i)}]}^{[B_{(j)}]}$ and ${\forall_\bullet}_{[C_{(i)}]}^{[B_{(j)}]}$ define respectively the left adjoint and the right adjoint to $P_\bullet([\pr1])$, that they both satisfy Beck-Chevalley condition, and Frobenius reciprocity for the existential quantifier holds---these follow from the correspondent properties of all $\exists_k$ and $\forall_k$. Furthermore, $\{\mathfrak{f}_i\}_{i\in I}$ and $\mathfrak{g}$ preserve the structures.
\end{proof}
%%%%%%%%%%%%%%%%%%%%%%%%%%%%%%%%%%%%%%%%%%%%%%%%%%%%%%%%%%%%%%%%%%%%%%%%%%%%%%%%%%%%%%%%%%%%%%%%%%%%%%%%%%%%%%%%%%%%%%%%%%%%%%%%%%%%%%%%%%%%%%%%%%%%%%%%%%%%%%%%%%%%%%%%%%%%%%%%%%%%%%%%%%%%%%%%%%%%%%%%%%%%%%%%%%%%%%%%%%%%%%%%%%%%%%%%%%%%%%%%%%%%%%%%%%%%%%%%%%%%%%%%%%%%%%%%%%%%%%%%%%%%%%%%%%%%%%%%%%%%%%%%%%%%%%%%%%%%%%%%%%%%%%%%%%%%%%%%%%%%%%%%%%%%%%%%%%%%%%%%%%

\section*{Acknowledgement}
The author would like to acknowledge her Ph.D. supervisors Sandra Mantovani and Pino Rosolini: this paper is the third chapter of her Ph.D. thesis, defended at Università degli Studi di Milano. The author would also like to thank Jacopo Emmenegger for his valuable comments.

\bibliography{Biblio}
\bibliographystyle{alpha}

\end{document}